\documentclass[journal]{IEEEtran}

\usepackage{graphicx}
\usepackage{amsmath}
\usepackage{amssymb}
\usepackage{amsthm}
\usepackage{mathtools}
\usepackage{bm}
\usepackage[dvipsnames]{xcolor}
\usepackage{booktabs}
\usepackage{cancel}
\usepackage{soul}
\usepackage{cite}

\usepackage{tikz}

\usepackage{epstopdf}

\definecolor{green}{RGB}{0,128,0}
\definecolor{blue}{RGB}{0,70,255}

\usepackage[hidelinks,bookmarks=false]{hyperref}

\theoremstyle{plain}
\newtheorem{theorem}{Theorem}[section]

\newtheorem{lemma}[theorem]{Lemma}
\newtheorem{corollary}[theorem]{Corollary}
\theoremstyle{definition}

\theoremstyle{remark}

\DeclareMathAlphabet\mathbfcal{OMS}{cmsy}{b}{n}

\newcommand{\E}{\mathbb{E}}

\newcommand{\R}{\mathbb{R}}

\newcommand{\FF}{\mathrm{F}}
\newcommand{\FFup}{\mathrm{F}^\uparrow}

\newcommand{\Ot}{\mathcal{T}}

\newcommand{\Erf}{\operatorname{Erf}}
\newcommand{\Erfc}{\operatorname{Erfc}}
\newcommand{\Iup}{I^\uparrow}
\newcommand{\Nu}{N_u(T)}
\newcommand{\Nup}{N_u^{\uparrow}(T)}
\newcommand{\Ndown}{N_u^{\downarrow}(T)}
\newcommand{\btheta}{\boldsymbol{\theta}}

\newcommand{\dd}{\mathrm{d}}

\newcommand{\Var}{\text{Var}}
\newcommand{\LHS}{\text{LHS}}
\newcommand{\RHS}{\text{RHS}}
\newcommand{\temp}{\vartheta}

\newcommand{\rp}{r^\prime}
\newcommand{\rpp}{r^{\prime\prime}}

\begin{document}

\title{Exact Variance and Fano Factor for Arbitrary Level Crossings in Stationary Gaussian Processes}

\author{Shivang~Rawat\IEEEauthorrefmark{1}\IEEEauthorrefmark{2},
        Flaviano~Morone\IEEEauthorrefmark{2}\IEEEauthorrefmark{3},
        David~J.~Heeger\IEEEauthorrefmark{3}\IEEEauthorrefmark{4},
        and~Stefano~Martiniani\IEEEauthorrefmark{1}\IEEEauthorrefmark{2}\IEEEauthorrefmark{3}\IEEEauthorrefmark{5}\IEEEcompsocitemizethanks{\IEEEcompsocthanksitem\IEEEauthorrefmark{1}Courant Institute of Mathematical Sciences, New York University.
\IEEEcompsocthanksitem\IEEEauthorrefmark{2}Center for Soft Matter Research, Department of Physics, New York University.
\IEEEcompsocthanksitem\IEEEauthorrefmark{3}Center for Neural Science, New York University.
\IEEEcompsocthanksitem\IEEEauthorrefmark{4}Department of Psychology, New York University.
\IEEEcompsocthanksitem\IEEEauthorrefmark{5}Simons Center for Computational Physical Chemistry, Department of Chemistry, New York University.}\thanks{Corresponding author: Shivang Rawat (e-mail: shivang1797@gmail.com).}}

\markboth{}{Rawat \MakeLowercase{\textit{et al.}}: Exact Variance and Fano Factor for Level Crossings}

\maketitle

\begin{abstract}
Understanding the statistics of level crossings 
in stochastic processes is crucial across many 
scientific disciplines. The traditional Kac-Rice formula gives the mean rate of level crossings and has found broad use. However, that mean rate captures only a coarse summary of the crossing process. It depends entirely on local properties of the stochastic process at a given instant and is therefore blind to the correlation structure of 
the process over time. To understand whether crossing events, such as neuronal spikes, tend to cluster in time, spread apart, or exhibit more complex temporal organization, one must go beyond the mean rate and study higher-order crossing statistics. 
Here we go beyond the mean by deriving the exact 
analytical formulae for the variance and Fano factor 
of arbitrary level crossings in smooth stationary 
Gaussian processes. 
Our exact solution reveals how the full temporal 
correlation structure dictates whether crossings 
cluster or become regular. In systems with 
oscillatory correlations, such as a stochastic 
damped harmonic oscillator, a recent crossing 
suppresses an immediate subsequent one, producing 
sub-Poissonian statistics. However, as damping 
increases and oscillations disappear, a large 
and slow excursion above the threshold can produce 
multiple closely spaced crossings, yielding 
super-Poissonian statistics. 
In purely relaxational, non-oscillatory systems, 
such as a mean-reverting process driven by 
Ornstein-Uhlenbeck noise, the competition between 
the timescales of the driving noise and system 
relaxation produces a richer landscape, including 
reentrant transitions between sub- and super-Poissonian 
statistics as the threshold level is varied. 
Taken together, the exact variance and Fano factor 
derived here complement the Kac-Rice mean rate, enabling 
more robust parameter estimation and model selection 
across any setting where Gaussian processes are used. 

\end{abstract}


The study of level crossings in stochastic 
processes is of fundamental importance across 
a range of scientific and engineering disciplines. 
In neuroscience, the rate and variability of 
threshold crossings by membrane potential fluctuations 
determine neuronal firing patterns, with the Fano 
factor serving as a standard diagnostic for 
distinguishing regular from bursty spiking \cite{verechtchaguina2006first, tchumatchenko2010correlations, badel2011firing, gowers2023upcrossing}. 
In structural and reliability engineering, 
the number of times a load exceeds a critical 
threshold directly governs fatigue life estimates~\cite{lutes2004random}, where knowledge 
of the mean crossing rate alone can underestimate 
failure risk if crossings are clustered. In finance, 
the variability and clustering of extreme-event 
crossings carry implications for risk management 
and policy \cite{jafari2006level, shayeganfar2012level}. 

Across all these domains, the mean crossing 
rate alone, as quantified by the celebrated 
Kac-Rice formula~\cite{kac1943, rice1944mathematical}, 
is insufficient: it is the higher-order statistics, 
and in particular the variance, that capture 
the degree of clustering or regularity of 
crossing events. 
Yet computing the variance of crossing counts 
has remained an open problem for crossings 
at arbitrary threshold levels, due to the 
asymmetric Gaussian integrals that arise when 
the threshold is nonzero. 
In this work we close this gap by deriving exact 
analytical expressions for the variance and Fano 
factor of level crossings in smooth stationary 
Gaussian processes. 

The Fano factor is the ratio of the variance 
to the mean of the crossing count and serves 
as a natural diagnostic to classify the 
statistics of level crossings. A Fano factor 
of 1 corresponds to Poisson-distributed crossings, 
while values below or above unity signal sub-Poissonian regularity or super-Poissonian clustering, respectively. 
Our results reveal that this quantity is entirely 
governed by the temporal correlation structure of 
the process. Specifically, oscillatory correlations 
typically enforce regularity by suppressing immediate 
subsequent crossings, while overdamped dynamics allow 
long excursions above the threshold to generate bursts 
of closely spaced crossings. In purely relaxational 
systems, the competition between the timescales of 
the driving noise and the system's relaxation produces 
a richer landscape, including reentrant transitions 
between sub-Poissonian and super-Poissonian regimes as the 
threshold is varied.

As illustrated in Fig.~\ref{fig:description_upcrossing}, an \textit{up}-crossing of a given level $u$ occurs when the process $X_t$ intersects the threshold $u$ with a positive slope (i.e., $X_t = u$ and $\dot{X}_t > 0$). We denote the number of such events within a time interval $[0, T]$ as $\Nup$. Similarly, \textit{down}-crossings, where $X_t=u$ and $\dot{X}_t < 0$, are denoted by $\Ndown$. The total number of crossings, encompassing both directions, is $\Nu = \Nup + \Ndown$. Our primary objective is to derive exact analytical expressions for the variance of these counting processes, specifically $\Var[\Nup]$ and $\Var[\Nu]$, and for the Fano factor, a dimensionless metric that quantifies the regularity of crossing events relative to a Poisson process.

\begin{figure*}[!t]
    \centering
    \includegraphics[width=0.9\textwidth]{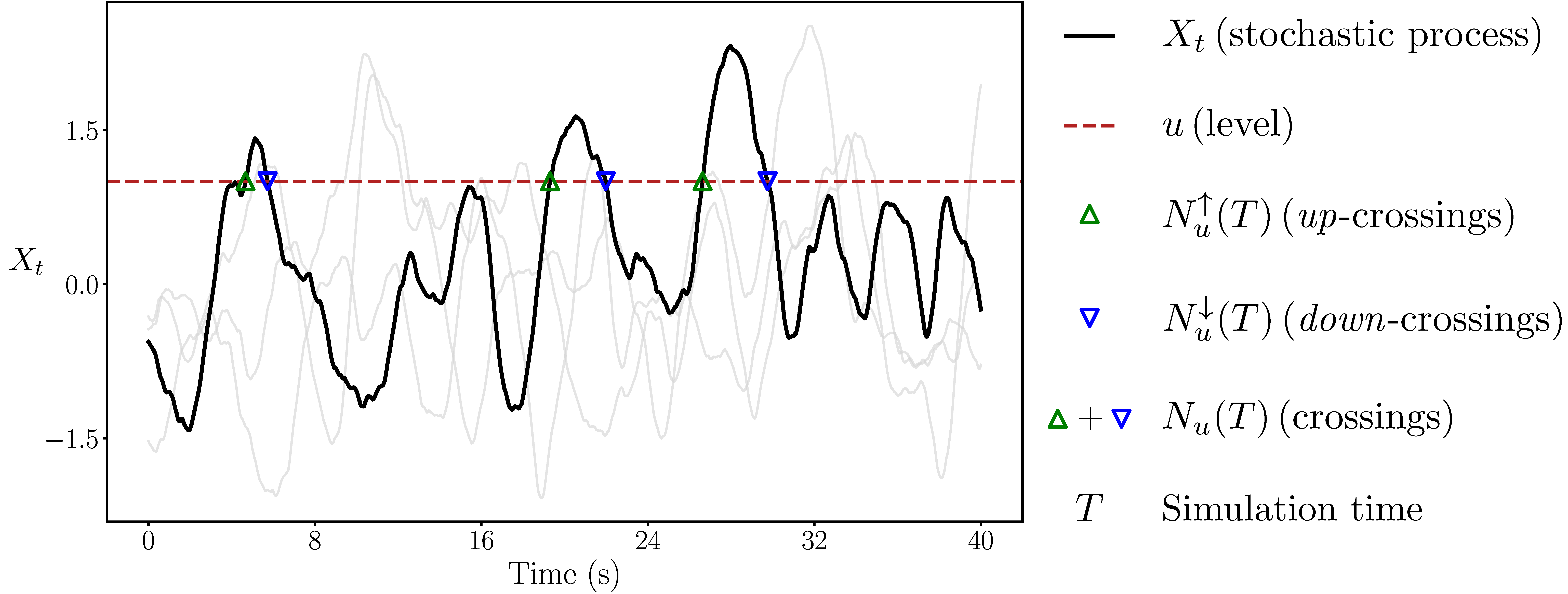}
    \caption[Illustration of level crossing events for a stationary Gaussian process.]{\textbf{Illustration of level crossing events for a stationary Gaussian process.} We plot the sample paths of a stationary Gaussian stochastic process $X_t$ generated by the stochastic damped harmonic oscillator model, which is detailed in Section~\ref{sec:sdho}. The model parameters are set to $\omega_0=1.0$, $\temp=1.0$, and $\zeta=0.5$, yielding a process variance $r(0)=\vartheta/\omega_0^2=1.0$. One prominent realization of $X_t$ is shown as a black curve, with other realizations depicted in gray, all over the time interval $[0, T]$. The horizontal dashed red line indicates the constant threshold level $u=1.0$. Upcrossings, defined as instances where $X_t=u$ and its time derivative $\dot{X}_t > 0$, are denoted by $\Nup$ and marked with green upward triangles. Downcrossings, where $X_t=u$ and $\dot{X}_t < 0$, are denoted by $\Ndown$ and marked with blue downward triangles. The total number of crossings, $\Nu$, is the sum of $\Nup$ and $\Ndown$.}
    \label{fig:description_upcrossing}
\end{figure*}

The study of higher-order moments of level 
crossings has a rich history, comprehensively 
surveyed in~\cite{Kratz2006_review}, with general 
expressions for these moments typically formulated 
as multiple integrals involving the process's derivatives. 
For the specific case of mean-level upcrossings ($u=0$), 
Steinberg {\it et al.}~\cite{steinberg1955short} 
derived an exact analytical solution expressed as 
a single integral over time, a result dating back 
to the 1950s. However, extending such a compact 
single-integral form to an arbitrary level $u$ has 
remained a longstanding open problem, owing to the 
asymmetric Gaussian integrals that arise when $u \neq 0$. 
Aza{\"\i}s {\it et al.}~\cite{azais1999bounds} and 
Cierco-Ayrolles {\it et al.}~\cite{cierco2003computing} 
obtained expressions for the variance of upcrossings 
at arbitrary levels as a byproduct of their work on 
the statistics of extreme values of Gaussian processes, 
but these were left in the form of multiple integrals 
rather than reduced to a compact single-integral 
expression. The present work provides this missing 
solution by deriving a closed-form, single-integral 
expression for the variance of arbitrary level 
upcrossings and total crossings, expressed in terms 
of the error function and Owen's $T$ function.

The paper is structured as follows: Section~\ref{section:description_of_problem} provides a precise mathematical formulation of the level-crossing problem for stationary Gaussian processes and states key results for the mean and variance. In Section~\ref{section:main_result_upcrossing}, we present our central contribution: exact analytical formulae for the variance and Fano factor of arbitrary level upcrossings and total crossings. Simplified expressions for mean-level crossings are detailed in Appendix~\ref{section:special_formulae_mean}. Finally, Section~\ref{section:applications_upcrossing} demonstrates the utility of our findings by applying them to three stochastic models, the damped harmonic oscillator, a mean-reverting process with Ornstein-Uhlenbeck noise, and a process with a rational quadratic correlation function, revealing a rich landscape of sub-Poissonian and super-Poissonian crossing statistics controlled by the correlation structure.

\section{Description of the Problem} \label{section:description_of_problem}

Let $\{X_t; \, t\in [0,T]\}$ denote a smooth (i.e., differentiable), stationary, zero-mean, Gaussian process characterized by the autocorrelation function $r(t)$, defined as
\begin{equation}
    r(t) \coloneqq R_X(s, s+t) = \E\left[X_s \, X_{s+t} \right].
\end{equation}
The function $r(t)$ is real-valued for all $t \, \in (-\infty,\infty)$ and satisfies the following properties:
\begin{itemize}
    \item $|r(t)| < r(0)$ for all $\, t \notin\{0\}$,
    \item $0 < r(0) < \infty$,
    \item $r(t) \rightarrow 0$ as $t \rightarrow \infty$.    
\end{itemize}
Due to stationarity, $r(t)$ depends only on the time difference $t$ and satisfies the symmetry property $r(-t) = r(t)$. Additionally, we assume that $X_t$ is smooth, i.e., $r(t)$ is twice-differentiable ($r^{\prime\prime}(t)$ exists). 

In this work, we find explicit expressions for the variance of the counting processes $\Nup$, $\Ndown$, and $\Nu$, which represent the number of upcrossings, downcrossings, and the total number of crossings, respectively, of a specified level $u$ by the process $X_t$ within the interval $t \in (0, T]$. All results derived for $\Nup$ also hold for $\Ndown$. Therefore, to avoid redundancy, we will primarily focus on $\Nup$, with the corresponding results for $\Ndown$ implied. The total number of crossings is simply given by the sum $\Nu = \Nup + \Ndown$. Additionally, we explore the Fano factors for these counting processes, a crucial metric that quantifies the dispersion, defined as the ratio between the variance and the mean of the counting process in the long-time limit. Formally, we have,
\begin{equation}
\begin{split}
    \FFup = \lim_{T \to \infty}\frac{\Var[\Nup]}{\E[\Nup]} \; \text{and} \;\FF = \lim_{T \to \infty}\frac{\Var[\Nu]}{\E[\Nu]}. 
\end{split}
\end{equation}

The counting processes $\Nup$ and $\Nu$ can be expressed as integrals of the underlying generator stochastic process $X_t$:
\begin{equation}
\begin{split}
    \Nup &= \int_0^T \delta\left(X_{t_1} - u\right) \, \mathrm{H}\left(\dot{X}_{t_1}\right) \dot{X}_{t_1} \, \dd t_1, \\
    \Nu &= \int_0^T \delta\left(X_{t_1} - u\right) \left|\dot{X}_{t_1}\right| \, \dd t_1.
\end{split} \label{eq:defn_Nu_delta_main}
\end{equation} 
Here, $\dot{X}_{t_1}$ denotes the derivative of the stochastic process, $\delta(x)$ represents the Dirac delta function, and $\mathrm{H}(x)$ is the Heaviside step function. The term $\delta\left(X_{t_1} - u\right)$ ensures that only the crossings of level $u$ are counted. The presence of $\mathrm{H}(\dot{X}_{t_1})$ restricts the count to upward crossings, while $|\dot{X}_{t_1}|$ accounts for both upward and downward crossings.

Since the autocorrelation function $r(t)$ is assumed to be twice differentiable, the expected values of $\Nup$ and $\Nu$ remain finite over a bounded interval $T$ \cite{ito1963expected, ylvisaker1965expected}. Furthermore, Geman \cite{geman1972variance} established that a necessary and sufficient condition for the variance of zero-level crossings, $\Var\left[N_0(T)\right]$, to be finite in a bounded interval $T$ is given by:
\begin{equation}
    \int_0^\epsilon \frac{\rpp(t) - \rpp(0)}{t}\, \dd t \, < \infty; \qquad \text{for some } \epsilon > 0. \label{eq:geman_condition}
\end{equation}
Later, Kratz \cite{kratz2006} demonstrated that this condition remains both necessary and sufficient for ensuring a finite variance of crossings at an arbitrary level, i.e., for $\Var\left[\Nu\right]$ to be finite over a bounded interval $T$. Consequently, this also guarantees a finite $\Var[\Nup]$. Throughout this work, we assume that this condition is satisfied by the correlation function.

\subsection{Mean of the Counting Processes}

The mean number of upcrossings follows from the Kac-Rice formula \cite{kac1943, rice1944mathematical}. We define the random vector $\mathbf{P} = [X_{0}, \, \dot{X}_{0}]^\top$, whose joint probability density function (p.d.f.) $f_{\mathbf{P}}(x,\,y)$ is a zero-mean bivariate Gaussian with covariance matrix
\begin{equation}
\renewcommand{\arraystretch}{1.5}
    \boldsymbol{\Sigma}_{\mathbf{P}} = \begin{bmatrix}
                r(0) & 0 \\
                0 & -\rpp(0)
                \end{bmatrix}, \label{eq:correlation_mat_P_main}
\end{equation}
where the off-diagonal entries vanish because $r^\prime(0)=0$ for a stationary process (see Appendix~\ref{section:appendix_upcrossing_mean} for details). The Kac-Rice formula then gives \cite{rice1944mathematical}
\begin{multline}
    \E\left[\Nup\right] = T \int_{0}^\infty y\, f_{\mathbf{P}}(u,y) \, \dd y \\
    = T \, \frac{1}{2\pi}\sqrt{\frac{-r^{\prime\prime}(0)}{r(0)}} \, e^{-\frac{u^2}{2r(0)}}. \label{eq:rice_formula_specific_main}
\end{multline}
Since $\E\left[\Nu\right]$ accounts for both upward and downward crossings, it is simply $2\, \E[\Nup]$:
\begin{equation}
    \E\left[\Nu\right] = T\, \frac{1}{\pi}\sqrt{\frac{-r^{\prime\prime}(0)}{r(0)}} \, e^{-\frac{u^2}{2r(0)}}.
\end{equation}
Notably, the mean crossing rate depends only on the local properties of the autocorrelation function, its value and curvature at the origin, and is independent of the detailed shape of $r(t)$ at finite lag times.

\subsection{Variance of the Counting Processes} \label{section:variance_derivation}

In contrast to the mean, which depends only on the local properties $r(0)$ and $r^{\prime\prime}(0)$, the variance encodes information about the full correlation structure of the process. To compute the variance, we consider the joint p.d.f. of the random vector $\mathbf{Q} = [X_{t_1}, \, X_{t_2}, \, \dot{X}_{t_1}, \, \dot{X}_{t_2}]^\top$ at two distinct times, which is a zero-mean multivariate Gaussian with the $4\times4$ covariance matrix:
\begin{equation}
\renewcommand{\arraystretch}{1.2}
    \boldsymbol{\Sigma}_{\mathbf{Q}}(t) =
                \begin{bmatrix}
                r(0) & r(t) & 0 & \rp(t) \\
                r(t) & r(0) & -\rp(t) & 0 \\
                0 & -\rp(t) & -\rpp(0) & -\rpp(t) \\
                \rp(t) & 0 & -\rpp(t) & -\rpp(0) \\
                \end{bmatrix}, \label{eq:correlation_mat_Q_main}
\end{equation}
where $t=t_2-t_1$. Using the second factorial moment and a change of variables (see Appendix~\ref{section:variance_derivation1} for a detailed derivation), the finite-time variance of $\Nup$ can be expressed as \cite{cramer_leadbetter_book, piterbarg1996asymptotic}:
\begin{multline}
    \Var[\Nup] = \E\left[\Nup\right] + 2 \,T \int_{0}^T \!\!\int_{0}^\infty \!\!\int_{0}^\infty \biggl(1-\frac{t}{T}\biggr) y_1 y_2 \\
    \times\bigl( f_{\mathbf{Q}}(u, u, y_1, y_2; \, t) \\
    - f_{\mathbf{P}}(u,y_1) \, f_{\mathbf{P}}(u,y_2)\bigr)\, \dd y_1 \, \dd y_2 \, \dd t.
\label{eq:variance_formula_triple_integral_main}
\end{multline}
The integrand measures the excess joint probability of two crossings at lag $t$ relative to the independent (Poisson) case. It can be further shown \cite{piterbarg1996asymptotic} that the asymptotic variance rate $\lim_{T \rightarrow \infty}\Var\left[\Nup\right] / T$ is finite provided $\int_0^\infty t \left(\left|r(t)\right| + \left|\rp(t)\right|+ \left|\rpp(t)\right|\right) \, \dd t < \infty$, yielding:
\begin{multline}
    \lim_{T\rightarrow \infty}\frac{\Var\left[\Nup\right]}{T} = \lim_{T\rightarrow \infty}\frac{\E\left[\Nup\right]}{T} \\
    + 2 \int_{0}^\infty \!\!\int_{0}^\infty \!\!\int_{0}^\infty y_1 y_2 \,\bigl( f_{\mathbf{Q}}(u, u, y_1, y_2; t) \\
    - f_{\mathbf{P}}(u,y_1) \, f_{\mathbf{P}}(u,y_2)\bigr)\, \dd y_1 \, \dd y_2 \, \dd t.
\label{eq:variance_formula_triple_integral_main_CLT}
\end{multline}
The Fano factor, defined as the ratio of variance to mean in the long-time limit, serves as a dimensionless measure of the regularity of crossing events. A Fano factor of unity corresponds to Poisson-distributed crossings; values below unity ($\FFup < 1$) indicate sub-Poissonian regularity (anti-bunching), while values above unity ($\FFup > 1$) signal super-Poissonian clustering (bunching). For upcrossings:
\begin{equation}
    \FFup = 1 + \frac{2\displaystyle\int_{0}^\infty \!\!\int_{0}^\infty \!\!\int_{0}^\infty y_1 y_2 \,G^\uparrow(t)\, \dd y_1 \, \dd y_2 \, \dd t}{\displaystyle\int_{0}^\infty y\, f_{\mathbf{P}}(u,\,y) \, \dd y}\, , \label{eq:fano_factor_upcrossing_1_main}
\end{equation}
where $G^\uparrow(t) = f_{\mathbf{Q}}(u, u, y_1, y_2; t) - f_{\mathbf{P}}(u,y_1) f_{\mathbf{P}}(u,y_2)$.
Using a similar approach, we can derive these quantities for the total crossing process $\Nu$ (see Appendix~\ref{section:crossings} for details):
\begin{multline}
    \Var[\Nu] = \E\left[\Nu\right] + 2 \,T \int_{0}^T \!\!\int_{-\infty}^\infty \!\!\int_{-\infty}^\infty \biggl(1-\frac{t}{T}\biggr) \\
    \times \left|y_1 y_2 \right|\bigl( f_{\mathbf{Q}}(u, u, y_1, y_2; \, t) \\
    - f_{\mathbf{P}}(u,y_1) \, f_{\mathbf{P}}(u,y_2)\bigr)\, \dd y_1 \, \dd y_2 \, \dd t,
\end{multline}
\begin{equation}
    \FF = 1 + \frac{2\displaystyle\int_{0}^\infty \!\!\int_{-\infty}^\infty \!\!\int_{-\infty}^\infty \left|y_1 y_2 \right| G(t)\, \dd y_1 \, \dd y_2 \, \dd t}{\displaystyle\int_{-\infty}^\infty |y|\, f_{\mathbf{P}}(u,\,y) \, \dd y}\, , \label{eq:fano_factor_upcrossing_1_main_cross}
\end{equation}
where $G(t) = f_{\mathbf{Q}}(u, u, y_1, y_2; t) - f_{\mathbf{P}}(u,y_1) f_{\mathbf{P}}(u,y_2)$.

While these expressions are exact, the double integrals over velocities ($y_1$, $y_2$) remain to be evaluated explicitly for a given level $u$. This is the central challenge addressed in the next section.

\begin{figure*}[!t]
    \centering
    \includegraphics[width=\textwidth]{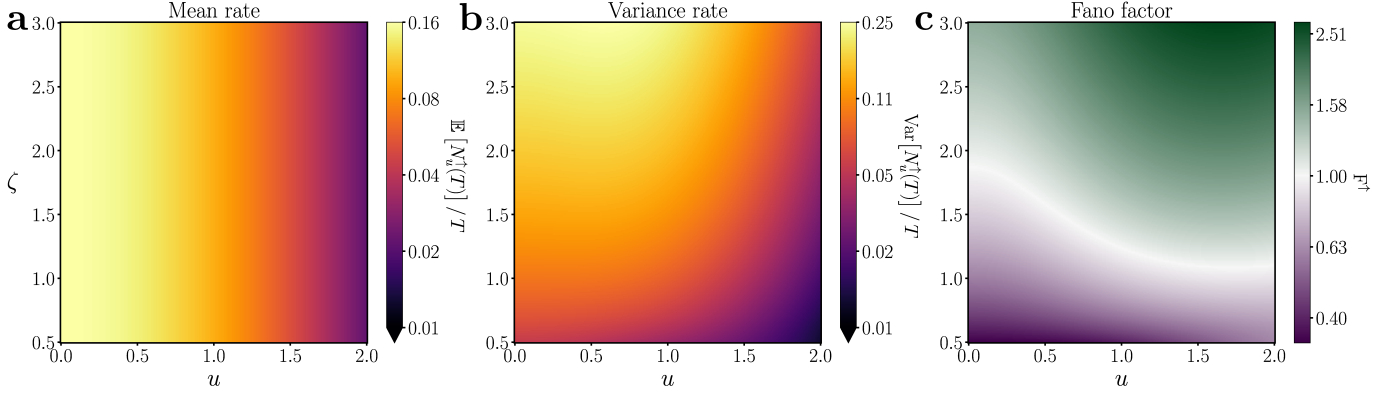}
    \caption[Mean rate, variance rate, and Fano factor of upcrossings for the stochastic damped harmonic oscillator as functions of level $u$ and damping ratio $\zeta$.]{\textbf{Mean rate, variance rate, and Fano factor of upcrossings for the stochastic damped harmonic oscillator as functions of level $u$ and damping ratio $\zeta$.}
    We plot (\textbf{a}) the mean rate of upcrossings $\E[\Nup]/T$; (\textbf{b}) the variance rate of upcrossings $\Var[\Nup]/T$; and (\textbf{c}) the upcrossing Fano factor $\FFup$ for the stochastic damped harmonic oscillator across various combinations of threshold ($u$) and damping constant ($\zeta$). The natural oscillation frequency is fixed at $\omega_0=1.0$ and the temperature at $\vartheta=1.0$, resulting in a process variance $r(0) = \vartheta/\omega_0^2 = 1.0$. Color bars indicate the magnitude of the respective quantities on a logarithmic scale. All plotted values are calculated using the analytical formulae derived in this work. Notably, for a fixed threshold $u$, the mean rate (\textbf{a}) is independent of the damping ratio $\zeta$, while the variance rate (\textbf{b}) is $\zeta$-dependent. Consequently, the Fano factor (\textbf{c}) spans a broad range, revealing both sub-Poissonian and super-Poissonian upcrossing statistics depending on the specific ($u, \zeta$) combination.
    }
    \label{fig:sdho}
\end{figure*}

\section{Main Result}
\label{section:main_result_upcrossing}
We first present the main result of this work as a theorem, and prove it subsequently.
\begin{theorem}
\label{thrm:main_theorem_upcrossing}
Consider a zero-mean stationary Gaussian stochastic process $X_t$ with the autocorrelation function $r(t)$ satisfying the properties defined in Section~\ref{section:description_of_problem}. Let $\Nup$ denote the number of upcrossings of an arbitrary level $u$.
Denoting $r(t)\rightarrow r$, $r(0)\rightarrow r_0$, $\rp(t) \rightarrow p$, $\rp(0)\rightarrow p_0$, $-\rpp(t) \rightarrow q$, $-\rpp(0)\rightarrow q_0$, and 
\begin{equation}
\begin{aligned}
\alpha &=-\frac{r+r_0}{2\left(p^2+(q-q_0)(r+r_0)\right)},\\
\beta &= -\frac{r_0-r}{2\left(p^2+(q+q_0)(r-r_0)\right)},\\
\gamma &=\frac{\sqrt{2}\,p}{r+r_0}\,u,\\
\delta &= \frac{1}{r+r_0}\,.
\end{aligned}
\end{equation}
The mean of $\Nup$ is given by:
\begin{equation}
    \E\left[\Nup\right] = T\, \frac{1}{2\pi}\sqrt{\frac{q_0}{r_0}} \, e^{-u^2/(2r_0)}, \label{eq:rice_formula_gaussian}
\end{equation}
and the finite-time variance of $\Nup$ is:
\begin{equation}
    \textnormal{Var}\left[\Nup\right] = \E\left[\Nup\right] + 2T \int_{0}^T \biggl(1-\frac{t}{T}\biggr)  \Iup  \dd t  \,,\label{eq:variance_finite_time_final}
\end{equation}
where $\Iup$ is a function of the time-dependent quantities $r, \, \alpha, \, \beta, \, \gamma, \, \delta$ given by:
\begin{multline}
    \Iup = \frac{e^{-\delta\,u^2}}{4\pi^2 \sqrt{r_0^2-r^2}} \Biggl[\frac{e^{-\alpha \gamma^2}}{2\sqrt{\alpha\beta}} \biggl(1 \\
    + \sqrt{\pi} \, \gamma \sqrt{\alpha{+}\beta} \, e^{\frac{\alpha^2 \gamma^2}{\alpha + \beta}} \, \Erf\Bigl(\frac{\alpha \gamma}{\sqrt{\alpha{+}\beta}}\Bigr)\biggr) \\
    + \pi \biggl(\frac{\alpha - \beta - 2\alpha\beta\gamma^2}{\alpha\beta}\biggr) \Ot\Bigl(\gamma\!\sqrt{\frac{2\alpha\beta}{\alpha{+}\beta}},\, \sqrt{\frac{\alpha}{\beta}}\Bigr)\Biggr] \\
    - \frac{1}{4\pi^2}\frac{q_0}{r_0}\,e^{-u^2/r_0},
\label{eq:Iup_final_result_owensT}
\end{multline}
where $\Erf(z)$ is the error function defined as:
\begin{equation}
\Erf(z) = \frac{2}{\sqrt{\pi}} \int_0^z e^{-t^2}  \dd t, 
\end{equation}
and $\Ot(h, a)$ is Owen's T function \cite{owen1980table} defined as:
\begin{equation}
\Ot(h, a) = \frac{1}{2\pi} \int_0^a \frac{e^{-\frac{1}{2}h^2(1+t^2)}}{1+t^2} \, \dd t. 
\end{equation}
\end{theorem}

\textit{Proof sketch.} The key challenge is the explicit evaluation of the double integral over the velocity variables $(y_1, y_2)$ in Eq.~\eqref{eq:variance_formula_triple_integral_main} for an arbitrary level $u$. To the best of our knowledge, this has only been accomplished for mean-level crossings ($u=0$) \cite{steinberg1955short, leadbetter1965variance, pawula1968analysis, lindgren1974spectral, wilson2017periodicity}, but not for $u\neq 0$ due to the asymmetric Gaussian integral that arises. We overcome this difficulty as follows. First, we introduce a $\pi/4$ rotation of the velocity variables, $(y_1, y_2) \rightarrow (z_1, z_2)$, which maps the positive quadrant onto the region $|z_1| < z_2$ and transforms the product $y_1 y_2$ into $(z_2^2 - z_1^2)/2$. In the rotated frame, the quadratic form in the exponent of the joint p.d.f. $f_{\mathbf{Q}}$ can be simplified by completing the square in $z_1$, yielding $S = -\delta u^2 - \alpha(z_1 - \gamma)^2 - \beta z_2^2$. The parameters $\alpha > 0$ and $\beta > 0$ (their positivity is established via a Cauchy--Schwarz argument on the spectral measure; see Theorem~\ref{thrm:positivity_of_correlation_term}), while $\gamma \propto u \, r^\prime(t)$ captures the asymmetry introduced by a nonzero threshold. The resulting integral takes the canonical form $\int\!\!\int (z_2^2 - z_1^2) \exp(-\alpha(z_1 - \gamma)^2 - \beta z_2^2) \, \dd z_2 \, \dd z_1$, which we evaluate analytically (Theorem~\ref{thrm:integral_main}) using integration by parts combined with known identities for the error function and Owen's $T$ function. The full derivation is provided in Appendix~\ref{section:proof_theorem_upcrossing}.

\begin{figure*}[!t]
    \centering
    \includegraphics[width=\textwidth]{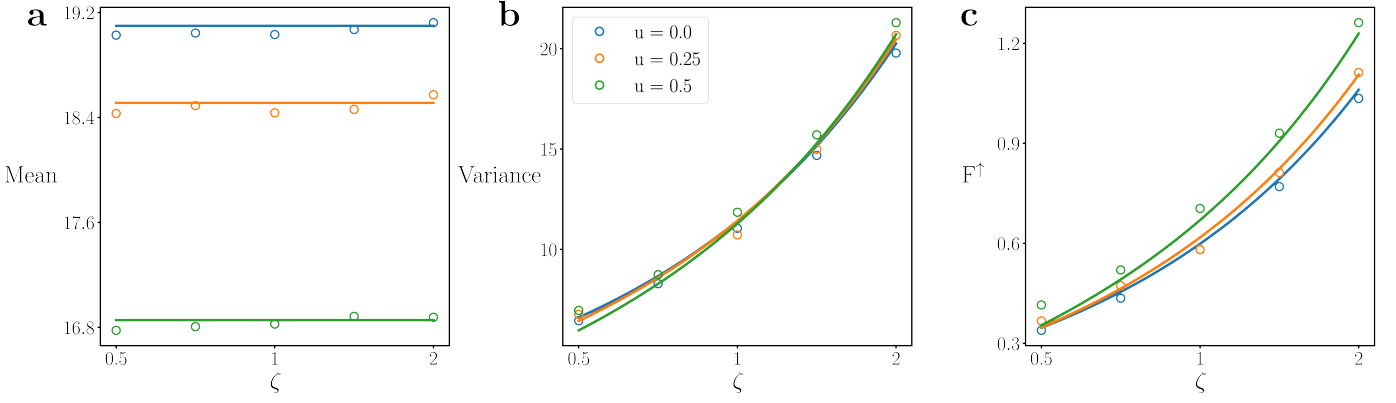}
    \caption[Comparison of analytical formulae with numerical simulations for the stochastic damped harmonic oscillator.]{\textbf{Comparison of analytical formulae with numerical simulations for the stochastic damped harmonic oscillator.}
    We plot (\textbf{a}) the mean number of upcrossings $\E[\Nup]$; (\textbf{b}) the variance of upcrossings $\Var[\Nup]$; and (\textbf{c}) the upcrossing Fano factor $\FFup$ as a function of the damping ratio $\zeta$ for three different threshold levels, $u=0.0, \, 0.25, \,0.5$.
    Solid lines represent the analytical solutions derived in this work, while circles denote results obtained from numerical simulations.
    The model parameters were fixed at $\omega_0=1.0$~Hz and $\temp=1.0$, for a total simulation time of $T=120$~s for each trial. For the numerical simulations, we used a discretization time step $dt=0.01 \times (\text{slowest timescale of the system})$ and computed the statistics from $5000$ independent trials. We find excellent agreement between the analytical predictions and the simulation results across the explored parameter range.
    }
    \label{fig:sdho_comparison}
\end{figure*}

In the long-time limit ($T \to \infty$), the variance rate and Fano factor take the compact forms:
\begin{equation}
    \lim_{T\rightarrow\infty}\frac{\textnormal{Var}\left[\Nup\right]}{T} =  \frac{1}{2\pi}\sqrt{\frac{q_0}{r_0}} \, e^{-\frac{u^2}{2r_0}} + 2 \int_{0}^\infty \Iup \dd t, \label{eq:variance_infinite_time_final}
\end{equation}
\begin{equation}
    \FFup = 1 + 4 \pi \sqrt{\frac{r_0}{q_0}} \, e^{\frac{u^2}{2r_0}} \int_{0}^\infty \Iup \, \dd t. \label{eq:fano_factor_final}
\end{equation}
Several important limits of these expressions deserve comment. First, when $u=0$ (mean-level crossings), the parameter $\gamma$ vanishes and Owen's $T$ function reduces to an arctangent, recovering the classical results of Steinberg \cite{steinberg1955short} and Leadbetter \cite{leadbetter1965variance} (see Appendix~\ref{section:special_formulae_mean}). Second, in the limit $u/\sqrt{r_0} \to \infty$, crossings become increasingly rare and uncorrelated, so the Fano factor approaches unity, consistent with the well-known result that rare threshold crossings become Poisson-distributed \cite{cramer_leadbetter_book}. Finally, the single-integral formulation in Eqs.~\eqref{eq:variance_finite_time_final}--\eqref{eq:fano_factor_final} is straightforward to evaluate numerically, requiring only standard quadrature and readily available implementations of the error function and Owen's $T$ function.

\begin{theorem}[See Appendix~\ref{section:explicit_formulae_for_crossings} for proof]
\label{thrm:main_theorem_crossing}
Consider a zero-mean stationary Gaussian stochastic process $X_t$ with the autocorrelation function $r(t)$ satisfying the properties defined in Section~\ref{section:description_of_problem}. Let $\Nu$ denote the number of total crossings of an arbitrary level $u$.
The mean of $\Nu$ is given by:
\begin{equation}
    \E\left[\Nu\right] = T\, \frac{1}{\pi}\sqrt{\frac{q_0}{r_0}} \, e^{-u^2/(2r_0)}, \label{eq:rice_formula_gaussian_cross}
\end{equation}
and the finite-time variance of $\Nu$ is:
\begin{equation}
    \textnormal{Var}\left[\Nu\right] = \E\left[\Nu\right] + 2T \int_{0}^T \biggl(1-\frac{t}{T}\biggr)  I  \dd t\,, \label{eq:variance_finite_time_final_cross}
\end{equation}
where $I$ is a function of the time-dependent quantities $r, \, \alpha, \, \beta, \, \gamma, \, \delta$ given by:
\begin{multline}
    I = \frac{e^{-\delta\,u^2}}{4\pi^2\sqrt{r_0^2-r^2}} \Biggl[\frac{2e^{-\alpha \gamma^2}}{\sqrt{\alpha\beta}} \biggl(1 \\
    + \sqrt{\pi} \, \gamma \sqrt{\alpha{+}\beta} \, e^{\frac{\alpha^2 \gamma^2}{\alpha + \beta}} \, \Erf\Bigl(\frac{\alpha \gamma}{\sqrt{\alpha{+}\beta}}\Bigr)\biggr) \\
    + 4\pi \biggl(\frac{\alpha - \beta - 2\alpha \beta \gamma^2}{\alpha\beta}\biggr) \\
    \times\biggl(\Ot\Bigl(\gamma\!\sqrt{\frac{2\alpha\beta}{\alpha{+}\beta}},\, \sqrt{\frac{\alpha}{\beta}}\Bigr)-\frac{1}{8}\biggr)\Biggr] \\
    - \frac{1}{\pi^2}\frac{q_0}{r_0}\,e^{-u^2/r_0}.
\label{eq:I_final_result_owensT_crossing_main}
\end{multline}
\end{theorem}

\begin{corollary}[See Appendix~\ref{section:explicit_formulae_for_crossings} for proof]
In the long-time limit, the variance becomes:
\begin{equation}
    \lim_{T\rightarrow\infty}\frac{\textnormal{Var}\left[\Nu\right]}{T} =  \frac{1}{\pi}\sqrt{\frac{q_0}{r_0}} \, e^{-\frac{u^2}{2r_0}} + 2 \int_{0}^\infty I \dd t \label{eq:variance_infinite_time_final_cross}
\end{equation}
The corresponding Fano factor is:
\begin{equation}
    \FF = 1 + 2 \pi \sqrt{\frac{r_0}{q_0}} \, e^{\frac{u^2}{2r_0}} \int_{0}^\infty I \, \dd t \label{eq:fano_factor_final_cross}
\end{equation}
\end{corollary}

\section{Applications}
\label{section:applications_upcrossing}

In this section, we verify the explicit variance and Fano factor results for upcrossings and total crossings of stationary Gaussian processes with different correlation functions. 

\subsection{Stochastic Damped Harmonic Oscillator}
\label{sec:sdho}

\begin{figure*}[!t]
    \centering
    \includegraphics[width=\textwidth]{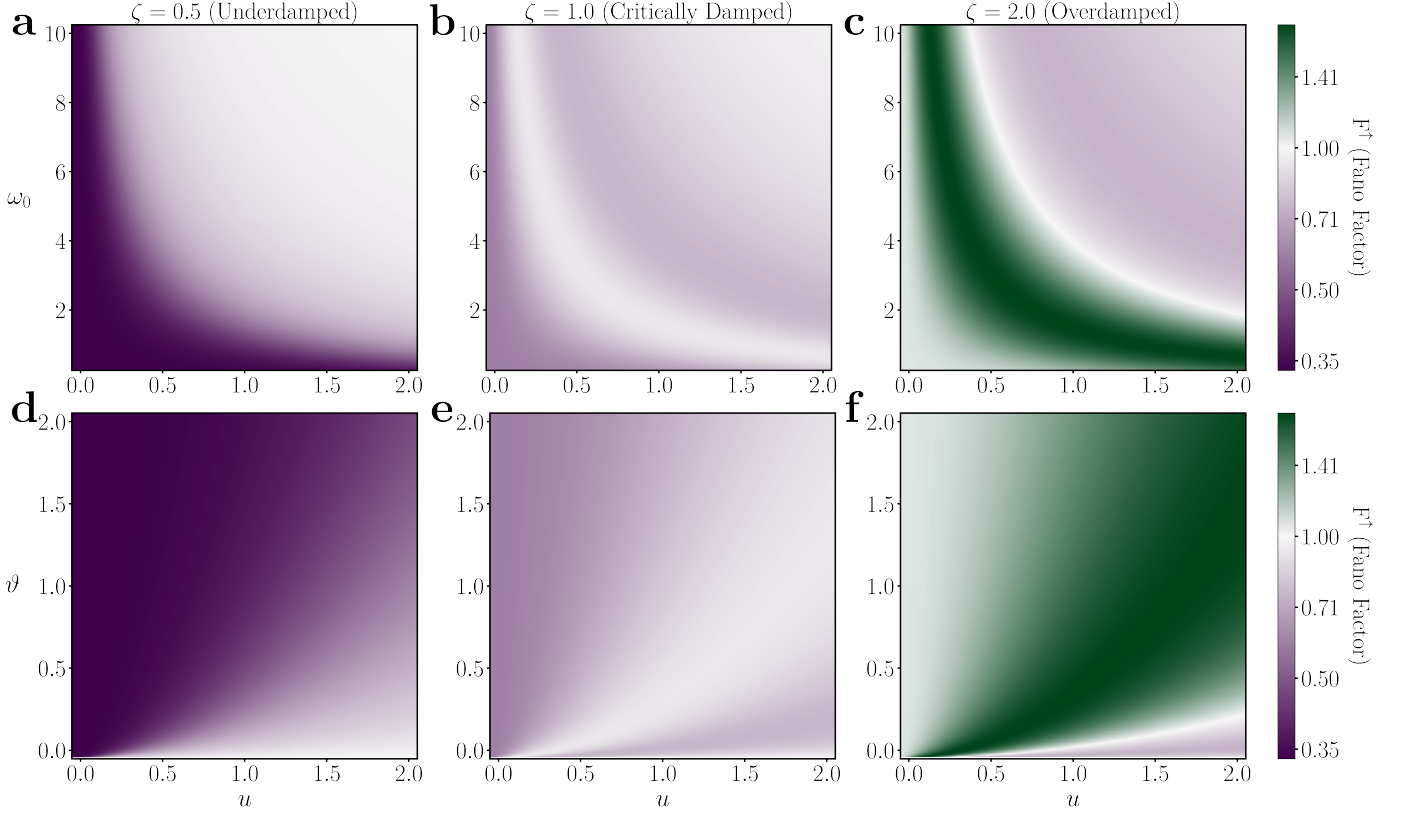}
    \caption[Fano factor landscapes for the stochastic damped harmonic oscillator across different parameter regimes.]{\textbf{Fano factor landscapes for the stochastic damped harmonic oscillator across different parameter regimes.}
    We plot the analytical upcrossing Fano factor ($\FFup$) for the stochastic damped harmonic oscillator across various parameter configurations.
    (\textbf{a}--\textbf{c}) Fano factor as a function of the natural frequency $\omega_0$ and threshold level $u$, for a fixed temperature $\vartheta=1.0$ and different damping ratios: (\textbf{a}) underdamped $\zeta=0.5$; (\textbf{b}) critically damped $\zeta=1.0$; and (\textbf{c}) overdamped $\zeta=2.0$. 
    (\textbf{d}--\textbf{f}) Fano factor as a function of the temperature (noise strength parameter) $\vartheta$ and threshold level $u$, for a fixed natural frequency $\omega_0=1.0$ and the same damping ratios: (\textbf{d}) underdamped $\zeta=0.5$; (\textbf{e}) critically damped $\zeta=1.0$; and (\textbf{f}) overdamped $\zeta=2.0$. 
    The color bars indicate the magnitude of the Fano factor on a logarithmic scale, where values near $\FFup=1.0$ are represented by lighter colors, sub-Poissonian values by purples, and super-Poissonian values by greens.
    }
    \label{fig:sdho_fanos}
\end{figure*}

The stochastic damped harmonic oscillator is a fundamental model in statistical physics \cite{doob1942brownian, west1982linear, chechkin2000linear, lin2011undamped, dybiec2017underdamped} describing systems subject to a restoring force, damping, and thermal fluctuations. While the first-order statistics of upcrossings for this model have been studied previously \cite{verechtchaguina2006first, masoliver2023counting, masoliver2025level}, here we apply our results to obtain the full second-order statistics.

The oscillator is governed by the stochastic differential equation
\begin{equation}
    \frac{\dd^2 x}{\dd t^2} + 2 \zeta \omega_0 \frac{\dd x}{\dd t} + \omega_0^2 x = \sqrt{4 \zeta \omega_0 \temp} \, \eta(t), \label{eq:sdho_sde}
\end{equation}
where $\omega_0 > 0$ is the natural frequency, $\zeta \ge 0$ is the dimensionless damping ratio, $\temp$ is the temperature (with $k_B = 1$), and $\eta(t)$ is zero-mean Gaussian white noise with $\E[\eta(t)\eta(s)]=\delta(t-s)$. The noise amplitude satisfies the fluctuation-dissipation theorem. The stationary position $x(t)$ is a zero-mean Gaussian process with variance $r_x(0)=\temp/\omega_0^2$ in all damping regimes. The autocorrelation function $r_x(t)$, obtained via the Wiener-Khinchin theorem from the power spectral density \cite{rawat2024element} (see Appendix~\ref{section:sdho_derivation} for the details), takes the following well-known forms \cite{doob1942brownian}:

\paragraph{Underdamped (\( 0 < \zeta < 1 \)):}
\begin{multline}
    r_x(t) = \frac{\temp}{\omega_0^2} e^{-\zeta \omega_0 |t|} \biggl( \cos\bigl(\omega_0 \sqrt{1{-}\zeta^2}\, |t|\bigr) \\
    + \frac{\zeta}{\sqrt{1{-}\zeta^2}} \sin\bigl(\omega_0 \sqrt{1{-}\zeta^2}\, |t|\bigr) \biggr).
\label{eq:acf_underdamped}
\end{multline}

\paragraph{Critically Damped (\( \zeta = 1 \)):}
\begin{equation}
    r_x(t) = \frac{\temp}{\omega_0^2} e^{-\omega_0 |t|} \left( 1 + \omega_0 |t| \right).
\label{eq:acf_critical}
\end{equation}

\paragraph{Overdamped (\( \zeta > 1 \)):}
\begin{multline}
    r_x(t) = \frac{\temp}{2 \omega_0^2} \biggl[ \Bigl(1 + \frac{\zeta}{\sqrt{\zeta^2{-}1}}\Bigr) e^{-\omega_0(\zeta - \sqrt{\zeta^2-1})|t|} \\
    + \Bigl(1 - \frac{\zeta}{\sqrt{\zeta^2{-}1}}\Bigr) e^{-\omega_0(\zeta + \sqrt{\zeta^2-1})|t|} \biggr].
\label{eq:acf_overdamped}
\end{multline}

\begin{figure*}[!t]
    \centering
    \includegraphics[width=\textwidth]{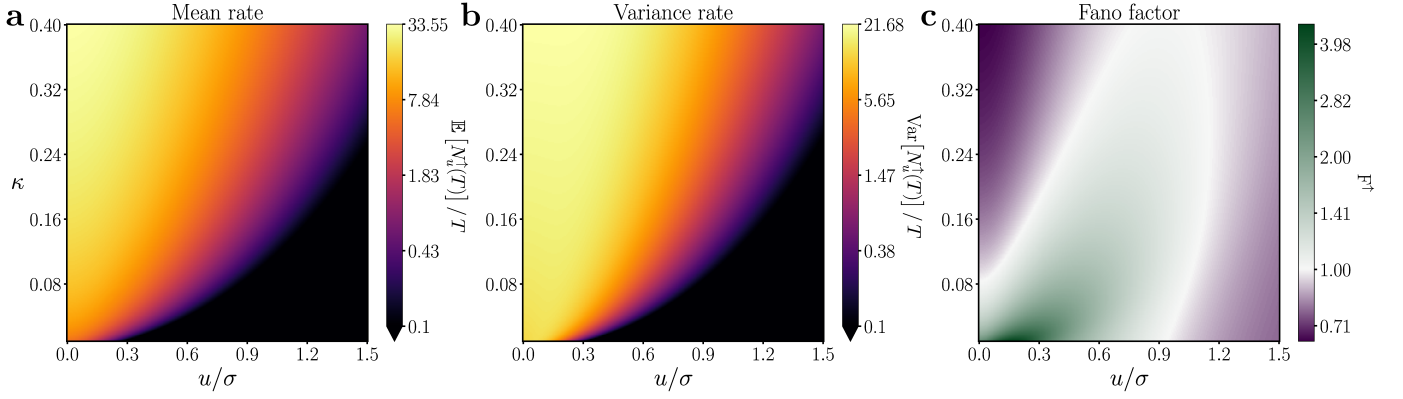} 
    \caption[Mean rate, variance rate, and Fano factor of upcrossings for a mean-reverting process driven by Ornstein-Uhlenbeck noise.]{\textbf{Mean rate, variance rate, and Fano factor of upcrossings for a mean-reverting process driven by Ornstein-Uhlenbeck noise.}
    We plot (\textbf{a}) the mean rate of upcrossings $\E[\Nup]/T$; (\textbf{b}) the variance rate of upcrossings $\Var[\Nup]/T$; and (\textbf{c}) the upcrossing Fano factor $\FFup$ for a mean-reverting process with characteristic timescale $\tau_e$, driven by Ornstein-Uhlenbeck (OU) noise with correlation time $\tau_f$; this system is described by Eq.~\eqref{eq:OU_noise_eqns}. The plots illustrate the dependence on two dimensionless parameters: the normalized threshold $u/\sigma$ and the timescale ratio $\kappa=\tau_f/\tau_e$.
    We set the standard deviation of the input OU noise $\sigma=1.0$ and its correlation time $\tau_f=3$~ms.
    Note that the resulting variance of the observed mean-reverting process $y(t)$ is $r_y(0)=\sigma^2 \kappa/(1+\kappa)$. 
    Color bars indicate the magnitude of the respective quantities on a logarithmic scale. All plotted values are computed using the analytical formulae derived in this work. We find that the Fano factor (\textbf{c}) spans a broad range, revealing both sub-Poissonian and super-Poissonian upcrossing statistics depending on the specific ($u/\sigma, \kappa$) combination.
    }
    \label{fig:OU_noise}
\end{figure*}

With these correlation functions, we can now compute the mean, variance, and Fano factor for upcrossings using Theorem~\ref{thrm:main_theorem_upcrossing}. In Fig.~\ref{fig:sdho}, we present these quantities as functions of the threshold level $u$ and the damping ratio $\zeta$, with $\temp=\omega_0=1$ so that $r_x(0)=1.0$ for all $\zeta$.

Notably, the mean upcrossing rate (Fig.~\ref{fig:sdho}a) is independent of $\zeta$ for a given $u$. This follows directly from the Rice formula (Eq.~\eqref{eq:rice_formula_gaussian}), which depends only on $r_x(0)$ and $q_0 = -r_x^{\prime\prime}(0)$, quantities determined by the equipartition theorem ($r_x(0) = \temp/\omega_0^2$ and $q_0 = \temp$) and thus insensitive to the damping mechanism. In contrast, the variance rate and Fano factor (Fig.~\ref{fig:sdho}b,c) are strongly $\zeta$-dependent, spanning a wide range from sub-Poissonian to super-Poissonian. Physically, in the underdamped regime, oscillatory correlations impose a quasi-periodic structure on the crossing times, leading to regularity and sub-Poissonian statistics ($\FFup < 1$). 
In the overdamped regime, the oscillatory mechanism 
is lost, and long-lived monotonic excursions above 
the threshold can generate bursts of closely spaced 
crossings, resulting in super-Poissonian statistics 
($\FFup > 1$) across all threshold levels for 
sufficiently large damping ($\zeta\gtrsim 2$), as shown 
in Fig.~\ref{fig:sdho}c. 
This demonstrates that the full correlation structure 
of the process, not merely its local properties at $t=0$, 
governs the variability of crossing events.

To validate our analytical framework, we compared the analytical solutions with direct numerical simulations (Fig.~\ref{fig:sdho_comparison}) and find excellent agreement across a wide range of parameters and threshold levels.

We further investigate the behavior of the upcrossing Fano factor $\FFup$, detailed in Fig.~\ref{fig:sdho_fanos}, by examining its landscape across various parameter configurations. First, we set the temperature to unity, i.e., $\temp=1.0$, and compute $\FFup$ as a function of the threshold $u$ and the natural frequency $\omega_0$ for different damping regimes (Fig.~\ref{fig:sdho_fanos}a--c). Then we set the natural frequency $\omega_0$ to unity and compute $\FFup$ as a function of the threshold $u$ and temperature $\temp$ for the same damping regimes (Fig.~\ref{fig:sdho_fanos}d--f). These heatmaps reveal distinct characteristics of $\FFup$ based on the damping regime:
\begin{itemize}
    \item In the underdamped ($\zeta=0.5$; plots \textbf{a}, \textbf{d}) and critically damped ($\zeta=1.0$; plots \textbf{b}, \textbf{e}) regimes, the statistics are predominantly sub-Poissonian ($\FFup < 1$), reflecting the regularity imposed by oscillatory or near-oscillatory correlations.
    \item In the overdamped regime ($\zeta=2.0$; plots \textbf{c}, \textbf{f}), regions of both sub-Poissonian and super-Poissonian ($\FFup > 1$) statistics are evident, reflecting the competition between Gaussian suppression of high-threshold crossings and clustering induced by long-lived correlations.
    \item Non-monotonic behavior in $\FFup$ is observed 
    in the critically damped and overdamped regimes, 
    reflecting a complex dependence on both the threshold 
    level $u$ and the damping ratio $\zeta$, as shown in 
    Fig.~\ref{fig:sdho_fanos}b,c,e,f.   
\end{itemize}

\subsection{Mean Reverting Process with Ornstein-Uhlenbeck Noise}
\label{sec:filtered_noise}
We consider a mean-reverting process driven by Ornstein-Uhlenbeck noise. This model is relevant in various fields, including finance for modeling stochastic volatility and in neuroscience for describing the membrane equation with filtered synaptic inputs \cite{moreno2002response}. The dynamics of the system are given by the following set of stochastic differential equations:
\begin{equation}
\begin{split}
    \frac{\mathrm{d}x}{\mathrm{d}t} &= -\frac{x}{\tau_f} + \sqrt{\frac{2 \sigma^2}{\tau_f}}\,\eta(t) \\
    \frac{\mathrm{d}y}{\mathrm{d}t} &= -\frac{1}{\tau_e}\Bigl(y - x\Bigr).
\end{split} \label{eq:OU_noise_eqns}
\end{equation}
Here, $x(t)$ is an Ornstein-Uhlenbeck process obtained by exponentially low-pass filtering the white-noise process \(\eta(t)\), where \(\E\left[\eta(t)\,\eta(s)\right]=\delta(t-s)\). The OU process $x(t)$ has a characteristic correlation time $\tau_f$, and its fluctuations are scaled by the parameter $\sigma$. The process $y(t)$ receives $x(t)$ as its input and exhibits mean-reverting behavior toward $0$ with a characteristic timescale $\tau_e$. Since $y(t)$ is a linear transformation of the Gaussian process $x(t)$, it is also a stationary Gaussian process.

To analyze the level-crossing statistics of $y(t)$, we first need its autocorrelation function. Following a procedure similar to that used for the stochastic damped harmonic oscillator (i.e., by solving the system in the Fourier domain), the autocorrelation function $r_y(t) = \E[y(s)y(s+t)]$ for the process $y(t)$ is found to be:
\begin{equation}
    r_y(t) = \frac{\sigma^2\kappa}{1 -\kappa^2} \left(e^{-|t|/\tau_e} - \kappa \, e^{-|t|/(\kappa \tau_e)} \right), \label{eq:corr_func_OU_noise}
\end{equation}
where $\kappa=\tau_f/\tau_e$ is the dimensionless ratio of the input noise correlation timescale $\tau_f$ to the system's relaxation timescale $\tau_e$.  The variance of this process is $r_y(0) = \sigma^2\kappa/(1+\kappa)$.

Using this autocorrelation function $r_y(t)$, we can apply Theorem~\ref{thrm:main_theorem_upcrossing} to compute the mean, variance, and Fano factor for level upcrossings of $y(t)$. The parameter $\kappa = \tau_f / \tau_e$ controls how filtered the noise appears to the mean-reverting process: when $\kappa \ll 1$, the driving noise varies much faster than the system 
relaxes and therefore appears nearly white, whereas 
when $\kappa\sim1$ or larger the noise varies slowly 
relative to the system's relaxation timescale. 

Figure~\ref{fig:OU_noise} illustrates these quantities 
as functions of the normalized threshold $u/\sigma$ 
and $\kappa$. The Fano factor (Fig.~\ref{fig:OU_noise}c) 
spans a broad range from sub-Poissonian to super-Poissonian, mirroring the behavior of the overdamped harmonic oscillator. 
This similarity arises because the bi-exponential form 
of $r_y(t)$ in Eq.~\eqref{eq:corr_func_OU_noise} can be 
mapped exactly to the overdamped SDHO correlation function (Eq.~\eqref{eq:acf_overdamped}), illustrating how physically distinct systems can exhibit identical crossing statistics 
when their correlation structures coincide. 
A particularly striking feature of Fig.~\ref{fig:OU_noise}c 
is the reentrant behavior observed for intermediate 
values of $\kappa$: as the threshold $u/\sigma$ is 
increased, the Fano factor transitions from the sub-Poissonian 
regime to the super-Poissonian regime and back. This 
reentrant behavior highlights how the interplay between 
the threshold level and the correlation structure can 
produce a rich statistical landscape even in a simple 
two-timescale system.

\subsection{Process with Rational Quadratic Correlation Function}
\label{sec:rational_quadratic}
\begin{figure}[!t]
    \centering
    \includegraphics[width=\columnwidth]{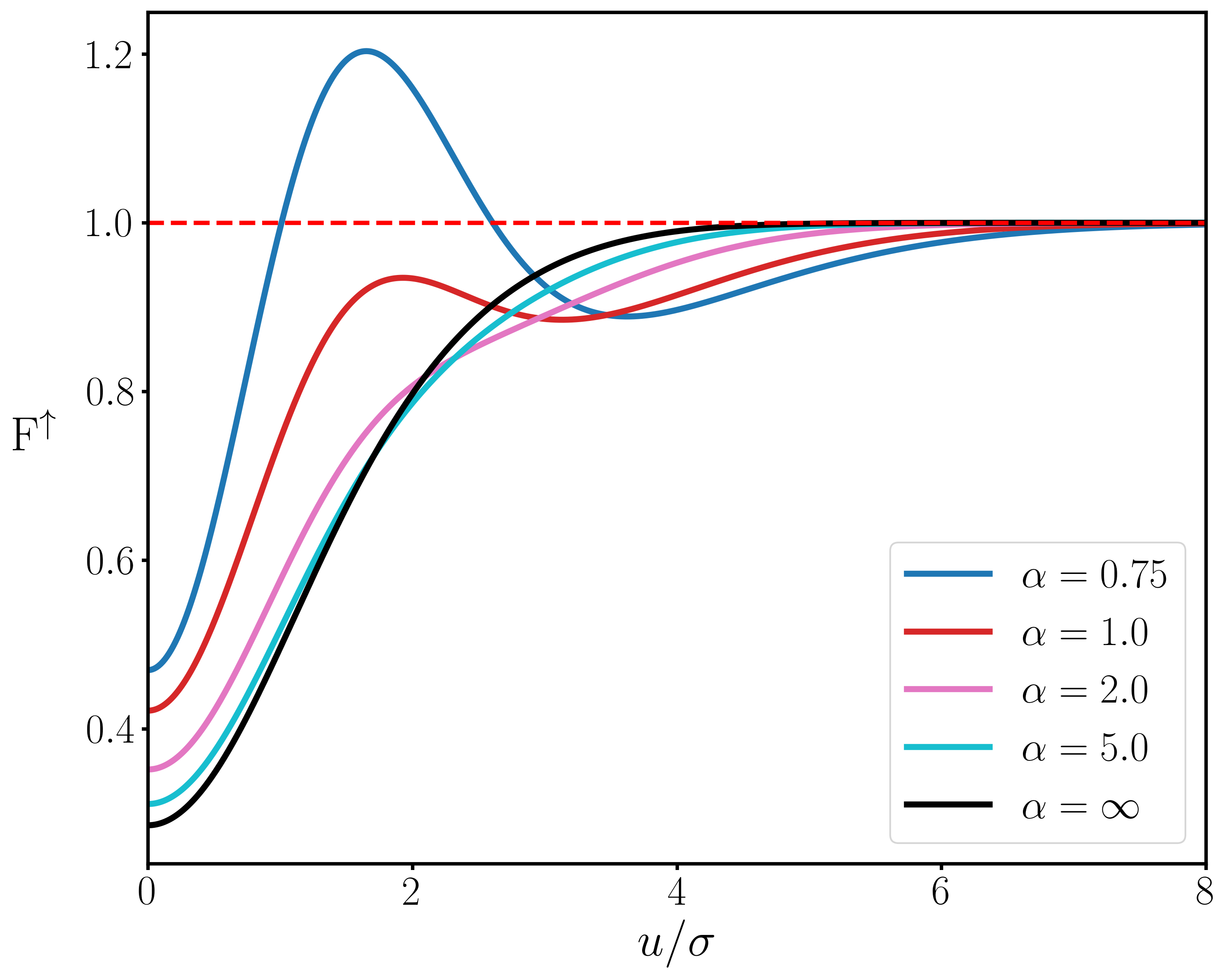} 
    \caption[Fano factor for a stationary Gaussian process with a rational-quadratic correlation function.]{\textbf{Fano factor for a stationary Gaussian process with a rational-quadratic correlation function.}
    We plot the upcrossing Fano factor $\FFup$ as a function of the normalized threshold level $u/\sigma$ for a stochastic process with the 
    rational-quadratic correlation function, $r(t) = \sigma^2 \, \left(1 + \frac{1}{2\alpha}\left(\frac{t}{\tau}\right)^2\right)^{-\alpha}$. Each colored line corresponds to a different value of the shape parameter $\alpha$. The black line represents the $\alpha \to \infty$ limit, where the correlation function simplifies to a squared exponential form, $r(t) = \sigma^2 \exp\left(-t^2/(2\tau^2)\right)$. As the Fano factor is evaluated in the infinite time limit, the results presented are independent of the timescale parameter $\tau$. 
    The behavior of $\FFup$ varies markedly with $\alpha$. For smaller values (e.g., $\alpha= 0.75$), $\FFup$ can exceed 1, indicating a regime of super-Poissonian statistics for certain $u/\sigma$. Conversely, for larger $\alpha$ values $\FFup$ remains sub-Poissonian ($\FFup < 1$) for all $u/\sigma$. For all $\alpha$, the Fano factor approaches the Poisson limit ($\FFup \to 1$) as $u/\sigma \rightarrow \infty$.
    }
    \label{fig:rational_quadratic}
\end{figure}

We consider a stationary Gaussian process whose covariance follows the rational-quadratic kernel, a popular choice in Gaussian-process regression \cite{williams2006gaussian}. This kernel can be written as a continuous mixture of squared-exponential kernels with different characteristic timescales, enabling it to capture variability that occurs over a wide range of timescales. For such a process $X_t$, its rational-quadratic correlation function $r(t)$ is given by:
\begin{equation}
    r(t) = \sigma^2 \, \left[1 + \frac{1}{2\alpha}\left(\frac{t}{\tau}\right)^2\right]^{-\alpha}
\end{equation}
where $\sigma^2 = r(0)$ is the variance of the process, $\tau>0$ is a characteristic timescale parameter, and $\alpha>0$ is a shape parameter that dictates the long-time behavior and smoothness of the correlation. As $\alpha\rightarrow\infty$, the rational-quadratic correlation function approaches the squared exponential correlation function: $r(t) = \sigma^2 \exp\left(-t^2/(2\tau^2)\right)$.

In Fig.~\ref{fig:rational_quadratic}, we plot the upcrossing Fano factor $\FFup$ as a function of the normalized threshold $u/\sigma$ for several values of $\alpha$. Since the Fano factor is evaluated in the long-time limit, the results are independent of $\tau$, depending only on $\alpha$ and $u/\sigma$. The shape parameter $\alpha$ critically influences the crossing statistics: smaller $\alpha$ (corresponding to heavier-tailed, longer-range correlations) leads to super-Poissonian Fano factors, while larger $\alpha$ (faster-decaying correlations) yields sub-Poissonian statistics. For very high thresholds, the Fano factor approaches unity for all $\alpha$, consistent with the 
well-known result that rare threshold crossings 
become Poisson-distributed.

\section{Discussion}

In this work, we presented exact analytical expressions for the variance and Fano factor of \textit{arbitrary} level crossings in smooth, stationary Gaussian processes. Our formulae, encapsulated in Theorems~\ref{thrm:main_theorem_upcrossing} and  \ref{thrm:main_theorem_crossing}, provide a single-integral formulation for the variance, analogous to the well-established result for mean-level upcrossings \cite{steinberg1955short}, but now generalized to any level $u$. This was achieved by explicitly solving the requisite asymmetric Gaussian integrals, expressed in terms of known special functions: the error function and Owen's $T$ function.

Our analysis highlights that the mean crossing rate and the crossing variability capture distinct properties of the underlying stochastic process. While the Kac-Rice mean depends only on the local properties of the autocorrelation function at the origin ($r(0)$ and $r^{\prime\prime}(0)$), the variance and Fano factor probe the entire correlation structure at all lag times. This distinction is vividly illustrated by the stochastic damped harmonic oscillator, where the mean upcrossing rate is completely independent of the damping ratio $\zeta$, yet the Fano factor varies dramatically, from strongly sub-Poissonian in the underdamped regime to super-Poissonian in the overdamped regime. The physical mechanism underlying this transition is transparent: oscillatory correlations impose regularity on crossing times (anti-bunching), while monotonically decaying correlations enable clustering (bunching). 
This sub-to-super-Poissonian transition appears to be a 
general feature of processes whose correlation function 
transitions from oscillatory to monotone decay, as further 
illustrated by the mathematical equivalence between the 
mean-reverting OU-noise model and the overdamped oscillator. 
The analytical expressions offer versatile tools with several important practical implications. First, they significantly enhance the toolkit for parameter estimation and inference in systems modeled by stationary Gaussian processes. While the Kac-Rice formula for the mean number of crossings is invaluable, the incorporation of exact second-order statistics can lead to more robust and accurate parameter determination, particularly when the mean rate alone is insufficient to discriminate between models, for instance, processes with identical $r(0)$ and $r^{\prime\prime}(0)$ but different $\zeta$ values. The Fano factor, as a single dimensionless summary statistic, is particularly attractive for experimental comparisons: it can be estimated from empirical data without requiring full time-series analysis and compared directly with our analytical predictions. Second, our results provide a solid foundation for developing and validating tighter upper and lower bounds on the variance of level crossings, extending previous work on mean-crossing statistics \cite{assaf2023asymptotic} to arbitrary levels.

In neuroscience, the Fano factor of spike trains is a standard metric for distinguishing regular from bursty firing patterns \cite{badel2011firing, gowers2023upcrossing}; our theory provides exact predictions for Gaussian-modeled membrane potential fluctuations. In structural and reliability engineering, the variance of threshold crossings governs fatigue life predictions, where the mean crossing rate alone can significantly underestimate failure risk. In environmental and climate science, the clustering of extreme events (super-Poissonian crossings of temperature or precipitation thresholds) has direct implications for risk assessment and policy.

Our analytical framework currently assumes that the underlying process is stationary, Gaussian, and smooth.
Extending these exact results to non-stationary or non-Gaussian regimes would require higher-order spectral information beyond the autocorrelation function.
Future work will be directed toward developing parameter estimation techniques that directly incorporate the analytical variance of level crossings derived herein, as well as exploring extensions to higher-order crossing moments and to multivariate Gaussian processes.

\section*{Code Availability}
The code used to produce the numerical simulations and figures in this work is publicly available at \url{https://github.com/shivangrawat/gaussian-crossings}.

\bibliographystyle{IEEEtran}
\bibliography{references}

\begin{thebibliography}{10}
\providecommand{\url}[1]{#1}
\csname url@samestyle\endcsname
\providecommand{\newblock}{\relax}
\providecommand{\bibinfo}[2]{#2}
\providecommand{\BIBentrySTDinterwordspacing}{\spaceskip=0pt\relax}
\providecommand{\BIBentryALTinterwordstretchfactor}{4}
\providecommand{\BIBentryALTinterwordspacing}{\spaceskip=\fontdimen2\font plus
\BIBentryALTinterwordstretchfactor\fontdimen3\font minus
  \fontdimen4\font\relax}
\providecommand{\BIBforeignlanguage}[2]{{%
\expandafter\ifx\csname l@#1\endcsname\relax
\typeout{** WARNING: IEEEtran.bst: No hyphenation pattern has been}%
\typeout{** loaded for the language `#1'. Using the pattern for}%
\typeout{** the default language instead.}%
\else
\language=\csname l@#1\endcsname
\fi
#2}}
\providecommand{\BIBdecl}{\relax}
\BIBdecl

\bibitem{verechtchaguina2006first}
T.~Verechtchaguina, I.~M. Sokolov, and L.~Schimansky-Geier, ``First passage
  time densities in resonate-and-fire models,'' \emph{Physical Review
  E—Statistical, Nonlinear, and Soft Matter Physics}, vol.~73, no.~3, p.
  031108, 2006.

\bibitem{tchumatchenko2010correlations}
T.~Tchumatchenko, A.~Malyshev, T.~Geisel, M.~Volgushev, and F.~Wolf,
  ``Correlations and synchrony in threshold neuron models,'' \emph{Physical
  review letters}, vol. 104, no.~5, p. 058102, 2010.

\bibitem{badel2011firing}
L.~Badel, ``Firing statistics and correlations in spiking neurons: A
  level-crossing approach,'' \emph{Nature Precedings}, pp. 1--1, 2011.

\bibitem{gowers2023upcrossing}
R.~P. Gowers and M.~J. Richardson, ``Upcrossing-rate dynamics for a minimal
  neuron model receiving spatially distributed synaptic drive,'' \emph{Physical
  Review Research}, vol.~5, no.~2, p. 023095, 2023.

\bibitem{lutes2004random}
L.~D. Lutes and S.~Sarkani, \emph{Random vibrations: analysis of structural and
  mechanical systems}.\hskip 1em plus 0.5em minus 0.4em\relax Elsevier, 2004.

\bibitem{jafari2006level}
G.~Jafari, M.~S. Movahed, S.~Fazeli, M.~R.~R. Tabar, and S.~Masoudi, ``Level
  crossing analysis of the stock markets,'' \emph{Journal of Statistical
  Mechanics: Theory and Experiment}, vol. 2006, no.~06, p. P06008, 2006.

\bibitem{shayeganfar2012level}
F.~Shayeganfar, M.~H{\"o}lling, J.~Peinke, and M.~R.~R. Tabar, ``The level
  crossing and inverse statistic analysis of german stock market index (dax)
  and daily oil price time series,'' \emph{Physica A: Statistical Mechanics and
  its Applications}, vol. 391, no. 1-2, pp. 209--216, 2012.

\bibitem{kac1943}
M.~Kac, ``{On the average number of real roots of a random algebraic
  equation},'' \emph{Bulletin of the American Mathematical Society}, vol.~49,
  no.~4, pp. 314 -- 320, 1943.

\bibitem{rice1944mathematical}
S.~O. Rice, ``Mathematical analysis of random noise,'' \emph{The Bell System
  Technical Journal}, vol.~23, no.~3, pp. 282--332, 1944.

\bibitem{Kratz2006_review}
\BIBentryALTinterwordspacing
M.~F. Kratz, ``Level crossings and other level functionals of stationary
  gaussian processes,'' \emph{Probability Surveys}, vol.~3, pp. 230--288, 2006.
  [Online]. Available: \url{http://eudml.org/doc/224029}
\BIBentrySTDinterwordspacing

\bibitem{steinberg1955short}
H.~Steinberg, P.~M. Schultheiss, C.~A. Wogrin, and F.~Zweig, ``Short-time
  frequency measurement of narrow-band random signals by means of a zero
  counting process,'' \emph{Journal of Applied Physics}, vol.~26, no.~2, pp.
  195--201, 1955.

\bibitem{azais1999bounds}
J.-M. Aza{\"\i}s, C.~Cierco-Ayrolles, and A.~Croquette, ``Bounds and asymptotic
  expansions for the distributionof the maximum of a smooth stationary gaussian
  process,'' \emph{ESAIM: Probability and Statistics}, vol.~3, pp. 107--129,
  1999.

\bibitem{cierco2003computing}
C.~Cierco-Ayrolles, A.~Croquette, and C.~Delmas, ``Computing the distribution
  of the maximum of gaussian random processes,'' \emph{Methodology and
  Computing in Applied Probability}, vol.~5, pp. 427--438, 2003.

\bibitem{ito1963expected}
K.~Ito, ``The expected number of zeros of continuous stationary gaussian
  processes,'' \emph{Journal of Mathematics of Kyoto University}, vol.~3,
  no.~2, pp. 207--216, 1963.

\bibitem{ylvisaker1965expected}
N.~D. Ylvisaker, ``The expected number of zeros of a stationary gaussian
  process,'' \emph{The Annals of Mathematical Statistics}, vol.~36, no.~3, pp.
  1043--1046, 1965.

\bibitem{geman1972variance}
D.~Geman, ``On the variance of the number of zeros of a stationary gaussian
  process,'' \emph{The Annals of Mathematical Statistics}, pp. 977--982, 1972.

\bibitem{kratz2006}
\BIBentryALTinterwordspacing
M.~F. Kratz and J.~R. Le{\'o}n, ``{On the second moment of the number of
  crossings by a stationary Gaussian process},'' \emph{The Annals of
  Probability}, vol.~34, no.~4, pp. 1601 -- 1607, 2006. [Online]. Available:
  \url{https://doi.org/10.1214/009117906000000142}
\BIBentrySTDinterwordspacing

\bibitem{cramer_leadbetter_book}
H.~Cram\'er and M.~R. Leadbetter, \emph{Stationary and related stochastic
  processes. {S}ample function properties and their applications}.\hskip 1em
  plus 0.5em minus 0.4em\relax John Wiley \& Sons, Inc., New
  York-London-Sydney, 1967.

\bibitem{piterbarg1996asymptotic}
V.~I. Piterbarg, \emph{Asymptotic methods in the theory of Gaussian processes
  and fields}.\hskip 1em plus 0.5em minus 0.4em\relax American Mathematical
  Soc., 1996, vol. 148.

\bibitem{owen1980table}
D.~B. Owen, ``A table of normal integrals,'' \emph{Communications in
  Statistics-Simulation and Computation}, vol.~9, no.~4, pp. 389--419, 1980.

\bibitem{leadbetter1965variance}
M.~R. Leadbetter and J.~D. Cryer, ``{The variance of the number of zeros of a
  stationary normal process},'' \emph{Bulletin of the American Mathematical
  Society}, vol.~71, no. 3.P1, pp. 561 -- 563, 1965.

\bibitem{pawula1968analysis}
R.~Pawula, ``Analysis of an estimator of the center frequency of a power
  spectrum,'' \emph{IEEE Transactions on Information Theory}, vol.~14, no.~5,
  pp. 669--676, 1968.

\bibitem{lindgren1974spectral}
G.~Lindgren, ``Spectral moment estimation by means of level crossings,''
  \emph{Biometrika}, vol.~61, no.~2, pp. 401--418, 1974.

\bibitem{wilson2017periodicity}
L.~R. Wilson and K.~I. Hopcraft, ``Periodicity in the autocorrelation function
  as a mechanism for regularly occurring zero crossings or extreme values of a
  gaussian process,'' \emph{Physical Review E}, vol.~96, no.~6, p. 062129,
  2017.

\bibitem{doob1942brownian}
J.~L. Doob, ``The brownian movement and stochastic equations,'' \emph{Annals of
  Mathematics}, vol.~43, no.~2, pp. 351--369, 1942.

\bibitem{west1982linear}
B.~J. West and V.~Seshadri, ``Linear systems with l{\'e}vy fluctuations,''
  \emph{Physica A: Statistical Mechanics and its Applications}, vol. 113, no.
  1-2, pp. 203--216, 1982.

\bibitem{chechkin2000linear}
A.~Chechkin and V.~Y. Gonchar, ``Linear relaxation processes governed by
  fractional symmetric kinetic equations,'' \emph{Journal of Experimental and
  Theoretical Physics}, vol.~91, pp. 635--651, 2000.

\bibitem{lin2011undamped}
N.~Lin and S.~Lototsky, ``Undamped harmonic oscillator driven by additive
  gaussian white noise: a statistical analysis,'' \emph{Communications on
  Stochastic Analysis}, vol.~5, no.~1, p.~13, 2011.

\bibitem{dybiec2017underdamped}
B.~Dybiec, E.~Gudowska-Nowak, and I.~M. Sokolov, ``Underdamped stochastic
  harmonic oscillator driven by l{\'e}vy noise,'' \emph{Physical Review E},
  vol.~96, no.~4, p. 042118, 2017.

\bibitem{masoliver2023counting}
J.~Masoliver and M.~Palassini, ``Counting of level crossings for inertial
  random processes: Generalization of the rice formula,'' \emph{Physical Review
  E}, vol. 107, no.~2, p. 024111, 2023.

\bibitem{masoliver2025level}
J.~Masoliver, ``Level-crossing counting for generalized langevin equations,''
  \emph{Physical Review E}, vol. 111, no.~1, p. 014113, 2025.

\bibitem{rawat2024element}
S.~Rawat and S.~Martiniani, ``Element-wise and recursive solutions for the
  power spectral density of biological stochastic dynamical systems at fixed
  points,'' \emph{Physical Review Research}, vol.~6, no.~4, p. 043179, 2024.

\bibitem{moreno2002response}
R.~Moreno, J.~de~La~Rocha, A.~Renart, and N.~Parga, ``Response of spiking
  neurons to correlated inputs,'' \emph{Physical Review Letters}, vol.~89,
  no.~28, p. 288101, 2002.

\bibitem{williams2006gaussian}
C.~K. Williams and C.~E. Rasmussen, \emph{Gaussian processes for machine
  learning}.\hskip 1em plus 0.5em minus 0.4em\relax MIT press Cambridge, MA,
  2006, vol.~2, no.~3.

\bibitem{assaf2023asymptotic}
E.~Assaf, J.~Buckley, and N.~Feldheim, ``An asymptotic formula for the variance
  of the number of zeroes of a stationary gaussian process,'' \emph{Probability
  Theory and Related Fields}, vol. 187, no.~3, pp. 999--1036, 2023.

\end{thebibliography}

\clearpage
\onecolumn

\appendices
\numberwithin{equation}{section}
\allowdisplaybreaks

\section{Theory for the Upcrossing Process} 
\label{section:known_theory}
We assume that all the properties of the autocorrelation function $r(t)$ described in Section~\ref{section:description_of_problem} are followed. The upcrossing counting process $\Nup$ can be written as the integral of the underlying generator stochastic process $X_t$,
\begin{equation}
    \Nup = \int_0^T \delta\left(X_{t_1} - u\right) \, \mathrm{H}\left(\dot{X}_{t_1}\right) \dd X_{t_1}
\end{equation}
where $\dot{X}_t$ is the derivative process, $\delta(x)$ is the Dirac delta function, and $\mathrm{H}(x)$ is the Heaviside step function. The condition $\delta\left(X_{t_1} - u\right)$ ensures that only the crossings of level $u$ are counted, and the condition $\mathrm{H}(\dot{X}_{t_1})$ ensures that only upcrossings are counted. Writing the expression above as an integral in time, we get,
\begin{equation}
    \Nup = \int_0^T \delta\left(X_{t_1} - u\right) \, \mathrm{H}\left(\dot{X}_{t_1}\right) \dot{X}_{t_1} \, \dd t_1 \label{eq:defn_Nu_delta}
\end{equation} 

\subsection{Mean Number of Upcrossings} 
\label{section:appendix_upcrossing_mean}
We take the expectation of $\Nup$ to get the mean number of upcrossings of $X_t$ in time $T$,
\begin{equation}
\begin{split}
    \E\left[\Nup\right] &= \E\left[ \int_0^T \delta\left(X_{t_1} - u\right) \, \mathrm{H}\left(\dot{X}_{t_1}\right) \dot{X}_{t_1} \, \dd t_1 \right]  \\
    &= \int_0^T \E\left[\delta\left(X_{t_1} - u\right) \, \mathrm{H}\left(\dot{X}_{t_1}\right) \dot{X}_{t_1} \right] \, \dd t_1 
\end{split} \label{eq:rice_formula_expectation}
\end{equation}
We define the random vector $\mathbf{P} = [X_{t_1}, \, \dot{X}_{t_1}]^\top$ with probability density function (p.d.f.) given by $f_{\mathbf{P}}(x,\,y)$. Since $X_{t}$ is a stationary process, the distribution of $\mathbf{P}$ does not depend on $t_1$, implying that we can write $\mathbf{P} = [X_{0}, \, \dot{X}_{0}]^\top$. Therefore, $f_{\mathbf{P}}(u,\,y)$ is independent of time $t_1$ and we can rewrite 
Eq.~\eqref{eq:rice_formula_expectation} as,
\begin{equation}
\begin{split}
    \E\left[\Nup\right] &= \int_0^T \E\left[\delta\left(X_{t_1} - u\right) \, \mathrm{H}\left(\dot{X}_{t_1}\right) \dot{X}_{t_1} \right] \, \dd t_1 \\
    &= \int_0^T \int_{-\infty}^\infty \int_{-\infty}^\infty \delta\left(x - u\right) \, \mathrm{H}\left(y \right) y\, f_{\mathbf{P}}(x,\,y) \, \dd x \, \dd y \, \dd t_1 \\
    &= \int_0^T \int_{-\infty}^\infty  \mathrm{H}\left(y \right) y\, f_{\mathbf{P}}(u,\,y) \, \dd y \, \dd t_1 \\
    &= \int_0^T \int_{0}^\infty y\, f_{\mathbf{P}}(u,\,y) \, \dd y \, \dd t_1\, .
\end{split} \label{eq:rice_formula_defn_double_integral}
\end{equation}
Since $f_{\mathbf{P}}(u,\,y)$ does not depend on time, $t_1$, we have,
\begin{equation}
    \E\left[\Nup\right] = T \int\limits_{0}^\infty y\, f_{\mathbf{P}}(u,\,y) \, \dd y  \label{eq:rice_formula_boxed}
\end{equation}

If we further assume that $X_{t}$ is zero-mean and Gaussian distributed, this implies that $\mathbf{P}$ is a zero-mean multivariate Gaussian stationary process with the correlation matrix defined as,

\begin{equation}
\renewcommand{\arraystretch}{1.5}
    \boldsymbol{\Sigma}_{\mathbf{P}} = \begin{bmatrix}
                \E[X_{0} X_{0}] & \E[X_{0} \dot{X}_{0}] \\
                \E[\dot{X}_{0} X_{0}] & \E[\dot{X}_{0} \dot{X}_{0}]
                \end{bmatrix} = \begin{bmatrix}
                r(0) & 0 \\
                0 & -\rpp(0)
                \end{bmatrix} 
\end{equation}
We have calculated the entries of this matrix as follows:
\begin{itemize}
    \item $\E\left[X_{0} X_{0}\right] = r(0)$ by definition.
    \item $\E[X_{0} \dot{X}_{0}] = \E[X_{t_1} \dot{X}_{t_2}]_{t_2=t_1}$. Taking the derivative outside of the expectation, we get, \\ $\E[X_{0} \dot{X}_{0}] = \frac{\partial \E[X_{t_1} X_{t_2}]}{\partial t_2}\bigr|_{t_2=t_1} = \frac{\partial r(t_2-t_1)}{\partial t_2}\bigr|_{t_2=t_1} = r^\prime(t_2-t_1)|_{t_2=t_1}=r^\prime(0)=0$.
    \item $\E\left[\dot{X}_{0} X_{0}\right] = \E[X_{0} \dot{X}_{0}] = 0$.
    \item $\E[\dot{X}_{0} \dot{X}_{0}] = \E[\dot{X}_{t_1} \dot{X}_{t_2}]_{t_2=t_1}$. Taking the derivatives outside of the expectation, we get, \\ $\E[\dot{X}_{0} \dot{X}_{0}] = \frac{\partial^2 \E[X_{t_1} X_{t_2}]}{\partial t_1\partial t_2}\bigr|_{t_2=t_1} = \frac{\partial^2 r(t_2-t_1)}{\partial t_1\partial t_2}\bigr|_{t_2=t_1} = -r^{\prime\prime}(t_2-t_1)|_{t_2=t_1} = -r^{\prime\prime}(0)$.
\end{itemize}

Using this covariance matrix, we can define the p.d.f. $f_{\mathbf{P}}(x,y)$ as,
\begin{equation}
\begin{split}
    f_{\mathbf{P}}(x,y) &= \frac{1}{2\pi \sqrt{\left|\boldsymbol{\Sigma}_{\mathbf{P}}\right|}}\exp\left(-\frac{1}{2}[x, y] \, \boldsymbol{\Sigma}_{\mathbf{P}}^{-1} \, [x, y]^\top\right) \\
    &= \frac{1}{2\pi \sqrt{r(0) \, (-\rpp(0))}}\exp\left(-\frac{1}{2}\left(\frac{x^2}{r(0)} + \frac{y^2}{(-\rpp(0))}\right)\right)  \,.
\end{split} 
\end{equation}
Substituting the expression for $f_{\mathbf{P}}(x,\,y)$ into the integral in Eq.~\eqref{eq:rice_formula_boxed} yields
\begin{equation}
    \E\left[\Nup\right] = T \, \frac{1}{2\pi}\sqrt{\frac{-r^{\prime\prime}(0)}{r(0)}} \, e^{-\frac{u^2}{2r(0)}}.
\end{equation}

\subsection{Variance of Upcrossings} \label{section:variance_derivation1}
The variance of $\Nup$ is given by,
\begin{equation}
    \Var\left[\Nup\right] = \E\left[\Nup^2\right] - \E\left[\Nup\right]^2
\end{equation}
$\E\left[\Nup\right]$ is known from the previous section, therefore, we focus on $\E\left[\Nup^2\right]$. Consider the counting process $\Nup$ as the integral over time of the differential process, i.e., 
\begin{equation}
    \Nup = \int_0^T \dd N^\uparrow_u(t_1)
\end{equation}
Therefore, we can write $\E\left[\Nup^2\right]$ as,
\begin{equation}
\begin{split}
    \E\left[\Nup^2\right] &= \E\left[\int_0^T \dd {N^\uparrow_u}(t_1) \int_0^T \dd {N^\uparrow_u}(t_2)\right] \\
    &= \E\left[\int_0^T \int_0^T \dd {N^\uparrow_u}(t_1)  \, \dd {N^\uparrow_u}(t_2)\right] \\
    &= \E\left[\int_0^T \int_0^T \mathbf{1}_{t_1=t_2} \, \dd  {N^\uparrow_u}(t_1) \, \dd {N^\uparrow_u}(t_2) + \int_0^T \int_0^T \mathbf{1}_{t_1\neq t_2}  \,\dd  {N^\uparrow_u}(t_1) \, \dd {N^\uparrow_u}(t_2)\right] \\
    &= \E\left[\int_0^T \dd  {N^\uparrow_u}(t_1)^2 + \int_0^T \int_0^T \mathbf{1}_{t_1\neq t_2}  \, \dd {N^\uparrow_u}(t_1)  \, \dd {N^\uparrow_u}(t_2)\right]
\end{split}
\end{equation}
Since there can only be zero or one crossing at any given time, $\dd  {N^\uparrow_u}(t_1)$ is either $0$ or $1$, so $\dd  {N^\uparrow_u}(t_1)^2 = \dd  {N^\uparrow_u}(t_1)$. Substituting this into the expression above, we get,
\begin{equation}
\begin{split}
    \E\left[\Nup^2\right] &= \E\left[\int_0^T \dd  {N^\uparrow_u}(t_1)\right] + \E\left[\int_0^T \int_0^T \mathbf{1}_{t_1\neq t_2}  \,\dd {N^\uparrow_u}(t_1)  \, \dd {N^\uparrow_u}(t_2)\right] \\
    &= \E\left[\Nup\right] + \underbrace{\E\left[\int_0^T \int_0^T \mathbf{1}_{t_1\neq t_2}  \,\dd {N^\uparrow_u}(t_1)  \, \dd {N^\uparrow_u}(t_2)\right]}_{\text{\normalsize \(M_2\)}}
\end{split}
\end{equation}
Hence, $\E\left[\Nup^2\right] = \E\left[\Nup\right] + M_2$, where $M_2$ is the second factorial moment of $\Nup$. Now, we use the definition of $\Nup$ from Eq.~\eqref{eq:defn_Nu_delta} to get the following expression for $M_2$,
\begin{equation}
    M_2 = \int_0^T \int_0^T \E\left[\mathbf{1}_{t_1\neq t_2} \, \mathrm{H}\left(\dot{X}_{t_1}\right) \mathrm{H}\left(\dot{X}_{t_2}\right)
    \delta\left(X_{t_1} - u\right) \, \delta\left(X_{t_2} - u\right) \, \dot{X}_{t_1} \, \dot{X}_{t_2} \right] \dd t_1 \, \dd t_2 \label{eq:I_rice_variance}
\end{equation}
We define the random vector $\mathbf{Q} = [X_{t_1}, \, X_{t_2}, \, \dot{X}_{t_1}, \, \dot{X}_{t_2}]^\top$ with p.d.f. given by $f_{\mathbf{Q}}(x_1, \,x_2, \,y_1, \,y_2; \,t_1, \,t_2)$. Since $X_{t}$ is a stationary process, $\mathbf{Q}$ only depends on $t_1$ and $t_2$ through their difference $\left(t_2-t_1\right)$ or $f_{\mathbf{Q}}(x_1, \,x_2, \,y_1, \,y_2;\,t_2 - t_1)$. Further, since $X_{t}$ is zero-mean and Gaussian distributed, $\mathbf{Q}$ is a zero-mean multivariate Gaussian stationary process with the correlation matrix defined as,
\begin{equation}
\renewcommand{\arraystretch}{1.5}
    \boldsymbol{\Sigma}_{\mathbf{Q}}(t) = \begin{bmatrix}
                \E[X_{t_1} X_{t_1} ] & \E[X_{t_1} X_{t_2} ] & \E[X_{t_1} \dot{X}_{t_1} ] & \E[X_{t_1} \dot{X}_{t_2} ] \\
                \E[X_{t_2} X_{t_1} ] & \E[X_{t_2} X_{t_2} ] & \E[X_{t_2} \dot{X}_{t_1} ] & \E[X_{t_2} \dot{X}_{t_2} ] \\
                \E[\dot{X}_{t_1} X_{t_1} ] & \E[\dot{X}_{t_1} X_{t_2} ] & \E[\dot{X}_{t_1} \dot{X}_{t_1} ] & \E[\dot{X}_{t_1} \dot{X}_{t_2} ] \\
                \E[\dot{X}_{t_2} X_{t_1} ] & \E[\dot{X}_{t_2} X_{t_2} ] & \E[\dot{X}_{t_2} \dot{X}_{t_1} ] & \E[\dot{X}_{t_2} \dot{X}_{t_2} ] \\
                \end{bmatrix} = 
                \begin{bmatrix}
                r(0) & r(t) & 0 & \rp(t) \\
                r(t) & r(0) & -\rp(t) & 0 \\
                0 & -\rp(t) & -\rpp(0) & -\rpp(t) \\
                \rp(t) & 0 & -\rpp(t) & -\rpp(0) \\
                \end{bmatrix}, \label{eq:correlation_mat_Q}
\end{equation}
where $t=t_2-t_1$. The following properties emerge from the fact that $r(t)$ is a symmetric function about the origin: $r(-t) = r(t)$, $\rp(-t) = - \rp(t)$, and $\rpp(-t) = \rpp(t)$. We used these properties to simplify the covariance matrix above. Using the p.d.f., we can write $M_2$ from Eq.~\eqref{eq:I_rice_variance} as,
\begin{equation}
\begin{split}
    M_2 = \int_0^T \int_0^T &\E\left[\mathbf{1}_{t_1\neq t_2} \, \mathrm{H}\left(\dot{X}_{t_1}\right) \mathrm{H}\left(\dot{X}_{t_2}\right) \delta\left(X_{t_1} - u\right) \, \delta\left(X_{t_2} - u\right) \, \dot{X}_{t_1} \, \dot{X}_{t_2} \right] \dd t_1 \, \dd t_2 \\
    = \int_0^T \int_0^T &\int_{-\infty}^\infty \int_{-\infty}^\infty \int_{-\infty}^\infty \int_{-\infty}^\infty  \biggl(\mathbf{1}_{t_1\neq t_2} \, \mathrm{H}\left(y_1\right) \mathrm{H}\left(y_2\right) \delta\left(x_1 - u\right) \, \delta\left(x_2 - u\right) \, \\& \qquad y_1 y_2 \, f_{\mathbf{Q}}(x_1, \,x_2, \,y_1, \,y_2; \, t_2 - t_1) \biggr)\, \dd x_1 \, \dd x_2 \, \dd y_1 \, \dd y_2 \, \dd t_1 \, \dd t_2 \\
    = \int_0^T \int_0^T &\int_{-\infty}^\infty \int_{-\infty}^\infty \mathbf{1}_{t_1\neq t_2} \, \mathrm{H}\left(y_1\right) \mathrm{H}\left(y_2\right) y_1 y_2 \, f_{\mathbf{Q}}(u, \,u, \,y_1, \,y_2; \, t_2-t_1) \, \dd y_1 \, \dd y_2 \, \dd t_1 \, \dd t_2 \\
    = \int_0^T \int_0^T &\int_{0}^\infty \int_{0}^\infty y_1 y_2 \, f_{\mathbf{Q}}(u, \,u, \,y_1, \,y_2; \, t_2 - t_1) \, \dd y_1 \, \dd y_2 \, \dd t_1 \, \dd t_2
\end{split}
\end{equation}
\begin{figure}[t]
    \centering
    \includegraphics[width=0.8\linewidth]{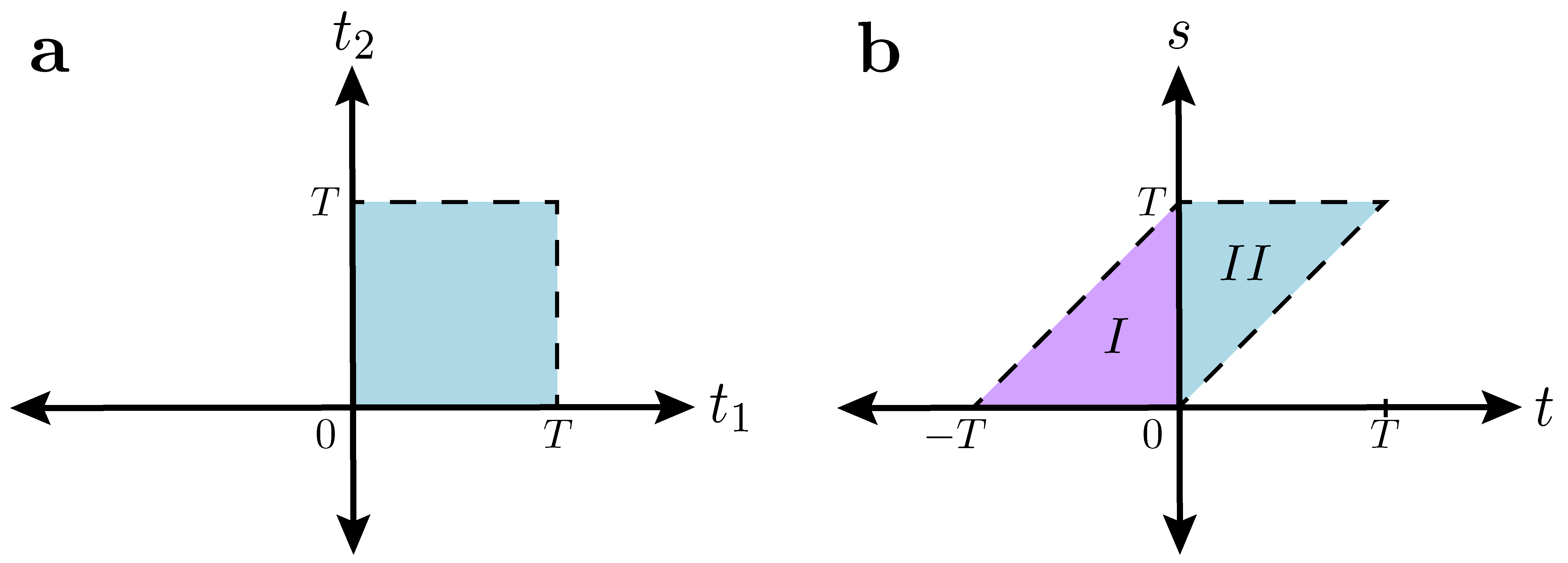}
    \caption[Variable transform visualized.]{\textbf{Variable transform visualized.} (a) The area under which the integral is performed for $t_1$ and $t_2$. (b) The area under which the integral should be performed upon variable transformation given by $t=t_2 -t_1$ and $s = t_2$. I and II represent further partitioning of this region.}
    \label{fig:integral_variance_sup}
\end{figure}
Note that we neglected the condition $\mathbf{1}_{t_1\neq t_2}$ in the integrand above because the only time-dependent term is the p.d.f., and it has no singularities; therefore, its contribution to the final integral is infinitesimally small. Consider $\E\left[\Nup\right]^2$, using Eq.~\eqref{eq:rice_formula_defn_double_integral}, we get,
\begin{equation}
\begin{split}
    \E\left[\Nup\right]^2 &= \left(\int_0^T \int_{0}^\infty y_1\, f_{\mathbf{P}}(u,\,y_1) \, \dd y_1 \, \dd t_1\right) \, \left(\int_0^T \int_{0}^\infty y_2\, f_{\mathbf{P}}(u,\,y_2) \, \dd y_2 \, \dd t_2\right) \\
    &= \int_0^T \int_0^T \int_{0}^\infty \int_{0}^\infty y_1 y_2 \, f_{\mathbf{P}}(u,\,y_1) \, f_{\mathbf{P}}(u,\,y_2) \, \dd y_1 \, \dd y_2 \, \dd t_1 \, \dd t_2
\end{split}
\end{equation}
The variance of $\Nup$ is given by,
\begin{equation}
\begin{split}
    \Var\left[\Nup\right] =& \E\left[\Nup^2\right] - \E\left[\Nup\right]^2 \\
    =& \E\left[\Nup\right] + M_2 - \E\left[\Nup\right]^2 \\
    =& \E\left[\Nup\right] + \int_0^T \int_0^T \int_{0}^\infty \int_{0}^\infty y_1 y_2 \,\bigl( f_{\mathbf{Q}}(u, \,u, \,y_1, \,y_2; \, t_2 - t_1) - f_{\mathbf{P}}(u,\,y_1) \, f_{\mathbf{P}}(u,\,y_2)\bigr)\, \dd y_1 \, \dd y_2 \, \dd t_1 \, \dd t_2
\end{split}
\raisetag{45pt}
\end{equation}
We define $I^\prime$ as follows:
\begin{equation}
    I^\prime = \int_0^T \int_0^T \int_{0}^\infty \int_{0}^\infty y_1 y_2 \left( f_{\mathbf{Q}}(u, u, y_1, y_2;  t_2 - t_1) - f_{\mathbf{P}}(u, y_1) \, f_{\mathbf{P}}(u, y_2)\right)\, \dd y_1 \, \dd y_2 \, \dd t_1 \, \dd t_2.
\end{equation}

We make the following variable change in $I^\prime$: $t=t_2-t_1$ and $s = t_2$. The integral over $t_1$ and $t_2$ is calculated on the positive quadrant (Fig.~\ref{fig:integral_variance_sup}a), and this transforms to the area in Fig.~\ref{fig:integral_variance_sup}b when done in the new coordinates, $t$ and $s$. We write the new integral separately over the regions I and II (Fig.~\ref{fig:integral_variance_sup}b),
\begin{equation}
\begin{split}
    I^\prime =& \int_{-T}^0 \int_0^{t+T} \int_{0}^\infty \int_{0}^\infty y_1 y_2 \,\left( f_{\mathbf{Q}}(u, \,u, \,y_1, \,y_2; \, t) - f_{\mathbf{P}}(u,\,y_1) \, f_{\mathbf{P}}(u,\,y_2)\right)\, \dd y_1 \, \dd y_2 \, \dd s \, \dd t \\
    &+ \int_{0}^T \int_t^{T} \int_{0}^\infty \int_{0}^\infty y_1 y_2 \,\left( f_{\mathbf{Q}}(u, \,u, \,y_1, \,y_2; \, t) - f_{\mathbf{P}}(u,\,y_1) \, f_{\mathbf{P}}(u,\,y_2)\right)\, \dd y_1 \, \dd y_2 \, \dd s \, \dd t
\end{split}
\end{equation}
Transforming the sign of $t$ in the first integral, we get,
\begin{equation}
\begin{split}
    I^\prime =& \int_{0}^T \int_0^{-t+T} \int_{0}^\infty \int_{0}^\infty y_1 y_2 \,\left( f_{\mathbf{Q}}(u, \,u, \,y_1, \,y_2; \, -t) - f_{\mathbf{P}}(u,\,y_1) \, f_{\mathbf{P}}(u,\,y_2)\right)\, \dd y_1 \, \dd y_2 \, \dd s \, \dd t \\
    &+ \int_{0}^T \int_t^{T} \int_{0}^\infty \int_{0}^\infty y_1 y_2 \,\left( f_{\mathbf{Q}}(u, \,u, \,y_1, \,y_2; \, t) - f_{\mathbf{P}}(u,\,y_1) \, f_{\mathbf{P}}(u,\,y_2)\right)\, \dd y_1 \, \dd y_2 \, \dd s \, \dd t
\end{split}
\end{equation}
Since the integrand does not depend on $s$, we integrate in $s$ and get,
\begin{equation}
\begin{split}
    I^\prime =& \int_{0}^T \int_{0}^\infty \int_{0}^\infty (T-t) \,y_1 y_2 \,\left( f_{\mathbf{Q}}(u, \,u, \,y_1, \,y_2; \, -t) - f_{\mathbf{P}}(u,\,y_1) \, f_{\mathbf{P}}(u,\,y_2)\right)\, \dd y_1 \, \dd y_2 \, \dd t \\
    &+ \int_{0}^T \int_{0}^\infty \int_{0}^\infty (T-t) \, y_1 y_2 \,\left( f_{\mathbf{Q}}(u, \,u, \,y_1, \,y_2; \, t) - f_{\mathbf{P}}(u,\,y_1) \, f_{\mathbf{P}}(u,\,y_2)\right)\, \dd y_1 \, \dd y_2 \, \dd t 
\end{split} \label{eq:sign_minus_t1}
\end{equation}
The two integrals are taken over the same volume and have the same integrands except for the sign of $t$ in the p.d.f. Consider the p.d.f. in the first integrand $f_{\mathbf{Q}}(u, \,u, \,y_1, \,y_2; \, -t)$. Since $\mathbf{Q}$ is a multivariate Gaussian, we have,
\begin{equation}
\begin{split}
    f_{\mathbf{Q}}(u, \,u, \,y_1, \,y_2; \, -t) = \frac{1}{4\pi^2\sqrt{
    \left|\boldsymbol{\Sigma}_{\mathbf{Q}}(-t)\right|}} \exp\left(- \frac{1}{2}[u,u,y_1,y_2]\, {\boldsymbol{\Sigma}_{\mathbf{Q}}(-t)}^{-1}\, [u,u,y_1,y_2]^\top\right)
\end{split}
\end{equation}
Using the definition of $\boldsymbol{\Sigma}_{\mathbf{Q}}(-t)$ in Eq.~\eqref{eq:correlation_mat_Q} and the properties of $r(t)$, $\rp(t)$ and $\rpp(t)$, we have,
\begin{equation}
\renewcommand{\arraystretch}{1.2}
    \boldsymbol{\Sigma}_{\mathbf{Q}}(-t) = \begin{bmatrix}
                r(0) & r(t) & 0 & -\rp(t) \\
                r(t) & r(0) & \rp(t) & 0 \\
                0 & \rp(t) & -\rpp(0) & -\rpp(t) \\
                -\rp(t) & 0 & -\rpp(t) & -\rpp(0) \\
                \end{bmatrix}.
\end{equation}
Notice that $\boldsymbol{\Sigma}_{\mathbf{Q}}(-t)$ is related to $\boldsymbol{\Sigma}_{\mathbf{Q}}(t)$ by just a change of sign on the antidiagonal. It can be easily verified that $\left|\boldsymbol{\Sigma}_{\mathbf{Q}}(-t)\right| = \left|\boldsymbol{\Sigma}_{\mathbf{Q}}(t)\right|$ and $f_{\mathbf{Q}}(u, \,u, \,y_1, \,y_2; \, -t) = f_{\mathbf{Q}}(u, \,u, \,y_2, \,y_1; \, t)$. Therefore, Eq.~\eqref{eq:sign_minus_t1} becomes,
\begin{equation}
\begin{split}
    I^\prime =& \int_{0}^T \int_{0}^\infty \int_{0}^\infty (T-t) \,y_1 y_2 \,\left( f_{\mathbf{Q}}(u, \,u, \,y_2, \,y_1; \, t) - f_{\mathbf{P}}(u,\,y_1) \, f_{\mathbf{P}}(u,\,y_2)\right)\, \dd y_1 \, \dd y_2 \, \dd t \\
    &+ \int_{0}^T \int_{0}^\infty \int_{0}^\infty (T-t) \, y_1 y_2 \,\left( f_{\mathbf{Q}}(u, \,u, \,y_1, \,y_2; \, t) - f_{\mathbf{P}}(u,\,y_1) \, f_{\mathbf{P}}(u,\,y_2)\right)\, \dd y_1 \, \dd y_2 \, \dd t 
\end{split} 
\end{equation}
In the first integral, we interchange the variables $y_1$ and $y_2$ to get,
\begin{equation}
\begin{split}
    \int_{0}^T &\int_{0}^\infty \int_{0}^\infty (T-t) \,y_1 y_2 \,\left( f_{\mathbf{Q}}(u, \,u, \,y_2, \,y_1; \, t) - f_{\mathbf{P}}(u,\,y_1) \, f_{\mathbf{P}}(u,\,y_2)\right)\, \dd y_1 \, \dd y_2 \, \dd t 
    \\&= \int_{0}^T \int_{0}^\infty \int_{0}^\infty (T-t) \, y_1 y_2 \,\left( f_{\mathbf{Q}}(u, \,u, \,y_1, \,y_2; \, t) - f_{\mathbf{P}}(u,\,y_1) \, f_{\mathbf{P}}(u,\,y_2)\right)\, \dd y_1 \, \dd y_2 \, \dd t 
\end{split} 
\end{equation}
Therefore, $I^\prime$ is given by,
\begin{equation}
    I^\prime = \, 2 \,T \int_{0}^T \int_{0}^\infty \int_{0}^\infty \left(1-\frac{t}{T}\right) \, y_1 y_2 \,\left( f_{\mathbf{Q}}(u, \,u, \,y_1, \,y_2; \, t) - f_{\mathbf{P}}(u,\,y_1) \, f_{\mathbf{P}}(u,\,y_2)\right)\, \dd y_1 \, \dd y_2 \, \dd t
\end{equation}
Therefore, the variance of $\Nup$ is given by,
\begin{equation}
\begin{aligned}
    \Var\left[\Nup\right] = \E\left[\Nup\right] + 2\,T \int\limits_{0}^T \int\limits_{0}^\infty \int\limits_{0}^\infty \left(1-\frac{t}{T}\right) y_1 y_2 \bigl( f_{\mathbf{Q}}(u, \,u, \,y_1, \,y_2; \, t) - f_{\mathbf{P}}(u,\,y_1) \, f_{\mathbf{P}}(u,\,y_2)\bigr)\, \dd y_1 \, \dd y_2 \, \dd t .
\end{aligned} \label{eq:variance_formula_triple_integral}
\end{equation}

\subsubsection{Central Limit Theorem for Variance}
We consider the infinite time limit for the variance per unit time, i.e., $\lim_{T \rightarrow \infty}\frac{\Var[\Nup]}{T}$. We can split the integral in Eq.~\eqref{eq:variance_formula_triple_integral} as follows,
\begin{equation}
\begin{split}
    \lim_{T \rightarrow \infty}\frac{\Var\left[\Nup\right]}{T} =& \lim_{T\rightarrow\infty} \frac{\E\left[\Nup\right]}{T} + 2 \int_{0}^\infty \int_{0}^\infty \int_{0}^\infty y_1 y_2 \bigl( f_{\mathbf{Q}}(u, \,u, \,y_1, \,y_2; \, t)  - f_{\mathbf{P}}(u,\,y_1) \, f_{\mathbf{P}}(u,\,y_2)\bigr)\dd y_1 \, \dd y_2 \, \dd t \\
    &- \lim_{T\rightarrow\infty}\frac{2}{T} \underbrace{\int_{0}^\infty \int_{0}^\infty \int_{0}^\infty t\, y_1 y_2 \,\left( f_{\mathbf{Q}}(u, \,u, \,y_1, \,y_2; \, t) - f_{\mathbf{P}}(u,\,y_1) \, f_{\mathbf{P}}(u,\,y_2)\right)\, \dd y_1 \, \dd y_2 \, \dd t}_{\text{\normalsize \(I^{\prime\prime}\)}}
\end{split}
\raisetag{30pt}
\end{equation}
Consider the second integral in the expression above: $2 \lim_{T\rightarrow\infty} I^{\prime\prime} / T$. If $I^{\prime\prime}$ is finite then $2 \lim_{T\rightarrow\infty} I^{\prime\prime} / T = 0$. In fact, if the following condition is satisfied, then $I^{\prime\prime}$ is finite \cite{piterbarg1996asymptotic}, pg. 60.
\begin{equation}
\begin{split}
    \int_0^\infty t \left(\left|r(t)\right| + \left|\rp(t)\right|+ \left|\rpp(t)\right|\right) \, \dd t < \infty\,.
\end{split} \label{eq:pieterbarg_condition}
\end{equation}
Assuming this condition is satisfied, the variance in the long-time limit is given by,
\begin{equation}
\begin{aligned}
    \lim_{T \rightarrow \infty}\frac{\Var\left[\Nup\right]}{T} = \lim_{T \rightarrow \infty}\frac{\E\left[\Nup\right]}{T} + 2 \int\limits_{0}^\infty \int\limits_{0}^\infty \int\limits_{0}^\infty y_1 y_2 \bigl( f_{\mathbf{Q}}(u, \,u, \,y_1, \,y_2; \, t) - f_{\mathbf{P}}(u,\,y_1) \, f_{\mathbf{P}}(u,\,y_2)\bigr) \dd y_1 \, \dd y_2 \, \dd t.
\end{aligned} \label{eq:variance_formula_triple_integral_CLT}
\end{equation}

\subsection{Fano Factor}
The Fano factor is defined as the ratio of variance and mean of $\Nup$ in the infinite time limit, 
\begin{equation}
    \FFup = \lim_{T \to \infty}\frac{\Var[\Nup]}{\E[\Nup]}
\end{equation}
Taking the formula for variance from Eq.~\eqref{eq:variance_formula_triple_integral_CLT}, we get,
\begin{equation}
\begin{split}
    \FFup &= \frac{\E\left[\Nup\right] + 2\, T\int_{0}^\infty \int_{0}^\infty \int_{0}^\infty y_1 y_2 \,\left( f_{\mathbf{Q}}(u, \,u, \,y_1, \,y_2; \, t) - f_{\mathbf{P}}(u,\,y_1) \, f_{\mathbf{P}}(u,\,y_2)\right)\, \dd y_1 \, \dd y_2 \, \dd t}{\E[\Nup]} \\
    &= 1 + \frac{2\,T\int_{0}^\infty \int_{0}^\infty \int_{0}^\infty y_1 y_2 \,\left( f_{\mathbf{Q}}(u, \,u, \,y_1, \,y_2; \, t) - f_{\mathbf{P}}(u,\,y_1) \, f_{\mathbf{P}}(u,\,y_2)\right)\, \dd y_1 \, \dd y_2 \, \dd t}{\E[\Nup]}
\end{split}
\raisetag{40pt}
\end{equation}
Substituting the mean formula from Eq.~\eqref{eq:rice_formula_boxed}, we get the final expression for the Fano factor,
\begin{equation}
    \FFup = 1 + \frac{2\int\limits_{0}^\infty \int\limits_{0}^\infty \int\limits_{0}^\infty y_1 y_2 \,\left( f_{\mathbf{Q}}(u, \,u, \,y_1, \,y_2; \, t) - f_{\mathbf{P}}(u,\,y_1) \, f_{\mathbf{P}}(u,\,y_2)\right)\, \dd y_1 \, \dd y_2 \, \dd t}{\int\limits_{0}^\infty y\, f_{\mathbf{P}}(u,\,y) \, \dd y}. \label{eq:fano_factor_upcrossing_1}
\end{equation}

\subsection{Proof of Theorem~\ref{thrm:main_theorem_upcrossing}} \label{section:proof_theorem_upcrossing}
The result for the mean can be found by using the new notation in Eq.~\eqref{eq:rice_formula_specific_main}. For deriving the expression for variance, we take the result derived for $\Var[\Nup]$, stated in Eq.~\eqref{eq:variance_formula_triple_integral_main}, and explicitly evaluate the integrals in $y_1$ and $y_2$ for an arbitrary level $u$. To the best of our knowledge, this has only been done for zero-level, i.e., $u=0$ \cite{steinberg1955short, leadbetter1965variance, pawula1968analysis, lindgren1974spectral, wilson2017periodicity}, but not for an arbitrary level $u$ due to the complexity of
the asymmetric Gaussian integral involved. We overcome this difficulty by explicitly solving the integral in terms
of known functions.

We denote the integral of interest by $\Iup$, defined as,
\begin{equation}
    \Iup = \int_{0}^\infty \int_{0}^\infty y_1 y_2 \,\bigl( f_{\mathbf{Q}}(u, \,u, \,y_1, \,y_2; \, t) - f_{\mathbf{P}}(u,\,y_1) \, f_{\mathbf{P}}(u,\,y_2)\bigr)\, \dd y_1 \, \dd y_2 \,. \label{eq:defn_of_I}
\end{equation}
Splitting the integral, we get,
\begin{equation}
    \Iup = \underbrace{\int_{0}^\infty \int_{0}^\infty y_1 y_2 \,f_{\mathbf{Q}}(u, \,u, \,y_1, \,y_2; \, t) \, \dd y_1 \, \dd y_2}_{\text{\normalsize \(I_1\)}} - \underbrace{\int_{0}^\infty \int_{0}^\infty y_1 y_2 \, f_{\mathbf{P}}(u,\,y_1) \, f_{\mathbf{P}}(u,\,y_2) \, \dd y_1 \, \dd y_2}_{\text{\normalsize \(I_2\)}} \,.
\label{eq:defn_of_I_split}
\end{equation}
To derive an explicit analytical expression for $\Iup$, we use the shorter notations for quantities depending on the correlation function $r(t)$: $r(t)\rightarrow r$, $r(0)\rightarrow r_0$, $\rp(t) \rightarrow p$, $\rp(0)\rightarrow p_0$, $-\rpp(t) \rightarrow q$, and $-\rpp(0)\rightarrow q_0$. Note that since $r(t)$ is a differentiable function with $r(0)$ as its maximum value (by definition), the graph is concave at $t=0$; therefore, $-\rpp(0)>0$ or $q_0>0$. In this new notation, we can write the covariance matrix for the vector $\mathbf{P} = [X_{t_1}, \, \dot{X}_{t_1}]^\top$, defined in Eq.~\eqref{eq:correlation_mat_P_main} as,
\begin{equation}
\renewcommand{\arraystretch}{0.75}
    \boldsymbol{\Sigma}_{\mathbf{P}} = \begin{bmatrix}
                r_0 & 0 \\
                0 & q_0
                \end{bmatrix}\,.
\end{equation}
Further, since $\mathbf{P}$ is a zero-mean Gaussian random vector, its p.d.f. is given by,
\begin{equation}
\begin{split}
    f_{\mathbf{P}}(x,y) &= \frac{1}{2\pi \sqrt{\left|\boldsymbol{\Sigma}_{\mathbf{P}}\right|}}\exp\left(-\frac{1}{2}[x, y] \, \boldsymbol{\Sigma}_{\mathbf{P}}^{-1} \, [x, y]^\top\right) \\
    &= \frac{1}{2\pi \sqrt{r_0 \, q_0}}\exp\left(-\frac{1}{2}\left(\frac{x^2}{r_0} + \frac{y^2}{q_0}\right)\right)  \,.
\end{split} \label{eq:f_P_proof}
\end{equation}
To calculate $I_2$ (Eq.~\eqref{eq:defn_of_I_split}), we substitute the expressions for $f_{\mathbf{P}}(u,\,y_1)$ and $f_{\mathbf{P}}(u,\,y_2)$ using the equation above:
\begin{equation}
\begin{split}
    I_2 &= \int_{0}^\infty \int_{0}^\infty y_1 y_2 \, f_{\mathbf{P}}(u,y_1) \, f_{\mathbf{P}}(u,y_2) \, \dd y_1 \, \dd y_2 \\
    &= \frac{e^{-u^2/r_0}}{4\pi^2 \, r_0 q_0}\int_{0}^\infty \int_{0}^\infty y_1 y_2  \exp\left(-\frac{y_1^2 + y_2^2}{2q_0}\right)  \dd y_1 \, \dd y_2 \\
    &= \frac{1}{4\pi^2 }\left(\frac{q_0}{r_0}\right)e^{-u^2/r_0} \,.
\end{split} \label{eq:final_expr_for_I2}
\end{equation}
Now we evaluate an explicit expression for $I_1$ (Eq.~\eqref{eq:defn_of_I_split}), given by,
\begin{equation}
    I_1 = \int_{0}^\infty \int_{0}^\infty y_1 y_2 \,f_{\mathbf{Q}}(u, \,u, \,y_1, \,y_2; \, t) \, \dd y_1 \, \dd y_2 \,.
\end{equation}
We first make the following transformation of variables in the integral:
\begin{equation}
    z_1 = \frac{y_2 - y_1}{\sqrt{2}} \quad \text{and} \quad z_2 = \frac{y_2 + y_1}{\sqrt{2}}. \label{eq:transformation_defined}
\end{equation}
We get the following relations between the old and new variables:
\begin{equation}
\begin{split}
    y_1^2 + y_2^2 &= z_1^2 + z_2^2 \,, \\
    y_2 - y_1 &= \sqrt{2} \, z_1 \,, \\
    y_1 y_2 &= \frac{z_2^2 - z_1^2}{2} \,.
\end{split} \label{eq:transformation_relations}
\end{equation}
The absolute value of the Jacobian of this transformation is $1$ and the area of integration changes as follows,
\begin{equation}
    I_1 = \frac{1}{2}\int_{-\infty}^\infty \int_{|z_1|}^\infty \left(z_2^2 - z_1^2\right) f_{\mathbf{Q}}(u, u, z_1, z_2; t) \, \dd z_2 \,\dd z_1 \,. \label{eq:I1_defn_z1_z2}
\end{equation}
Now, we find how the p.d.f. changes due to this transformation of variables, i.e., we find the expression for $f_{\mathbf{Q}}(u, \,u, \,z_1, \,z_2; \, t)$. First, in the new notation, the covariance matrix for the vector $\mathbf{Q} = [X_{t_1}, \, X_{t_2}, \, \dot{X}_{t_1}, \, \dot{X}_{t_2}]^\top$, defined in Eq.~\eqref{eq:correlation_mat_Q_main} becomes,
\begin{equation}
\renewcommand{\arraystretch}{1.2}
    \boldsymbol{\Sigma}_{\mathbf{Q}}(t) = \begin{bmatrix}
                r_0 & r & 0 & p \\
                r & r_0 & -p & 0 \\
                0 & -p & q_0 & q \\
                p & 0 & q & q_0 \\
                \end{bmatrix}, \label{eq:correlation_mat_Q_main_new_notation}
\end{equation}
Further, since $\mathbf{Q}$ is a zero-mean Gaussian random vector, its p.d.f. is given by,
\begin{equation}
\begin{split}
    f_{\mathbf{Q}}(x_1, x_2, y_1, y_2; t) = \frac{1}{4\pi^2\sqrt{
    \left|\boldsymbol{\Sigma}_{\mathbf{Q}}(t)\right|}} \,\exp\left(S\right),
\end{split}
\end{equation}
where
\begin{equation}
    S = - \frac{1}{2}[x_1,x_2,y_1,y_2]\, {\boldsymbol{\Sigma}_{\mathbf{Q}}(t)}^{-1}\, [x_1,x_2,y_1,y_2]^\top \,.
\end{equation}
We denote the determinant $\left|\boldsymbol{\Sigma}_{\mathbf{Q}}(t)\right|$ by $\Delta$. Calculating the determinant of the matrix above, we find,
\begin{equation}
    \Delta = \bigl(p^2 + (q-q_0) (r+r_0)\bigr)\, \bigl(p^2 + (q+q_0) (r-r_0)\bigr) \,.\label{eq:determinant_expr}
\end{equation}
Since we need the expression for $f_{\mathbf{Q}}(u, \,u, \,y_1, \,y_2; \, t)$ to find $I_1$, we focus on the expression in the exponential ($S$), evaluated at $\left(u, \,u, \, y_1,\,y_2\right)$. Finding the inverse of the matrix ${\boldsymbol{\Sigma}_{\mathbf{Q}}(t)}^{-1}$, and upon simplification, we get,
\begin{equation}
    S = -\frac{1}{2\Delta} \bigl(c_0 \, u^2  + c_1  \left(y_1^2 + y_2^2\right) + c_2 \, u \left(y_2-y_1\right) + c_3 \, y_1y_2 \bigr)\,,
\end{equation}
where $c_0$, $c_1$, $c_2$ and $c_3$ are given by,
\begin{equation}
\begin{split}
    c_0 &= 2 \left(q - q_0\right) \left(p^2 + \left(q + q_0\right) \left(r - r_0\right)\right) , \\
    c_1 &=  q_0 r_0^2 - q_0 r^2 - r_0 p^2 , \\
    c_2 &= 2\, p \left(p^2 + \left(q + q_0\right) \left(r - r_0\right)\right) , \\
    c_3 &= 2 \left(p^2 r+q r^2-q r_0^2\right) .
\end{split}
\end{equation}
Transforming $S$ into new variables defined in Eq.~\eqref{eq:transformation_defined}, using the relations stated in Eq.~\eqref{eq:transformation_relations}, we get,
\begin{equation}
\begin{split}
    S =& -\frac{1}{2\Delta} \biggl(c_0 \, u^2  + \left(c_1-\frac{c_3}{2}\right) z_1^2 +  \sqrt{2} \,c_2 \, u \,z_1
     + \left(c_1 + \frac{c_3}{2}\right) z_2^2  \biggr) \\
    =& -\frac{1}{2\Delta} \left(\tilde{c}_0 \, u^2  + \tilde{c}_1\, z_1^2 +  \tilde{c}_2\, u \,z_1 + \tilde{c}_3 \,z_2^2  \right).
\end{split}
\end{equation}
where $\tilde{c}_0$, $\tilde{c}_1$, $\tilde{c}_2$ and $\tilde{c}_3$ are given by,
\begin{equation}
\begin{split}
    \tilde{c}_0 &= 2 \left(q - q_0\right) \left(p^2 + \left(q + q_0\right) \left(r - r_0\right)\right), \\
    \tilde{c}_1 &= - \left(r+r_0\right) \left(p^2 + \left(q + q_0\right) \left(r - r_0\right)\right),\\
    \tilde{c}_2 &=  2\sqrt{2} \,p \left(p^2 + \left(q + q_0\right) \left(r - r_0\right)\right),\\
    \tilde{c}_3 &= - \left(r_0 - r\right) \left(p^2 + \left(q - q_0\right) \left(r + r_0\right)\right) .
\end{split}
\end{equation}
Now we complete the square in $z_1$, which gives us,
\begin{equation}
\begin{split}
S &= -\frac{1}{2\Delta} \left( \left(\tilde{c}_0 - \frac{\tilde{c}_2^2}{4\tilde{c}_1}\right) u^2 + \tilde{c}_1 \left(z_1 + \frac{\tilde{c}_2}{2\tilde{c}_1} u \right)^2  + \tilde{c}_3\,z_2^2 \right) \\
&= - \underbrace{\frac{1}{2\Delta} \left(\tilde{c}_0 - \frac{\tilde{c}_2^2}{4\tilde{c}_1}\right)}_{\text{\normalsize \(\delta\)}} u^2 - \underbrace{\frac{\tilde{c}_1}{2\Delta}}_{\text{\normalsize \(\alpha\)}}  \Biggl(z_1 -  \underbrace{\frac{-\tilde{c}_2}{2\tilde{c}_1} u}_{\text{\normalsize \(\gamma\)}} \Biggr)^2  -  \underbrace{\frac{\tilde{c}_3}{2\Delta}}_{\text{\normalsize \(\beta\)}} \,z_2^2 \,.
\end{split}
\end{equation}
In this new notation, $S$ is given simply by the expression,
\begin{equation}
    S = -\delta \, u^2 - \alpha \left(z_1 - \gamma \right)^2 - \beta \, z_2^2 \,,
\end{equation}
where $\alpha$, $\beta$, $\gamma$, and $\delta$ are given by the following expressions, found using the definition above and the definition of $\Delta$ from Eq.~\eqref{eq:determinant_expr},
\begin{equation}
\begin{aligned}
\alpha &=-\frac{r+r_0}{2\left(p^2+(q-q_0)(r+r_0)\right)},\\
\beta &= -\frac{r_0-r}{2\left(p^2+(q+q_0)(r-r_0)\right)},\\
\gamma &=\frac{\sqrt{2}\,p}{r+r_0}\,u,\\
\delta &= \frac{1}{r+r_0}\,.
\end{aligned}
\label{eq:defn_alpha_beta_gamma_delta}
\end{equation}
Now, we determine the signs of these variables for $t \geq 0$.
\begin{itemize}
    \item $\alpha$ and $\beta$: Strictly positive due to Theorem~\ref{thrm:positivity_of_correlation_term}.
    \item $\gamma:$ The sign of $\gamma$ depends on the signs of $u$ and $p$; therefore, $\gamma$ can be negative or nonnegative.
    \item $\delta:$ Since $r_0 \geq |r| \; \forall \; t \geq 0$ and $r_0 > 0$ by definition, $\implies \; r+r_0 \geq 0 \; \forall \; t \geq 0$. Therefore, $\delta$ is strictly positive for all $t\geq 0$.
\end{itemize}
To summarize, $\alpha > 0$, $\beta>0$, $\delta>0$, and $\gamma \in \R$ for all $t \geq 0$. Now, we express the determinant in terms of $r$, $\alpha$, and $\beta$. It can be verified using the expressions for $\alpha$ and $\beta$ in Eq.~\eqref{eq:defn_alpha_beta_gamma_delta} and the expression for the determinant in Eq.~\eqref{eq:determinant_expr} that,
\begin{equation}
    \Delta = \frac{r_0^2 - r^2}{4 \alpha \beta} \,.
\end{equation}
Therefore, we get the following expression for $f_{\mathbf{Q}}(u, \,u, \,z_1, \,z_2; \, t)$:
\begin{equation}
\begin{split}
    f_{\mathbf{Q}}(u, u, z_1, z_2;  t)  &= \frac{1}{4\pi^2\sqrt{\Delta}}\, \exp\left(S\right) \\
    &= \frac{1}{2\pi^2}\sqrt{\frac{\alpha\beta}{r_0^2-r^2}} \,\exp\left( -\delta \, u^2 - \alpha \left(z_1 - \gamma \right)^2 - \beta \, z_2^2\right) ,
\end{split}  \label{eq:expression_for_f_Q}
\end{equation}
and the integral $I_1$ (Eq.~\eqref{eq:I1_defn_z1_z2}) can be written as,
\begin{equation}
    I_1 = \frac{e^{-\delta\,u^2}}{4\pi^2}\sqrt{\frac{\alpha\beta}{r_0^2-r^2}}\int_{-\infty}^\infty \int_{|z_1|}^\infty \left(z_2^2 - z_1^2\right) \times\exp\left(- \alpha \left(z_1 - \gamma \right)^2 - \beta \, z_2^2\right) \dd z_2 \, \dd z_1 \,.
\end{equation}
Since we have already shown that $\alpha>0$, $\beta>0$, and $\gamma\in \R$, the result of Theorem~\ref{thrm:integral_main} applies; therefore, we get the following expression for $I_1$:
\begin{equation}
\begin{split}
    I_1 = \frac{e^{-\delta\,u^2}}{4\pi^2 \sqrt{r_0^2-r^2}} \Biggl[\frac{e^{-\alpha \gamma^2}}{2\sqrt{\alpha\beta}} &\left(1 + \sqrt{\pi} \gamma \sqrt{\alpha + \beta} e^{\frac{\alpha^2 \gamma^2}{\alpha + \beta}} \Erf\left(\frac{\alpha \gamma}{\sqrt{\alpha + \beta}}\right)\right)  \\& + \pi \left(\frac{\alpha - \beta - 2\alpha\beta\gamma^2}{\alpha\beta}\right) \Ot\left(\gamma\sqrt{\frac{2\alpha\beta}{\alpha+\beta}}, \sqrt{\frac{\alpha}{\beta}}\right)\Biggr] \,. \label{eq:final_expr_for_I1}
\end{split}
\end{equation}
Finally, since $\Iup=I_1-I_2$, using Eq.~\eqref{eq:final_expr_for_I2} and \eqref{eq:final_expr_for_I1} we have,
\begin{equation}
\begin{aligned}
    \Iup = \frac{e^{-\delta\,u^2}}{4\pi^2 \sqrt{r_0^2-r^2}} \Biggl[&\frac{e^{-\alpha \gamma^2}}{2\sqrt{\alpha\beta}} \left(1 + \sqrt{\pi} \, \gamma \sqrt{\alpha + \beta} \, e^{\frac{\alpha^2 \gamma^2}{\alpha + \beta}} \, \Erf\left(\frac{\alpha \gamma}{\sqrt{\alpha + \beta}}\right)\right)  \\& + \pi \left(\frac{\alpha - \beta - 2\alpha\beta\gamma^2}{\alpha\beta}\right) \Ot\left(\gamma\sqrt{\frac{2\alpha\beta}{\alpha+\beta}}, \sqrt{\frac{\alpha}{\beta}}\right)\Biggr] - \frac{1}{4\pi^2 }\left(\frac{q_0}{r_0}\right)e^{-u^2/r_0} \,.
\end{aligned} \label{eq:I_final_result_owensT}
\end{equation}
From Eq.~\eqref{eq:variance_formula_triple_integral_main}, the $\Var\left[\Nup\right]$ depends on $\Iup$ as follows:
\begin{equation}
    \Var\left[\Nup\right] = \E\left[\Nup\right] + 2\,T \int_{0}^T \left(1-\frac{t}{T}\right) \Iup \, \dd t\,, \label{eq:var_in_terms_of_I}
\end{equation}
where $\E\left[\Nup\right]$ is given by the Rice formula derived in Eq.~\eqref{eq:rice_formula_specific_main},
\begin{equation}
    \E\left[\Nup\right] = T\, \frac{1}{2\pi}\sqrt{\frac{q_0}{r_0}} \, e^{-u^2/(2r_0)} \,.
\end{equation}

\section{Theory for the Crossing Process}
\label{section:crossings}
In this section, we derive the formulae for the mean, variance, and Fano factor for the crossing process, i.e., including both \textit{up} and \textit{down} crossings. The steps involved in deriving these formulae are similar to those used for the upcrossing process in Section~\ref{section:known_theory}; therefore, we offer a brief discussion of the proof. The counting process for the total number of crossings can be written as,
\begin{equation}
    \Nu = \int_0^T \delta\left(X_{t_1} - u\right) \left|\dot{X}_{t_1}\right| \, \dd t_1 \label{eq:defn_Nu_delta_crossing}
\end{equation}

\subsection{Mean Number of Crossings}
Taking the expectation of the equation above, we get,
\begin{equation}
\begin{split}
    \E\left[\Nu\right] &= \int_0^T \E\left[\delta\left(X_{t_1} - u\right) \left|\dot{X}_{t_1}\right|  \right] \, \dd t_1 \\
    &= \int_0^T \int_{-\infty}^\infty |y|\, f_{\mathbf{P}}(u,\,y) \, \dd y \, \dd t_1
\end{split} \label{eq:rice_formula_crossings1}
\end{equation}
where $f_{\mathbf{P}}(x,\,y)$ is the p.d.f. of the random vector $\mathbf{P} = [X_{t_1}, \, \dot{X}_{t_1}]^\top$.

\subsection{Variance and Fano Factor of Crossings}
Using an argument similar to that in Section~\ref{section:variance_derivation1} we can write $\E\left[\Nu^2\right]$ as the following integral,
\begin{equation}
\begin{split}
    \E\left[\Nu^2\right] &=  \E\left[\Nu\right] + \underbrace{\E\left[\int_0^T \int_0^T \mathbf{1}_{t_1\neq t_2}  \,\dd {N_u}(t_1)  \, \dd {N_u}(t_2)\right]}_{\text{\normalsize \(M_2\)}}
\end{split}
\end{equation}
Now using Eq.~\eqref{eq:defn_Nu_delta_crossing}, $M_2$ can be written as,
\begin{equation}
\begin{split}
    M_2 &= \int_0^T \int_0^T \E\left[\mathbf{1}_{t_1\neq t_2} \, \delta\left(X_{t_1} - u\right) \, \delta\left(X_{t_2} - u\right)  \left|\dot{X}_{t_1}\right|  \left|\dot{X}_{t_2}\right| \right] \dd t_1 \, \dd t_2 \\
    &= \int_0^T \int_0^T \int_{-\infty}^\infty \int_{-\infty}^\infty \left|y_1 y_2\right| \, f_{\mathbf{Q}}(u, \,u, \,y_1, \,y_2; \, t_2 - t_1) \, \dd y_1 \, \dd y_2 \, \dd t_1 \, \dd t_2
\end{split}
\end{equation}
where $f_{\mathbf{Q}}(u, \,u, \,y_1, \,y_2; \, t_2 - t_1)$ is the p.d.f. of the random vector $\mathbf{Q} = [X_{t_1}, \, X_{t_2}, \, \dot{X}_{t_1}, \, \dot{X}_{t_2}]^\top$. The variance of $\Nu$ is given by,
\begin{equation}
\begin{split}
    \Var\left[\Nu\right] &= \E\left[\Nu^2\right] - \E\left[\Nu\right]^2 \\
    &=  \int_0^T \int_0^T \int_{-\infty}^\infty \int_{-\infty}^\infty \left|y_1 y_2\right| \, f_{\mathbf{Q}}(u, \,u, \,y_1, \,y_2; \, t_2 - t_1) \, \dd y_1 \, \dd y_2 \, \dd t_1 \, \dd t_2  + \E\left[\Nu\right] - \E\left[\Nu\right]^2 
\end{split}
\end{equation}
We use Eq.~\eqref{eq:rice_formula_crossings1} to write $\E\left[\Nu\right]^2$ as a double integral, which gives us,
\begin{equation}
\begin{split}
    \Var\left[\Nu\right] &= \E\left[\Nu^2\right] - \E\left[\Nu\right]^2 \\
    &= \E\left[\Nu\right] + \int_0^T \int_0^T \int_{-\infty}^\infty \int_{-\infty}^\infty \left|y_1 y_2\right| \bigl(f_{\mathbf{Q}}(u, \,u, \,y_1, \,y_2; \, t_2 - t_1) - f_{\mathbf{P}}(u,\,y_1) \, f_{\mathbf{P}}(u,\,y_2)\bigr)\, \dd y_1 \, \dd y_2 \, \dd t_1 \, \dd t_2
\end{split}
\raisetag{40pt}
\end{equation}
We define:
\begin{equation}
    I^\prime = \int_0^T \int_0^T \int_{-\infty}^\infty \int_{-\infty}^\infty \left|y_1 y_2\right| \,\left( f_{\mathbf{Q}}(u, u, y_1, y_2;  t_2 - t_1) - f_{\mathbf{P}}(u,y_1) \, f_{\mathbf{P}}(u,y_2)\right)\, \dd y_1 \, \dd y_2 \, \dd t_1 \, \dd t_2
\end{equation}

Making the variable transformation $t_2-t_1=t$, $t_2=s$ and using the arguments about the integral from Section~\ref{section:variance_derivation1}, we can write $I^\prime$ as,
\begin{equation}
    I^\prime = \, 2 \,T \int_{0}^T \int_{-\infty}^\infty \int_{-\infty}^\infty \left(1-\frac{t}{T}\right) \left| y_1 y_2 \right|\left( f_{\mathbf{Q}}(u, \,u, \,y_1, \,y_2; \, t) - f_{\mathbf{P}}(u,\,y_1) \, f_{\mathbf{P}}(u,\,y_2)\right)\, \dd y_1 \, \dd y_2 \, \dd t
\end{equation}
Therefore, the variance is given by,
\begin{equation}
    \Var\left[\Nu\right] = \E\left[\Nu\right] + 2\,T \int\limits_{0}^T \int\limits_{-\infty}^\infty \int\limits_{-\infty}^\infty \left(1-\frac{t}{T}\right)  \left| y_1 y_2 \right| \bigl(f_{\mathbf{Q}}(u, \,u, \,y_1, \,y_2; \, t) - f_{\mathbf{P}}(u,\,y_1) \, f_{\mathbf{P}}(u,\,y_2)\bigr) \dd y_1 \, \dd y_2 \, \dd t
 \label{eq:variance_formula_triple_integral_crossing}
\end{equation}
In the long-time limit, if the condition in Eq.~\eqref{eq:pieterbarg_condition} is satisfied, we have,
\begin{equation}
    \lim_{T \rightarrow \infty}\frac{\Var\left[\Nu\right]}{T} = \lim_{T \rightarrow \infty}\frac{\E\left[\Nu\right]}{T} + 2 \int\limits_{0}^\infty \int\limits_{-\infty}^\infty \int\limits_{-\infty}^\infty \left|y_1 y_2 \right|\bigl( f_{\mathbf{Q}}(u, \,u, \,y_1, \,y_2;  t)  - f_{\mathbf{P}}(u,\,y_1) \, f_{\mathbf{P}}(u,\,y_2)\bigr)\, \dd y_1  \dd y_2  \dd t
 \label{eq:variance_formula_triple_integral_CLT_crossing}
\end{equation}
Finally, the Fano factor is given by,
\begin{equation}
    \FF = 1 + \frac{2\int\limits_{0}^\infty \int\limits_{-\infty}^\infty \int\limits_{-\infty}^\infty \left|y_1 y_2\right| \,\left( f_{\mathbf{Q}}(u, \,u, \,y_1, \,y_2; \, t) - f_{\mathbf{P}}(u,\,y_1) \, f_{\mathbf{P}}(u,\,y_2)\right)\, \dd y_1 \, \dd y_2 \, \dd t}{\int\limits_{-\infty}^\infty |y|\, f_{\mathbf{P}}(u,\,y) \, \dd y}. \label{eq:fano_factor_upcrossing_1_crossing}
\end{equation}

\section{Explicit Formulae for Level Crossings}
\label{section:explicit_formulae_for_crossings}
We employ an approach similar to that outlined in the proof of Theorem~\ref{thrm:main_theorem_upcrossing}. First, we use the variance formula from Eq.~\eqref{eq:variance_formula_triple_integral_crossing} to calculate the double integral in $y_1$ and $y_2$,
\begin{equation}
    I = \underbrace{\int_{-\infty}^\infty \int_{-\infty}^\infty \left|y_1 y_2 \right| f_{\mathbf{Q}}(u, u, y_1, y_2;  t) \, \dd y_1 \, \dd y_2}_{\text{\normalsize \(I_1\)}} - \underbrace{\int_{-\infty}^\infty \int_{-\infty}^\infty \left|y_1 y_2 \right| f_{\mathbf{P}}(u,y_1) \, f_{\mathbf{P}}(u,y_2) \, \dd y_1 \, \dd y_2}_{\text{\normalsize \(I_2\)}} \label{eq:defn_of_I_split_crossing}
\end{equation}
$I_2$ can be evaluated as,
\begin{equation}
\begin{split}
    I_2 &= \frac{e^{-u^2/r_0}}{4\pi^2 \, r_0 q_0}\int_{-\infty}^\infty \int_{-\infty}^\infty \left|y_1 y_2 \right| \exp\left(-\frac{1}{2q_0}\left(y_1^2 + y_2^2\right)\right) \, \dd y_1 \, \dd y_2 \\
    &= \frac{1}{\pi^2 }\left(\frac{q_0}{r_0}\right)e^{-u^2/r_0} \\
\end{split} \label{eq:final_expr_for_I2_crossing}
\end{equation}
To evaluate the expression for $I_1$, we make the variable transform $(y_1, \, y_2)\rightarrow (z_1, z_2)$, which gives us,
\begin{equation}
    I_1 = \frac{1}{2}\int_{-\infty}^\infty \int_{-\infty}^\infty \left|z_2^2 - z_1^2\right| \, f_{\mathbf{Q}}(u, \,u, \,z_1, \,z_2; \, t) \, \dd z_2 \, \dd z_1 \label{eq:I1_defn_z1_z2_crossing}
\end{equation}
Following the proof of Theorem~\ref{thrm:main_theorem_upcrossing}, the joint p.d.f. in the rotated variables takes the form $f_{\mathbf{Q}}(u, u, z_1, z_2; t) = \frac{1}{2\pi^2}\sqrt{\frac{\alpha\beta}{r_0^2-r^2}} \exp\left( -\delta u^2 - \alpha (z_1 - \gamma)^2 - \beta z_2^2\right)$. Substituting this, we get,
\begin{equation}
\begin{split}
    I_1 &= \frac{1}{2}\int_{-\infty}^\infty \int_{-\infty}^\infty \left|z_2^2 - z_1^2\right| \left(\frac{1}{2\pi^2}\sqrt{\frac{\alpha\beta}{r_0^2-r^2}} \,\exp\left( -\delta \, u^2 - \alpha \left(z_1 - \gamma \right)^2 - \beta \, z_2^2\right)\right) \dd z_2 \, \dd z_1 \\
    &= \frac{e^{-\delta\,u^2}}{4\pi^2}\sqrt{\frac{\alpha\beta}{r_0^2-r^2}}\int_{-\infty}^\infty \int_{-\infty}^\infty \left|z_2^2 - z_1^2\right| \exp\left(- \alpha \left(z_1 - \gamma \right)^2 - \beta \, z_2^2\right) \dd z_2 \, \dd z_1
\end{split}
\end{equation}
We derive the result for this double integral in Theorem~\ref{thrm:integral_crossing}, which gives us,
\begin{equation}
\begin{split}
    I_1 = \frac{e^{-\delta\,u^2}}{4\pi^2\sqrt{r_0^2-r^2}} \Biggl[&\frac{2e^{-\alpha \gamma^2}}{\sqrt{\alpha\beta}} \left(1 + \sqrt{\pi} \, \gamma \sqrt{\alpha + \beta}  \, e^{\frac{\alpha^2 \gamma^2}{\alpha + \beta}}  \, \Erf\left(\frac{\alpha \gamma}{\sqrt{\alpha + \beta}}\right)\right) \\& + 4 \pi \left(\frac{\alpha - \beta - 2\alpha \beta \gamma^2}{\alpha\beta}\right) \left(\Ot\left(\gamma\sqrt{\frac{2\alpha\beta}{\alpha+\beta}}, \sqrt{\frac{\alpha}{\beta}}\right)-\frac{1}{8}\right)\Biggr],\label{eq:final_expr_for_I1_crossing}
\end{split}
\end{equation}
and the final expression for $I$ is,
\begin{equation}
\begin{aligned}
    I = \frac{e^{-\delta\,u^2}}{4\pi^2\sqrt{r_0^2-r^2}} \Biggl[&\frac{2e^{-\alpha \gamma^2}}{\sqrt{\alpha\beta}} \left(1 + \sqrt{\pi} \, \gamma \sqrt{\alpha + \beta}  \, e^{\frac{\alpha^2 \gamma^2}{\alpha + \beta}}  \, \Erf\left(\frac{\alpha \gamma}{\sqrt{\alpha + \beta}}\right)\right) \\& + 4 \pi \left(\frac{\alpha - \beta - 2\alpha \beta \gamma^2}{\alpha\beta}\right) \left(\Ot\left(\gamma\sqrt{\frac{2\alpha\beta}{\alpha+\beta}}, \sqrt{\frac{\alpha}{\beta}}\right)-\frac{1}{8}\right)\Biggr] - \frac{1}{\pi^2 }\left(\frac{q_0}{r_0}\right)e^{-u^2/r_0}.
\end{aligned} \label{eq:I_final_result_owensT_crossing}
\end{equation}
Plugging the p.d.f. of $\mathbf{P}$ into Eq.~\eqref{eq:rice_formula_crossings1} gives us,
\begin{equation}
    \E\left[\Nu\right] = T \, \frac{1}{\pi}\sqrt{\frac{q_0}{r_0}} \, e^{-u^2/(2r_0)} \label{eq:rice_formula_gaussian_crossing}
\end{equation}
Using the result from Eq.~\eqref{eq:variance_formula_triple_integral_crossing} and the expression for $I$ found in Eq.~\eqref{eq:I_final_result_owensT_crossing}, for a finite time $T$, we have,
\begin{equation}
    \Var\left[\Nu\right] = T \left(\frac{1}{\pi}\sqrt{\frac{q_0}{r_0}} \, e^{-u^2/(2r_0)} + 2\int_{0}^T \left(1-\frac{t}{T}\right)  \, I \, \dd t\right), \label{eq:variance_finite_time_final_crossing}
\end{equation}
and in the long-time limit, we have,
\begin{equation}
    \lim_{T \rightarrow \infty}\frac{\Var\left[\Nu\right]}{T} =  \frac{1}{\pi}\sqrt{\frac{q_0}{r_0}} \, e^{-u^2/(2r_0)} + 2\int_{0}^\infty I \, \dd t. \label{eq:variance_infinite_time_final_crossing}
\end{equation}
\begin{equation}
    \FF = 1 + 2 \pi \sqrt{\frac{r_0}{q_0}} \,e^{u^2/(2r_0)} \int_{0}^\infty I \, \dd t. \label{eq:fano_factor_final_crossing}
\end{equation}

\section{Special Formulae for Zero Level Upcrossings and Crossings}
\label{section:special_formulae_mean}

\subsection{Upcrossings}
\label{section:mean_crossings}
We present the simplified formulae for the mean and variance of zero-level (mean-level) upcrossings, i.e., when $u=0$. First, we substitute $u=0$ in Eq.~\eqref{eq:rice_formula_gaussian} to get the result for $\E\left[N_0^\uparrow(T)\right]$,
\begin{equation}
    \E\left[N_0^\uparrow(T)\right] = T \frac{1}{2\pi}\sqrt{\frac{q_0}{r_0}}.
\end{equation}
Similarly, substituting $u=0$ in Eq.~\eqref{eq:Iup_final_result_owensT}, we get,
\begin{equation}
\begin{split}
    I_0 =& \frac{1}{4\pi^2\sqrt{r_0^2-r^2}} \Biggl[\frac{1}{2\sqrt{\alpha\beta}} + \pi \left(\frac{\alpha - \beta}{\alpha\beta}\right) \Ot\left(0, \sqrt{\frac{\alpha}{\beta}}\right)\Biggr] - \frac{1}{4\pi^2 }\left(\frac{q_0}{r_0}\right) \\
    =& \frac{1}{4\pi^2\sqrt{{r_0^2-r^2}}} \Biggl[\frac{1}{2\sqrt{\alpha\beta}} + \pi \left(\frac{\alpha - \beta}{\alpha\beta}\right) \left(\frac{1}{2\pi}\arctan\left(\sqrt{\frac{\alpha}{\beta}}\right)\right)\Biggr] - \frac{1}{4\pi^2 }\left(\frac{q_0}{r_0}\right) \\
    =& \frac{1}{8\pi^2\sqrt{r_0^2-r^2}}\Biggl[\frac{1}{\sqrt{\alpha\beta}} + \left(\frac{\alpha - \beta}{\alpha\beta}\right) \arctan\left(\sqrt{\frac{\alpha}{\beta}}\right)\Biggr] - \frac{1}{4\pi^2 }\left(\frac{q_0}{r_0}\right)
\end{split}
\end{equation}
Note that in the simplification above, we have used the fact that $\gamma=0$ when $u=0$, and $\Ot(0,a)=\arctan(a)/(2\pi)$. The variance and Fano factor are found by substituting $u=0$ in Eq.~\eqref{eq:variance_finite_time_final}, \eqref{eq:variance_infinite_time_final} and \eqref{eq:fano_factor_final}. 
\begin{equation}
\begin{aligned}
    \Var\left[\Nup\right] = T \Biggl(\frac{1}{2\pi}\sqrt{\frac{q_0}{r_0}}+ \frac{1}{4\pi^2}\int_{0}^T \left(1-\frac{t}{T}\right)  \Biggl[ \frac{1}{\sqrt{r_0^2-r^2}}\Biggl[\frac{1}{\sqrt{\alpha\beta}} + \left(\frac{\alpha - \beta}{\alpha\beta}\right) \arctan\left(\sqrt{\frac{\alpha}{\beta}}\right)\Biggr] - \frac{2q_0}{r_0}\Biggr]  \dd t\Biggr).
\end{aligned}
\end{equation}
This formula has been reported by Steinberg \cite{steinberg1955short} and Leadbetter \cite{leadbetter1965variance}. In the long-time limit, we have,
\begin{equation}
    \frac{\Var\left[\Nup\right]}{T} =\frac{1}{2\pi}\sqrt{\frac{q_0}{r_0}} + \frac{1}{4\pi^2}\int\limits_{0}^\infty  \frac{1}{\sqrt{r_0^2-r^2}}\Biggl[\frac{1}{\sqrt{\alpha\beta}} + \left(\frac{\alpha - \beta}{\alpha\beta}\right) \arctan\left(\sqrt{\frac{\alpha}{\beta}}\right)\Biggr] - \frac{2q_0}{r_0}  \, \dd t
\end{equation}
\begin{equation}
    \FFup = 1 +\frac{1}{\pi} \sqrt{\frac{r_0}{q_0}} \int_{0}^\infty \frac{1}{2\sqrt{r_0^2-r^2}}\Biggl[\frac{1}{\sqrt{\alpha\beta}} + \left(\frac{\alpha - \beta}{\alpha\beta}\right) \arctan\left(\sqrt{\frac{\alpha}{\beta}}\right)\Biggr] - \frac{q_0}{r_0} \, \dd t
\end{equation}

\subsection{Crossings}
We find the simplified formulae for the mean and variance of zero-level (mean-level) crossings by substituting $u=0$ in the results reported in Section~\ref{section:explicit_formulae_for_crossings}. We first substitute $u=0$ in Eq.~\eqref{eq:rice_formula_gaussian_crossing} to get the formula for the mean,
\begin{equation}
    \E\left[\Nu\right] = T \, \frac{1}{\pi}\sqrt{\frac{q_0}{r_0}}
\end{equation}
Next, to find the formula for variance, we evaluate $I$ from Eq.~\eqref{eq:I_final_result_owensT_crossing} at $u=0$,
\begin{equation}
\begin{split}
    I_0 =& \frac{1}{4\pi^2\sqrt{r_0^2-r^2} }\Biggl[\frac{2}{\sqrt{\alpha\beta}} + 4\pi \left(\frac{\alpha - \beta}{\alpha\beta}\right) \left(\Ot\left(0, \sqrt{\frac{\alpha}{\beta}}\right) - \frac{1}{8}\right)\Biggr] - \frac{1}{\pi^2 }\left(\frac{q_0}{r_0}\right) \\
    =& \frac{1}{4\pi^2\sqrt{r_0^2-r^2}}\Biggl[\frac{2}{\sqrt{\alpha\beta}} + 4\pi \left(\frac{\alpha - \beta}{\alpha\beta}\right) \left(\frac{1}{2\pi}\arctan\left(\sqrt{\frac{\alpha}{\beta}}\right) - \frac{1}{8}\right)\Biggr] - \frac{1}{\pi^2 }\left(\frac{q_0}{r_0}\right) \\
    =& \frac{1}{4\pi^2\sqrt{r_0^2-r^2}}\Biggl[\frac{2}{\sqrt{\alpha\beta}} + 2 \left(\frac{\alpha - \beta}{\alpha\beta}\right)\left(\arctan\left(\sqrt{\frac{\alpha}{\beta}}\right) - \frac{\pi}{4}\right) \Biggr] - \frac{1}{\pi^2 }\left(\frac{q_0}{r_0}\right)\\
    =& \frac{1}{2\pi^2\sqrt{r_0^2-r^2}}\Biggl[\frac{1}{\sqrt{\alpha\beta}} + \left(\frac{\alpha - \beta}{\alpha\beta}\right) \arctan\left(\frac{\sqrt{\alpha/\beta}-1}{\sqrt{\alpha/\beta}+1}\right)\Biggr] - \frac{1}{\pi^2 }\left(\frac{q_0}{r_0}\right)
\end{split}
\end{equation}
Therefore, we have the following expression for the variance,
\begin{equation}
    \Var\left[\Nu\right] = T \Biggl(\frac{1}{\pi}\sqrt{\frac{q_0}{r_0}}  + \frac{2}{\pi^2}\int\limits_{0}^T \biggl(1-\frac{t}{T}\biggr)  \Biggl[\frac{1}{2\sqrt{r_0^2-r^2}}\Biggl[\frac{1}{\sqrt{\alpha\beta}}  + \left(\frac{\alpha - \beta}{\alpha\beta}\right) \arctan\left(\frac{\sqrt{\alpha/\beta}-1}{\sqrt{\alpha/\beta}+1}\right)\Biggr] - \frac{q_0}{r_0}\Biggr]  \dd t\Biggr)
\end{equation}
In the long-time limit, we have,
\begin{equation}
    \frac{\Var\left[\Nu\right]}{T} =\frac{1}{\pi}\sqrt{\frac{q_0}{r_0}} + \frac{1}{\pi^2} \int\limits_{0}^\infty  \frac{1}{\sqrt{r_0^2-r^2}}\Biggl[\frac{1}{\sqrt{\alpha\beta}} + \left(\frac{\alpha - \beta}{\alpha\beta}\right) \arctan\left(\frac{\sqrt{\alpha/\beta}-1}{\sqrt{\alpha/\beta}+1}\right)\Biggr] - \frac{2q_0}{r_0} \, \dd t
\end{equation}
\begin{equation}
    \FF = 1 +\frac{2}{\pi} \sqrt{\frac{r_0}{q_0}} \int_{0}^\infty   \frac{1}{2\sqrt{r_0^2-r^2}}\Biggl[\frac{1}{\sqrt{\alpha\beta}} + \left(\frac{\alpha - \beta}{\alpha\beta}\right) \arctan\left(\frac{\sqrt{\alpha/\beta}-1}{\sqrt{\alpha/\beta}+1}\right)\Biggr] - \frac{q_0}{r_0}\,  \dd t
\end{equation}

\section{Note About the Dimensionality}
In this section, we note the units of various variables ($\alpha, \beta, \gamma, \delta$) and expressions ($\Iup, I$) used to define our main results. 
We note that the autocorrelation function $r(t)$ will be of the form,
\begin{equation}
    r \coloneq r(t) = \sigma^2 \, g\left(t/\tau;\,\btheta\right)
\end{equation}
where $\sigma$ defines the dimension of the stochastic process, $\tau$ defines the timescales of the process, and $\btheta$ denotes dimensionless variables associated with the correlation function. Therefore, $g\left(t/\tau;\,\btheta\right)$ is a dimensionless function. Taking the derivatives with respect to time on both sides, we get,
\begin{equation}
    p = r^\prime(t) = \frac{\sigma^2}{\tau} \, g^\prime\left(t/\tau;\,\btheta\right) \qquad \text{and} \qquad q = -r^{\prime\prime}(t) = -\frac{\sigma^2}{\tau^2} \, g^{\prime\prime}\left(t/\tau;\,\btheta\right)
\end{equation}
Substituting $t=0$ in these expressions, 
\begin{equation}
    r_0 = \sigma^2 \, g\left(0;\,\btheta\right) \qquad \text{and} \qquad q_0 = -\frac{\sigma^2}{\tau^2} \, g^{\prime\prime}\left(0;\,\btheta\right)
\end{equation}
Consider the mean, variance, and Fano factor results for level upcrossings, downcrossings, and total crossings. All of the results depend on $\alpha$, $\beta$, $\gamma$, and $\delta$. Therefore, first, we determine the dimensionality of these variables by plugging in the expressions above. First, consider $\alpha$,
\begin{equation}
\begin{split}
    \alpha &= -\frac{r+r_0}{2\left(p^2+(q-q_0)(r+r_0)\right)} \\
    &= - \frac{\tau^2}{\sigma^2}\left(\frac{g\left(t/\tau;\,\btheta\right) + g\left(0;\,\btheta\right)}{2\left((g^\prime\left(t/\tau;\,\btheta\right))^2 - (g^{\prime\prime}\left(t/\tau;\,\btheta\right)-g^{\prime\prime}\left(0;\,\btheta\right))(g\left(t/\tau;\,\btheta\right)+g\left(0;\,\btheta\right))\right)}\right)
\end{split}
\end{equation}
where the expression in the bracket is dimensionless; therefore,
\begin{equation}
    \alpha = \frac{\tau^2}{\sigma^2} \, h_\alpha\left(\tilde{t}; \, \btheta\right)
\end{equation}
where $\tilde{t} = t/\tau$ is a dimensionless time and $h_\alpha\left(\tilde{t};\, \btheta\right)$ is the dimensionless expression. A similar analysis reveals,
\begin{equation}
    \beta = \frac{\tau^2}{\sigma^2} \, h_\beta\left(\tilde{t}; \, \btheta\right), \qquad \gamma = \frac{u}{\tau} \, h_\gamma\left(\tilde{t}; \, \btheta\right) \qquad \text{and} \qquad \delta = \frac{1}{\sigma^2} \, h_\delta\left(\tilde{t}; \, \btheta\right)
\end{equation}
where $h_\beta\left(\tilde{t};\, \btheta\right)$, $h_\gamma\left(\tilde{t};\, \btheta\right)$ and $h_\delta\left(\tilde{t};\, \btheta\right)$ are dimensionless functions. Using these expressions, it can be shown that $\Iup$ and $I$ can be separated into the following dimensional and dimensionless parts,
\begin{equation}
    \Iup = \frac{1}{\tau^2} \, h_{\Iup}\left(\tilde{t};\, \psi, \, \btheta\right) \qquad \text{and} \qquad I = \frac{1}{\tau^2} \, h_{I}\left(\tilde{t};\, \psi, \, \btheta\right)
\end{equation}
where we additionally define the dimensionless parameter $\psi = u / \sigma$. Separating the dimensional and dimensionless expression for mean of upcrossings, we get,
\begin{equation}
    \E\left[\Nup\right] = T \,\frac{1}{2\pi}\sqrt{\frac{q_0}{r_0}} \, e^{-u^2/(2r_0)} = \frac{T}{\tau}\left( \frac{1}{2\pi}\sqrt{\frac{-g^{\prime\prime}\left(0;\,\btheta\right)}{g\left(0;\,\btheta\right)}} \, e^{-\psi^2/(2g\left(0;\,\btheta\right))}\right).
\end{equation}
Similarly, for variance in finite time, we have,
\begin{equation}
\begin{split}
    \Var\left[\Nup\right] &= T \left(\frac{1}{2\pi}\sqrt{\frac{q_0}{r_0}} \, e^{-u^2/(2r_0)} + 2\int_{0}^T \left(1-\frac{t}{T}\right)  \, \Iup \, \dd t\right) \\
    &= \frac{T}{\tau} \left(\frac{1}{2\pi}\sqrt{\frac{-g^{\prime\prime}\left(0;\,\btheta\right)}{g\left(0;\,\btheta\right)}} \, e^{-\psi^2/(2g\left(0;\,\btheta\right))} + 2\int_{0}^{T/\tau} \left(1-\frac{\tau}{T} \, \tilde{t}\right)  \,  h_{\Iup}\left(\tilde{t};\, \psi, \, \btheta\right) \, \dd \tilde{t}\right),
\end{split}
\end{equation}
and in the long-time limit,
\begin{equation}
    \Var\left[\Nup\right] = \frac{T}{\tau} \left(\frac{1}{2\pi}\sqrt{\frac{-g^{\prime\prime}\left(0;\,\btheta\right)}{g\left(0;\,\btheta\right)}} \, e^{-\psi^2/(2g\left(0;\,\btheta\right))} + 2\int_{0}^{\infty}   h_{\Iup}\left(\tilde{t};\, \psi, \, \btheta\right) \, \dd \tilde{t}\right).
\end{equation}
Finally, the Fano factor (Eq.~\eqref{eq:fano_factor_upcrossing_1}) can be shown to be independent of $\tau$,
\begin{equation}
    \FFup = 1 + 4 \pi \sqrt{\frac{r_0}{q_0}} \,e^{u^2/(2r_0)} \int_{0}^\infty \Iup \, \dd t = 1 + 4 \pi \sqrt{\frac{g\left(0;\,\btheta\right)}{-g^{\prime\prime}\left(0;\,\btheta\right)}} \, e^{\psi^2/(2g\left(0;\,\btheta\right))} \int_{0}^\infty h_{\Iup}\left(\tilde{t};\, \psi, \, \btheta\right)  \, \dd \tilde{t}
\end{equation}
Similar results can be found for the mean, variance, and Fano factor of the level crossing process.

\section{Derivation of SDHO Autocorrelation Functions}
\label{section:sdho_derivation}
Here we derive the autocorrelation functions for the stochastic damped harmonic oscillator stated in the main text (Eqs.~\eqref{eq:acf_underdamped}--\eqref{eq:acf_overdamped}). The SDE (Eq.~\eqref{eq:sdho_sde}) can be recast as a two-dimensional linear stochastic dynamical system:
\begin{equation}
\begin{split}
    \frac{\dd x}{\dd t} &= v \\
    \frac{\dd v}{\dd t} &= -\omega_0^2 x - 2 \zeta \omega_0 v + \sqrt{4 \zeta \omega_0 \temp} \, \eta(t),
\end{split}
\label{eq:sdho_sde_2d}
\end{equation}
where \( v(t) = \dot{x}(t) \) is the velocity of the oscillator.
This system conforms to the general form of a linear stochastic dynamical system:
\begin{equation}
    \dd\mathbf{x} = \mathbf{J} \mathbf{x} \dd t + \mathbf{L} \dd\mathbf{W}
\end{equation}
where $\mathbf{x} = [x, v]^\top$ is the state vector, and $\dd\mathbf{W}=[\dd W_1, \dd W_2]^\top$ is a vector of independent standard Wiener increments. The Jacobian matrix $\mathbf{J}$ and the noise diffusion matrix $\mathbf{L}$ are given by:
\begin{equation}
\renewcommand{\arraystretch}{1.0}
    \mathbf{J} = \begin{bmatrix}
        0 & 1 \\
        -\omega_0^2 & -2 \zeta \omega_0
    \end{bmatrix} \quad \text{and} \quad
    \mathbf{L} = \begin{bmatrix}
        0 & 0 \\
        0 &\sqrt{4 \zeta \omega_0 \temp}
    \end{bmatrix}.
\label{eq:sdho_matrices}
\end{equation}
The power spectral density matrix of the state vector \(\mathbf{x}\) can be obtained using the formula \cite{rawat2024element}:
\begin{equation}
    \boldsymbol{\mathcal{S}}_\mathbf{x}(\omega) = (\dot{\iota}\omega \mathbf{I} - \mathbf{J})^{-1} \, \mathbf{L} \, \mathbf{L}^\top (-\dot{\iota}\omega \mathbf{I} - \mathbf{J})^{-\top}.
\label{eq:psd_general_formula_sdho}
\end{equation}
Extracting the (1,1) component of \(\boldsymbol{\mathcal{S}}_\mathbf{x}(\omega)\) yields the power spectral density of the position variable \(x(t)\):
\begin{equation}
    \mathcal{S}_x(\omega) = \frac{4 \zeta \omega_0 \temp}{(\omega_0^2 - \omega^2)^2 + (2 \zeta \omega_0 \omega)^2}.
\label{eq:sdho_psd}
\end{equation}
According to the Wiener--Khinchin theorem, the autocorrelation function of the stationary process \(x(t)\), denoted by \(r_x(t) = \E[x(t')x(t'+t)]\), is the inverse Fourier transform of its power spectral density:
\begin{equation}
    r_x(t) = \mathcal{F}^{-1}\left\{\mathcal{S}_x(\omega)\right\}(t).
\label{eq:acf_ift_sdho}
\end{equation}
Calculating the inverse Fourier transform of Eq.~\eqref{eq:sdho_psd} provides the correlation functions $r_x(t)$ in the three distinct damping regimes stated in the main text (Eqs.~\eqref{eq:acf_underdamped}--\eqref{eq:acf_overdamped}). Note that $r_x(0) = \temp/\omega_0^2$ in all cases, representing the variance of the position $x(t)$ as determined by the equipartition theorem.

\section{Additional Theorems}

\begin{theorem} \label{thrm:positivity_of_correlation_term}
Let $r(t)$ be the autocorrelation function of a one-dimensional real Gaussian stationary process. For brevity, we introduce the following notations: $r(t)\rightarrow r$, $r(0)\rightarrow r_0$, $\rp(t) \rightarrow p$, $\rp(0)\rightarrow p_0$, $-\rpp(t) \rightarrow q$, and $-\rpp(0)\rightarrow q_0$. Then, for all $t \geq 0$ the following expressions are strictly positive,
\begin{equation}
\begin{split}
 &-\frac{r+r_0}{2\left(p^2+(q-q_0)(r+r_0)\right)} > 0 \\
 &-\frac{r_0-r}{2\left(p^2+(q+q_0)(r-r_0)\right)} > 0
\end{split}
\end{equation}
\end{theorem}

\begin{proof}
Since $r(t)$ is the autocorrelation function of a real stationary process, it has the spectral representation
\begin{equation}
    r(t)=\int_{-\infty}^{\infty}\cos(\omega t)\,dF(\omega),
\end{equation}
where $F$ is a nonnegative measure. Differentiating with respect to time, we obtain
\begin{equation}
    r'(t)=-\int_{-\infty}^{\infty}\omega\sin(\omega t)\,dF(\omega)
    \quad\text{and}\quad
    r''(t)=-\int_{-\infty}^{\infty}\omega^2\cos(\omega t)\,dF(\omega).
\end{equation}
In the notation defined in the theorem and using the integral formulae above, we have,
\begin{equation}
\begin{split}
    r &= \int_{-\infty}^{\infty}\cos(\omega t)\,dF(\omega) \\
    p &= -\int_{-\infty}^{\infty}\omega\sin(\omega t)\,dF(\omega) \\
    q &= \int_{-\infty}^{\infty}\omega^2\cos(\omega t)\,dF(\omega) \\
    r_0 &= \int_{-\infty}^{\infty}dF(\omega) \\
    p_0 &= 0 \\
    q_0 &= \int_{-\infty}^{\infty}\omega^2 \,dF(\omega)
\end{split}
\end{equation}
Using these relations, we have the following integral expressions for the relevant quantities,
\begin{equation}
    r+r_0=\int_{-\infty}^{\infty}\bigl(1+\cos(\omega t)\bigr)\,dF(\omega)
    \quad\text{and}\quad
    r_0-r=\int_{-\infty}^{\infty}\bigl(1-\cos(\omega t)\bigr)\,dF(\omega),
\end{equation}
and,
\begin{equation}
    q_0-q=\int_{-\infty}^{\infty}\omega^2\Bigl(1-\cos(\omega t)\Bigr)dF(\omega)
    \quad\text{and}\quad
    q_0+q=\int_{-\infty}^{\infty}\omega^2\Bigl(1+\cos(\omega t)\Bigr)dF(\omega).    
\end{equation}
Now, writing
\begin{equation}
    p^2=\left(\int_{-\infty}^{\infty}\omega\sin(\omega t)\,dF(\omega)\right)^2,    
\end{equation}
and applying the Cauchy--Schwarz inequality,
\begin{equation}
\left( \int_{S} f(x) g(x) \, d\mu(x) \right)^2 \le \left( \int_{S} f(x)^2 \, d\mu(x) \right) \left( \int_{S} g(x)^2 \, d\mu(x) \right), 
\end{equation}
gives us,
\begin{equation}
    p^2\le \left(\int_{-\infty}^{\infty}\omega^2\frac{\sin^2(\omega t)}{1+\cos(\omega t)}\,dF(\omega)\right)
    \left(\int_{-\infty}^{\infty}(1+\cos(\omega t))\,dF(\omega)\right).
\end{equation}
A short calculation shows that
\begin{equation}
    \frac{\sin^2(\omega t)}{1+\cos(\omega t)}
    =\frac{4\sin^2\bigl(\frac{\omega t}{2}\bigr)\cos^2\bigl(\frac{\omega t}{2}\bigr)}
    {2\cos^2\bigl(\frac{\omega t}{2}\bigr)}
    =2\sin^2\Bigl(\frac{\omega t}{2}\Bigr) = 1 - \cos\left(\omega t\right).
\end{equation}
Therefore,
\begin{equation}
    p^2\le \left(\int_{-\infty}^{\infty}\omega^2\left(1 - \cos\left(\omega t\right)\right)dF(\omega)\right)\left(\int_{-\infty}^{\infty}(1+\cos(\omega t))\,dF(\omega)\right),
\end{equation}
or,
\begin{equation}
    p^2 \leq \left(q_0-q\right) \left(r_0+r\right).
\end{equation}
Since equality in the Cauchy--Schwarz inequality can occur only if the integrands are proportional, which is excluded for a nondegenerate spectral measure, i.e., a spectral measure not concentrated on just a finite set of points, we have,
\begin{equation}
    p^2-(q_0-q)(r+r_0)<0.
\end{equation}
Further, since $r+r_0>0\; \forall \;t \geq 0$, we have,
\begin{equation}
 -\frac{r+r_0}{2\left(p^2+(q-q_0)(r+r_0)\right)} > 0 
\end{equation}
Applying Cauchy--Schwarz, we can also write,
\begin{equation}
\begin{split}
    p^2 &\leq \left(\int_{-\infty}^{\infty}\omega^2\frac{\sin^2(\omega t)}{1-\cos(\omega t)}\,dF(\omega)\right)
    \left(\int_{-\infty}^{\infty}(1-\cos(\omega t))\,dF(\omega)\right) \\
    p^2 &\leq \left(\int_{-\infty}^{\infty}\omega^2\left(1+\cos(\omega t)\right)\,dF(\omega)\right)
    \left(\int_{-\infty}^{\infty}(1-\cos(\omega t))\,dF(\omega)\right) \\
    p^2 &\leq \left(q_0+q\right)\left(r_0-r\right) \\
\end{split}
\end{equation}
Since the spectral measure is nondegenerate, strict inequality follows,
\begin{equation}
    p^2 -\left(q_0+q\right)\left(r_0-r\right) < 0
\end{equation}
Further, since $r_0-r>0\; \forall \;t \geq 0$, we have,
\begin{equation}
 -\frac{r_0-r}{2\left(p^2+(q+q_0)(r-r_0)\right)} > 0 
\end{equation}
\end{proof}

\begin{theorem} \label{thrm:integral_main}
Let $\alpha > 0$, $\beta > 0$, and $\gamma \in \R$. The following result holds:
\begin{equation}
\begin{split}
 \int_{-\infty}^\infty \int_{|x|}^\infty (y^2-x^2) \, e^{-\alpha(x-\gamma)^2-\beta y^2} \dd y \, \dd x =& \frac{e^{-\alpha \gamma^2}}{2\alpha\beta} \left(1 + \sqrt{\pi} \, \gamma \sqrt{\alpha + \beta} \, e^{\frac{\alpha^2 \gamma^2}{\alpha + \beta}} \, \Erf\left(\frac{\alpha \gamma}{\sqrt{\alpha + \beta}}\right)\right) \\&
    + \pi \left(\frac{\alpha - \beta - 2\alpha\beta\gamma^2}{(\alpha\beta)^{3/2}}\right) \Ot\left(\gamma\sqrt{\frac{2\alpha\beta}{\alpha+\beta}}, \sqrt{\frac{\alpha}{\beta}}\right)
\end{split} \label{eq:main_theorem_equation}
\end{equation}
where $\Erf(z)$ is the error function defined as
\begin{equation}
\Erf(z) = \frac{2}{\sqrt{\pi}} \int_0^z e^{-t^2}  \dd t, \label{eq:erf_defn_appendix}
\end{equation}
and $\Ot(h, a)$ is Owen's T function \cite{owen1980table} defined as
\begin{equation}
\Ot(h, a) = \frac{1}{2\pi} \int_0^a \frac{e^{-\frac{1}{2}h^2(1+t^2)}}{1+t^2} \, \dd t. \label{eq:Ot_defn_appendix}
\end{equation}
\end{theorem}

\begin{proof}
Denoting the integral that we have to evaluate by,
\begin{equation}
    I = \int_{-\infty}^\infty \int_{|x|}^\infty (y^2-x^2) \, e^{-\alpha(x-\gamma)^2-\beta y^2} \dd y \, \dd x
\end{equation}
Integrating the inner integral in $y$ and using the result from Lemma~\ref{lemma:inner_integral}, we have,
\begin{equation}
\begin{split}
    I &= \int_{-\infty}^\infty \frac{1}{2 \beta}  |x|\, e^{ -\alpha (x - \gamma)^{2} -\beta x^{2}} - \frac{\sqrt{\pi}}{4 \beta^{3/2}} \left(2 \beta x^2 - 1\right) e^{ -\alpha (x - \gamma)^{2} } \Erfc\left(\sqrt{\beta}|x|\right) \dd x \\
    &= \frac{1}{2 \beta} \underbrace{\int_{-\infty}^\infty |x| \, e^{ -\alpha (x - \gamma)^{2} -\beta x^{2}} \dd x}_{\text{\normalsize \(I_1\)}}
    - \frac{\sqrt{\pi}}{4 \beta^{3/2}} \underbrace{\int_{-\infty}^\infty \left(2 \beta x^2 - 1\right) e^{ -\alpha (x - \gamma)^{2} } \Erfc\left(\sqrt{\beta}|x|\right) \dd x}_{\text{\normalsize \(I_2\)}}
\end{split} \label{eq:I_defn}
\end{equation}
$I_1$ is shown to be given by the following expression in Lemma~\ref{lemma:I1},
\begin{equation}
    I_1 = \frac{e^{-\alpha  \gamma ^2}}{(\alpha +\beta )^{3/2}} \left(\sqrt{\alpha +\beta } + \sqrt{\pi } \, \alpha  \gamma \, e^{\frac{\alpha ^2 \gamma ^2}{\alpha +\beta }} \Erf \left(\frac{\alpha  \gamma }{\sqrt{\alpha +\beta }}\right)\right)
\end{equation}
We make the variable change of $\sqrt{\beta}\, x = k$ in the integral $I_2$ which gives us,
\begin{equation}
\begin{split}
    I_2 &= \frac{1}{\sqrt{\beta}} \int_{-\infty}^\infty \left(2 k^2 - 1\right) e^{ -\frac{\alpha}{\beta} (k - \gamma\sqrt{\beta})^{2} } \Erfc\left(|k|\right) \dd k \\
    &= \frac{1}{\sqrt{\beta}} \underbrace{\int_{-\infty}^\infty \left(2 k^2 - 1\right) e^{ -\frac{\alpha}{\beta} (k - \gamma\sqrt{\beta})^{2}} \dd k}_{\text{\normalsize \(I_{21}\)}} - \frac{1}{\sqrt{\beta}} \underbrace{\int_{-\infty}^\infty \left(2 k^2 - 1\right) e^{ -\frac{\alpha}{\beta} (k - \gamma\sqrt{\beta})^{2} } \Erf\left(|k|\right) \dd k}_{\text{\normalsize \(I_{22}\)}}
\end{split} \label{eq:I2_defn}
\end{equation}
The integral in $I_{21}$ is given by Lemma~\ref{lemma:I21}, upon simplification, we get,
\begin{equation}
    I_{21} = \sqrt{\pi} \sqrt{\frac{\beta}{\alpha}} \left(\frac{\beta}{\alpha} + 2 \beta \gamma^2 - 1 \right)
\end{equation}
and $I_{22}$ can be written as,
\begin{equation}
\begin{split}
    I_{22} &= \int_{-\infty}^\infty \left(2 k^2 - 1\right) e^{ -\frac{\alpha}{\beta} (k - \gamma\sqrt{\beta})^{2} } \Erf\left(|k|\right) \dd k \\
    &= \int_{0}^\infty e^{ -\frac{\alpha}{\beta} (k - \gamma\sqrt{\beta})^{2} } \left(2 k^2 - 1\right)\, \Erf\left(k\right) \dd k + \int_{-\infty}^0 e^{ -\frac{\alpha}{\beta} (k - \gamma\sqrt{\beta})^{2} } \left(2 k^2 - 1\right)\, \Erf\left(-k\right) \dd k \\
    &= \int_{0}^\infty e^{ -\frac{\alpha}{\beta} (k - \gamma\sqrt{\beta})^{2} } \left(2 k^2 - 1\right)\, \Erf\left(k\right) \dd k + \int_{0}^\infty e^{ -\frac{\alpha}{\beta} (k + \gamma\sqrt{\beta})^{2} } \left(2 k^2 - 1\right)\, \Erf\left(k\right) \dd k \\
    &= \int_{0}^\infty \left(e^{ -\frac{\alpha}{\beta} (k - \gamma\sqrt{\beta})^{2}} + e^{ -\frac{\alpha}{\beta} (k + \gamma\sqrt{\beta})^{2}} \right)\left(2 k^2 - 1\right)\, \Erf\left(k\right) \dd k 
\end{split}
\end{equation}
Applying integration by parts, we get
\begin{equation}
\begin{split}
    I_{22} =& \int_{0}^\infty \underbrace{\left(e^{ -\frac{\alpha}{\beta} (k - \gamma\sqrt{\beta})^{2}} + e^{ -\frac{\alpha}{\beta} (k + \gamma\sqrt{\beta})^{2}} \right)\left(2 k^2 - 1\right)}_{\text{\normalsize second part}}\, \underbrace{\Erf\left(k\right)}_{\text{\normalsize first part}} \dd k \\
    =& \underbrace{\left[\Erf\left(k\right) \int \left(e^{ -\frac{\alpha}{\beta} (k - \gamma\sqrt{\beta})^{2}} + e^{ -\frac{\alpha}{\beta} (k + \gamma\sqrt{\beta})^{2}} \right)\left(2 k^2 - 1\right) \dd k \right]_0^\infty}_{\text{\normalsize \(I_{221}\)}} \\& - \underbrace{\int_{0}^\infty \left(\frac{2}{\sqrt{\pi}} e^{-k^2}\right) \left(\int \left(e^{ -\frac{\alpha}{\beta} (k - \gamma\sqrt{\beta})^{2}} + e^{ -\frac{\alpha}{\beta} (k + \gamma\sqrt{\beta})^{2}} \right)\left(2 k^2 - 1\right) \dd k\right) \dd k}_{\text{\normalsize \(I_{222}\)}}
\end{split} \label{eq:I21_by_parts}
\end{equation}
First, the indefinite integral can be evaluated using the standard results on the normal integrals \cite{owen1980table},
\begin{equation}
\begin{split}
    \int &\biggl(e^{ -\frac{\alpha}{\beta}  (k - \gamma\sqrt{\beta})^{2}} +  e^{ -\frac{\alpha}{\beta} (k + \gamma\sqrt{\beta})^{2}} \biggr)  \left(2 k^2 - 1\right) \dd k = \frac{\beta}{\alpha} e^{-\frac{\alpha}{\beta}  \left(k + \gamma\sqrt{\beta } \right)^2} \left( \gamma\sqrt{\beta } -\left(k + \gamma\sqrt{\beta } \right) e^{\frac{4 \alpha  \gamma }{\sqrt{\beta }} k} - k\right) \\& + \frac{\sqrt{\pi}}{2\alpha} \sqrt{\frac{\beta }{\alpha }} \left(\beta - \alpha  +2 \alpha\beta \gamma ^2 \right) \left(\Erf\left(\sqrt{\frac{\alpha}{\beta}}\left(k+\gamma\sqrt{\beta}\right)\right)+\Erf\left(\sqrt{\frac{\alpha}{\beta}}\left(k-\gamma\sqrt{\beta}\right) \right)\right)
\end{split}
\label{eq:integral_vk}
\raisetag{30pt}
\end{equation}
Let us denote the integral above by $v(k)$. We have the following:
\begin{equation}
    \biggl[\Erf(k) \int \biggl(e^{ -\frac{\alpha}{\beta} (k - \gamma\sqrt{\beta})^{2}} + e^{ -\frac{\alpha}{\beta} (k + \gamma\sqrt{\beta})^{2}} \biggr)\left(2 k^2 - 1\right) \dd k \biggr]_0^\infty = \lim_{k \rightarrow \infty} \Erf(k)\, v(k) - \lim_{k \rightarrow 0} \Erf(k) \,v(k)
\end{equation}
Using the fact that $\lim_{k \rightarrow \infty} \Erf(k) = 1$ and $\lim_{k \rightarrow \infty} \left(k\,e^{-a \,k}\right) = 0$, we have,
\begin{equation}
    \lim_{k \rightarrow \infty} \Erf(k) \, v(k) = \frac{\sqrt{\pi}}{\alpha} \sqrt{\frac{\beta }{\alpha }} \left(\beta - \alpha  +2 \alpha\beta \gamma ^2 \right),
\end{equation}
and using the fact that $\lim_{k \rightarrow 0} \Erf(k) = 0$, we have,
\begin{equation}
    \lim_{k \rightarrow 0} \Erf(k) \, v(k) = 0.
\end{equation}
Therefore, the first term of $I_{22}$ in Eq.~\eqref{eq:I21_by_parts} is,
\begin{equation}
    I_{221} = \biggl[\Erf(k) \int \biggl(e^{ -\frac{\alpha}{\beta} (k - \gamma\sqrt{\beta})^{2}} + e^{ -\frac{\alpha}{\beta} (k + \gamma\sqrt{\beta})^{2}} \biggr)\left(2 k^2 - 1\right) \dd k \biggr]_0^\infty = \frac{\sqrt{\pi}}{\alpha} \sqrt{\frac{\beta }{\alpha }} \left(\beta - \alpha  +2 \alpha\beta \gamma ^2 \right)
\end{equation}
Using the integral in Eq.~\eqref{eq:integral_vk}, we can write $I_{222}$ as the following definite integral,
\begin{equation}
\begin{split}
    I_{222} =& \int_{0}^\infty \left(\frac{2}{\sqrt{\pi}} e^{-k^2}\right) \biggl(\frac{\beta}{\alpha} e^{-\frac{\alpha}{\beta}  \left(k + \gamma\sqrt{\beta } \right)^2} \left( \gamma\sqrt{\beta } -\left(k + \gamma\sqrt{\beta } \right) e^{\frac{4 \alpha  \gamma }{\sqrt{\beta }} k} - k\right) \\&+ \frac{\sqrt{\pi}}{2\alpha} \sqrt{\frac{\beta }{\alpha }} \left(\beta - \alpha  + 2 \alpha\beta \gamma ^2 \right) \left(\Erf\left(\sqrt{\frac{\alpha}{\beta}}\left(k+\gamma\sqrt{\beta}\right)\right)+\Erf\left(\sqrt{\frac{\alpha}{\beta}}\left(k-\gamma\sqrt{\beta}\right) \right)\right)\biggr) \,\dd k 
\end{split} 
\end{equation}
Upon expanding, we get,
\begin{equation}
\begin{split}
    I_{222} =& \frac{2}{\sqrt{\pi}}  \frac{\beta}{\alpha}  \underbrace{\int_{0}^\infty e^{-\frac{\alpha}{\beta}  \left(k + \gamma\sqrt{\beta } \right)^2 - k^2} \left( \gamma\sqrt{\beta } -\left(k + \gamma\sqrt{\beta } \right) e^{\frac{4 \alpha  \gamma }{\sqrt{\beta }} k} - k\right)  \dd k}_{\text{\normalsize \(I_{2221}\)}} \\
    &+ \frac{1}{\alpha} \sqrt{\frac{\beta }{\alpha }} \left(\beta - \alpha  + 2 \alpha\beta \gamma ^2 \right) \underbrace{ \int_{0}^\infty e^{-k^2}\left(\Erf\left(\sqrt{\frac{\alpha}{\beta}}\left(k+\gamma\sqrt{\beta}\right)\right)+\Erf\left(\sqrt{\frac{\alpha}{\beta}}\left(k-\gamma\sqrt{\beta}\right) \right)\right) \,\dd k }_{\text{\normalsize \(I_{2222}\)}}
\end{split} \label{eq:I222_defn}
\end{equation}
Using Lemma~\ref{lemma:I2221}, we have the following expression for $I_{2221}$,
\begin{equation}
    I_{2221} = -\frac{\beta \, e^{-\alpha  \gamma ^2} }{(\alpha +\beta )^{3/2}} \biggl(\sqrt{\alpha +\beta } +\sqrt{\pi } \,\gamma \, (2 \alpha +\beta ) \, e^{\frac{\alpha ^2 \gamma ^2}{\alpha +\beta }} \Erf\left(\frac{\alpha  \gamma }{\sqrt{\alpha +\beta }}\right)\biggr).
\end{equation}
Finally, using Lemma~\ref{lemma:I2222}, we have the following expression for $I_{2222}$,
\begin{equation}
    I_{2222} = 4 \sqrt{\pi} \,\Ot\left(\sqrt{\frac{2 \alpha \beta}{\alpha + \beta}} \, \gamma, \sqrt{\frac{\alpha}{\beta}}\right)
\end{equation}

Now, we go backward to find a simplified expression for $I$. First, we find $I_{222}$ using the definition in Eq.~\eqref{eq:I222_defn},
\begin{equation}
\begin{split}
    I_{222} =& \frac{2}{\sqrt{\pi}}  \frac{\beta}{\alpha} I_{2221} + \frac{1}{\alpha} \sqrt{\frac{\beta }{\alpha }} \left(\beta - \alpha  + 2 \alpha\beta \gamma ^2 \right) I_{2222} \\
    =& -\frac{2 \beta^2 \, e^{-\alpha  \gamma ^2} }{\sqrt{\pi} \, \alpha \, (\alpha +\beta )^{3/2}} \biggl(\sqrt{\alpha +\beta } +\sqrt{\pi } \,\gamma \, (2 \alpha +\beta ) \, e^{\frac{\alpha ^2 \gamma ^2}{\alpha +\beta }} \Erf\left(\frac{\alpha  \gamma }{\sqrt{\alpha +\beta }}\right)\biggr) \\&+ \frac{4 \sqrt{\pi}}{\alpha} \sqrt{\frac{\beta }{\alpha }} \left(\beta - \alpha  + 2 \alpha\beta \gamma ^2 \right) \Ot\left(\sqrt{\frac{2 \alpha \beta}{\alpha + \beta}} \, \gamma, \sqrt{\frac{\alpha}{\beta}}\right)
\end{split}
\end{equation}
Next, we find $I_{22}$ using the definition in Eq.~\eqref{eq:I21_by_parts},
\begin{equation}
\begin{split}
    I_{22} =& I_{221} - I_{222} \\
    =& \frac{\sqrt{\pi}}{\alpha} \sqrt{\frac{\beta }{\alpha }} \left(\beta - \alpha  +2 \alpha\beta \gamma ^2 \right) + \frac{2 \beta^2 \, e^{-\alpha  \gamma ^2} }{\sqrt{\pi} \, \alpha \, (\alpha +\beta )^{3/2}} \biggl(\sqrt{\alpha +\beta } +\sqrt{\pi } \,\gamma \, (2 \alpha +\beta ) \, e^{\frac{\alpha ^2 \gamma ^2}{\alpha +\beta }} \Erf\left(\frac{\alpha  \gamma }{\sqrt{\alpha +\beta }}\right)\biggr) \\&- \frac{4 \sqrt{\pi}}{\alpha} \sqrt{\frac{\beta }{\alpha }} \left(\beta - \alpha  + 2 \alpha\beta \gamma ^2 \right) \Ot\left(\sqrt{\frac{2 \alpha \beta}{\alpha + \beta}} \, \gamma, \sqrt{\frac{\alpha}{\beta}}\right)
\end{split}
\raisetag{30pt}
\end{equation}
Using the definition in Eq.~\eqref{eq:I2_defn}, $I_2$ is given by,
\begin{equation}
\begin{split}
    I_2 =& \frac{1}{\sqrt{\beta}} \left(I_{21} - I_{22}\right) \\
    =& \cancel{\sqrt{\frac{\pi}{\alpha}} \left(\frac{\beta}{\alpha} + 2 \beta \gamma^2 - 1 \right)} - \cancel{\frac{\sqrt{\pi}}{\alpha^{3/2}}  \left(\beta - \alpha  + 2 \alpha\beta \gamma ^2 \right)} + \frac{4 \sqrt{\pi}}{\alpha^{3/2}}  \left(\beta - \alpha  + 2 \alpha\beta \gamma ^2 \right) \Ot\left(\sqrt{\frac{2 \alpha \beta}{\alpha + \beta}} \, \gamma, \sqrt{\frac{\alpha}{\beta}}\right) \\
    &-\frac{2 \beta^{3/2} \, e^{-\alpha  \gamma ^2} }{\sqrt{\pi} \, \alpha \, (\alpha +\beta )^{3/2}} \biggl(\sqrt{\alpha +\beta } +\sqrt{\pi } \,\gamma \, (2 \alpha +\beta ) \, e^{\frac{\alpha ^2 \gamma ^2}{\alpha +\beta }} \Erf\left(\frac{\alpha  \gamma }{\sqrt{\alpha +\beta }}\right)\biggr) 
\end{split}
\raisetag{30pt}
\end{equation}
Finally, using the definition in Eq.~\eqref{eq:I_defn}, we get the expression for $I$,
\begin{equation}
\begin{split}
    I =&  \frac{1}{2\beta} I_1 - \frac{\sqrt{\pi}}{4\beta^{3/2}} I_2 \\
    =& \frac{e^{-\alpha  \gamma ^2}}{2 \beta \, (\alpha +\beta )^{3/2}} \left(\sqrt{\alpha +\beta } + \sqrt{\pi } \, \alpha  \gamma \, e^{\frac{\alpha ^2 \gamma ^2}{\alpha +\beta }} \Erf \left(\frac{\alpha  \gamma }{\sqrt{\alpha +\beta }}\right)\right) \\&- \frac{\pi}{(\alpha\beta)^{3/2}}  \left(\beta - \alpha  + 2 \alpha\beta \gamma ^2 \right) \Ot\left(\sqrt{\frac{2 \alpha \beta}{\alpha + \beta}} \, \gamma, \sqrt{\frac{\alpha}{\beta}}\right) \\
    &+ \frac{ \, e^{-\alpha  \gamma ^2} }{2 \alpha \, (\alpha +\beta )^{3/2}} \biggl(\sqrt{\alpha +\beta } +\sqrt{\pi } \,\gamma \, (2 \alpha +\beta ) \, e^{\frac{\alpha ^2 \gamma ^2}{\alpha +\beta }} \Erf\left(\frac{\alpha  \gamma }{\sqrt{\alpha +\beta }}\right)\biggr)
\end{split}
\end{equation}
Simplifying the expression above gives us the final result,
\begin{equation}
    I = \frac{e^{-\alpha \gamma^2}}{2\alpha\beta} \left(1 + \sqrt{\pi}  \gamma \sqrt{\alpha + \beta}   e^{\frac{\alpha^2 \gamma^2}{\alpha + \beta}}  \Erf\left(\frac{\alpha \gamma}{\sqrt{\alpha + \beta}}\right)\right) + \pi \left(\frac{\alpha - \beta - 2\alpha\beta\gamma^2}{(\alpha\beta)^{3/2}}\right) \Ot\left(\gamma\sqrt{\frac{2\alpha\beta}{\alpha+\beta}}, \sqrt{\frac{\alpha}{\beta}}\right)
\end{equation}

\end{proof}

\begin{corollary}[to Theorem.~\ref{thrm:integral_main}]
Let $\alpha > 0$ and $\beta > 0$. The following result holds:
\begin{equation}
 \int_{-\infty}^\infty \int_{|x|}^\infty (y^2-x^2)\, e^{-\alpha x^2-\beta y^2} \dd y \, \dd x = \frac{\sqrt{\alpha\beta} + \left(\alpha - \beta\right) \, \arctan\left( \sqrt{\alpha/\beta} \right) }{2 (\alpha\beta)^{3/2}} 
\end{equation}
\end{corollary}

\begin{proof}
The integral is a special case of the integral in Theorem~\ref{thrm:integral_main} with the parameter $\gamma=0$. Eq.~\eqref{eq:main_theorem_equation} states that,
\begin{equation}
\begin{split}
 \int_{-\infty}^\infty \int_{|x|}^\infty (y^2-x^2) \, e^{-\alpha(x-\gamma)^2-\beta y^2} \dd y \, \dd x =& \frac{e^{-\alpha \gamma^2}}{2\alpha\beta} \left(1 + \sqrt{\pi} \, \gamma \sqrt{\alpha + \beta} \, e^{\frac{\alpha^2 \gamma^2}{\alpha + \beta}} \, \Erf\left(\frac{\alpha \gamma}{\sqrt{\alpha + \beta}}\right)\right) \\&
    + \pi \left(\frac{\alpha - \beta - 2\alpha\beta\gamma^2}{(\alpha\beta)^{3/2}}\right) \Ot\left(\gamma\sqrt{\frac{2\alpha\beta}{\alpha+\beta}}, \sqrt{\frac{\alpha}{\beta}}\right).
\end{split} 
\end{equation}
Upon substituting $\gamma=0$ on both sides, we get,
\begin{equation}
  \int_{-\infty}^\infty \int_{|x|}^\infty (y^2-x^2)\, e^{-\alpha x^2-\beta y^2} \dd y \, \dd x = \frac{1}{2\alpha\beta} 
    + \pi \left(\frac{\alpha - \beta}{(\alpha\beta)^{3/2}}\right) \Ot\left(0, \sqrt{\frac{\alpha}{\beta}}\right).
\end{equation}
Since $\Ot(0,a) = \arctan\left(a\right) / (2\pi)$ \cite{owen1980table}, the integral becomes,
\begin{equation}
\begin{split}
  \int_{-\infty}^\infty \int_{|x|}^\infty (y^2-x^2)\, e^{-\alpha x^2-\beta y^2} \dd y \, \dd x &= \frac{1}{2\alpha\beta} 
    +  \left(\frac{\alpha - \beta}{2(\alpha\beta)^{3/2}}\right) \arctan\left( \sqrt{\frac{\alpha}{\beta}}\right) \\
    &=  \frac{\sqrt{\alpha\beta} + \left(\alpha - \beta\right) \, \arctan\left( \sqrt{\alpha/\beta} \right) }{2 (\alpha\beta)^{3/2}}.
\end{split}
\end{equation}
\end{proof}
\begin{figure}[t!]
    \centering
    \includegraphics[width=0.4 \linewidth]{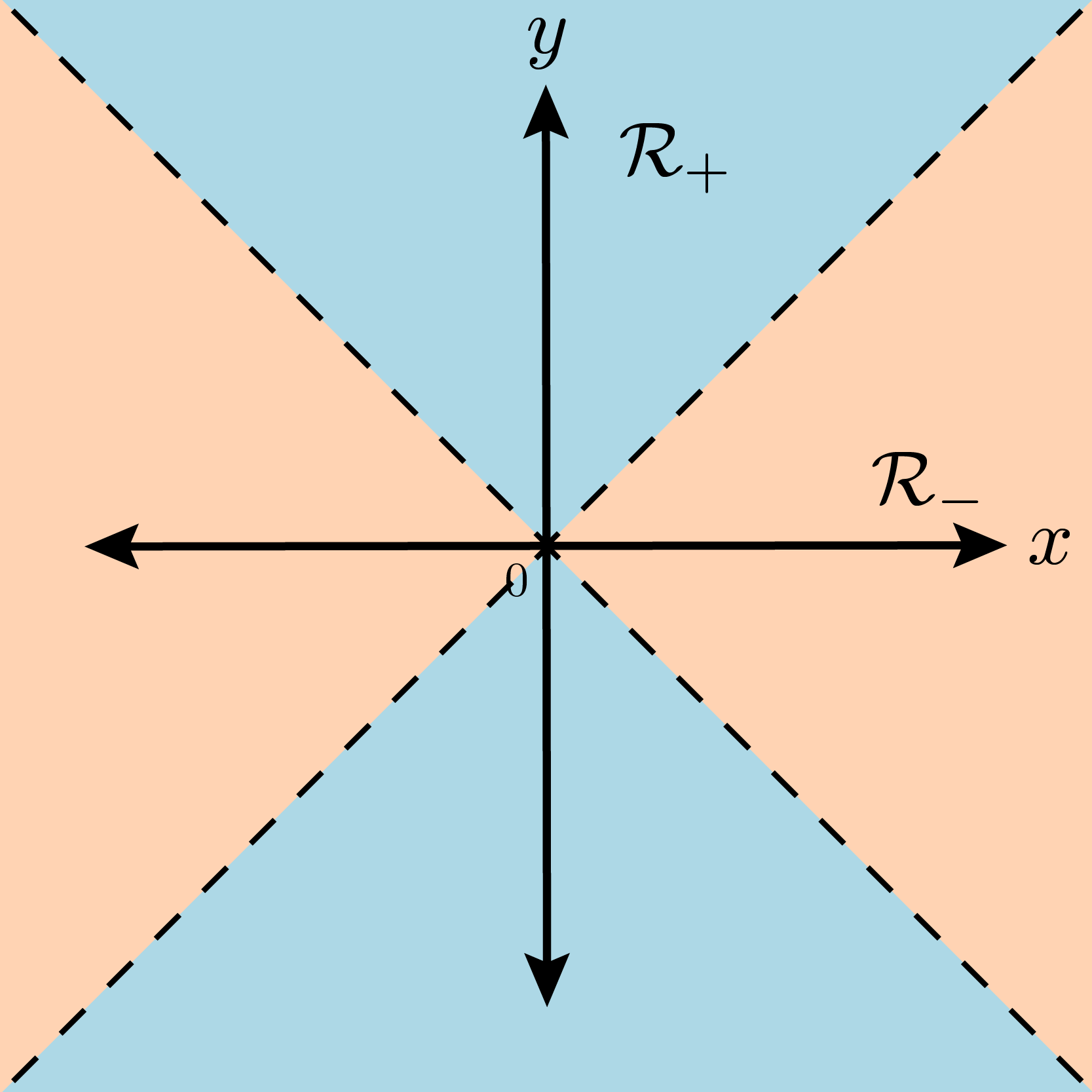}
    \caption[Splitting of $\R^2$ into two disjoint regions.]{\textbf{Splitting of $\R^2$ into two disjoint regions.} $\mathcal{R}_+$ is the region in blue satisfying $|y| \geq |x|$ and $\mathcal{R}_-$ is the region in orange satisfying $|x| > |y|$.}
    \label{fig:integration_limits}
\end{figure}
\begin{theorem} \label{thrm:integral_crossing}
Let $\alpha > 0$, $\beta > 0$, and $\gamma \in \R$. The following result holds:
\begin{equation}
\begin{split}
 \int_{-\infty}^\infty \int_{-\infty}^\infty \left|y^2-x^2\right| \, e^{-\alpha(x-\gamma)^2-\beta y^2} \dd y \, \dd x =& \frac{2e^{-\alpha \gamma^2}}{\alpha\beta} \left(1 + \sqrt{\pi} \, \gamma \sqrt{\alpha + \beta}  \, e^{\frac{\alpha^2 \gamma^2}{\alpha + \beta}}  \, \Erf\left(\frac{\alpha \gamma}{\sqrt{\alpha + \beta}}\right)\right) \\& + 4 \pi \left(\frac{\alpha - \beta - 2\alpha \beta \gamma^2}{\left(\alpha\beta\right)^{3/2}}\right) \left(\Ot\left(\gamma\sqrt{\frac{2\alpha\beta}{\alpha+\beta}}, \sqrt{\frac{\alpha}{\beta}}\right)-\frac{1}{8}\right)
\end{split} \label{eq:crossing_theorem_equation}
\end{equation}
\end{theorem}

\begin{proof}
Let us denote the integral we have to evaluate by $I$.
\begin{equation}
    I = \int \int_{\R^2} \left|y^2-x^2\right| \, e^{-\alpha(x-\gamma)^2-\beta y^2} \dd y \, \dd x
\end{equation}
The integral is evaluated over $(x, y) \in \R^2$. We split $\R^2$ into two disjoint regions (Fig.~\ref{fig:integration_limits}), defined as follows:
\begin{itemize}
    \item $\mathcal{R}_+ : \{(x, y) : |y| \geq |x| \}$, i.e., $y^2\geq x^2$, and
    \item $\mathcal{R}_- : \{(x, y) : |x| > |y| \}$, i.e., $x^2 > y^2$.
\end{itemize}

Writing the integral over these disjoint regions gives us,
\begin{equation}
\begin{split}
    I &= \int \int_{\mathcal{R}_+ } \left|y^2-x^2\right| \, e^{-\alpha(x-\gamma)^2-\beta y^2} \dd y \, \dd x + \int \int_{\mathcal{R}_-} \left|y^2-x^2\right| \, e^{-\alpha(x-\gamma)^2-\beta y^2} \dd y \, \dd x \\
    &= \underbrace{\int \int_{\mathcal{R}_+} \left(y^2-x^2\right) \, e^{-\alpha(x-\gamma)^2-\beta y^2} \dd y \, \dd x}_{\text{\normalsize \(I_1\)}} - \underbrace{\int \int_{\mathcal{R}_-} \left(y^2-x^2\right) \, e^{-\alpha(x-\gamma)^2-\beta y^2} \dd y \, \dd x}_{\text{\normalsize \(I_2\)}}
\end{split} \label{eq:splitting_over_plus_and_minus}
\end{equation}
\begin{figure}[t!]
    \centering
    \includegraphics[width=0.35 \linewidth]{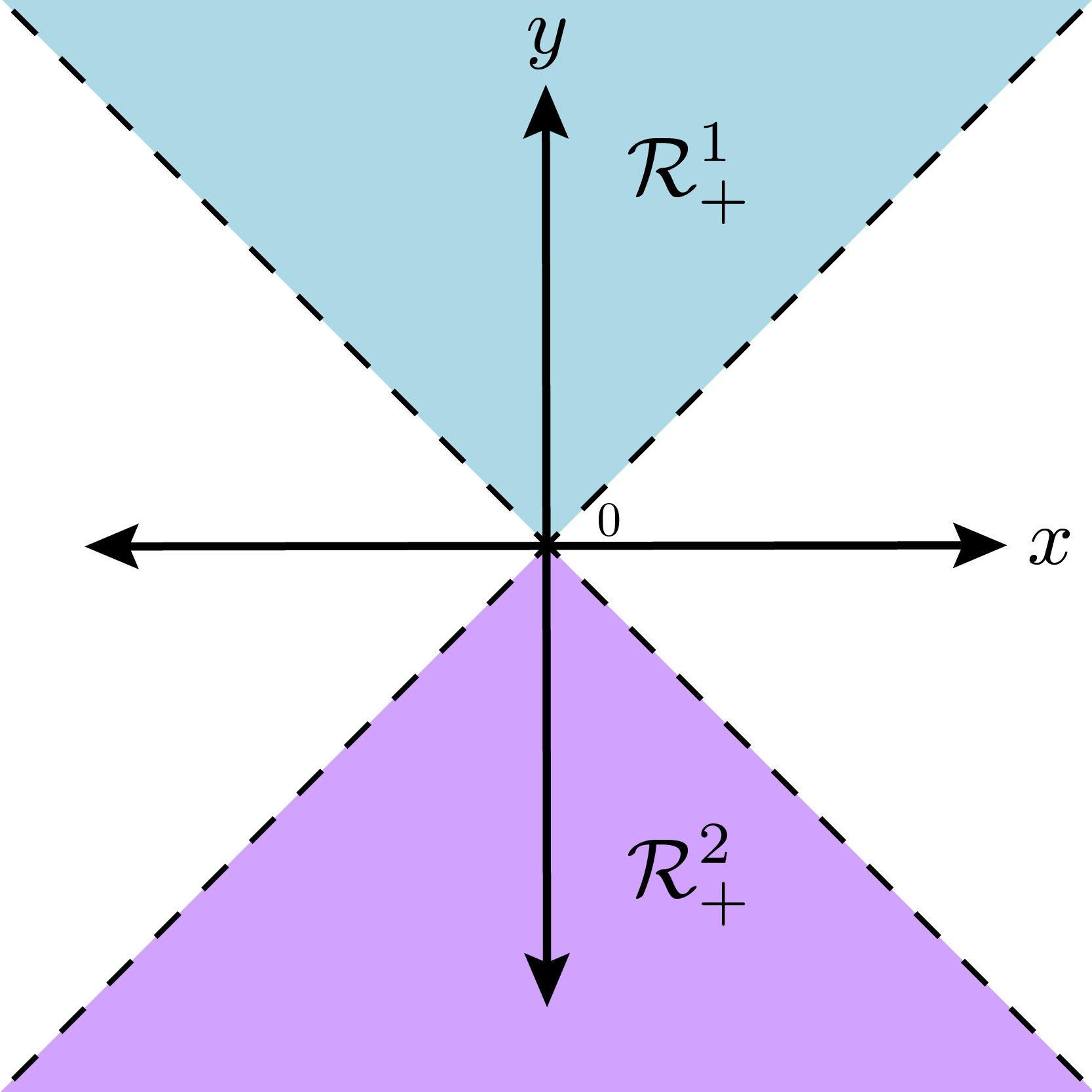}
    \caption[Splitting of $\mathcal{R}_+$ into two disjoint regions.]{\textbf{Splitting of $\mathcal{R}_+$ into two disjoint regions.} $\mathcal{R}_+^1$ is the region in blue satisfying $y \geq |x|$ and $\mathcal{R}_+^2$ is the region in purple satisfying $y \leq -|x|$.}
    \label{fig:integration_limits_updown}
\end{figure}
Therefore, $I = I_1 - I_2$. For evaluating the integral $I_1$, we split $\mathcal{R}_+$ into two disjoint subregions (Fig.~\ref{fig:integration_limits_updown}) defined as follows:
\begin{itemize}
    \item $\mathcal{R}_+^1 : \{(x, y) : y \geq |x| \}$, and
    \item $\mathcal{R}_+^2 : \{(x, y) : y \leq -|x| \}$.
\end{itemize}
Therefore, $I_1$ can be written as,
\begin{equation}
\begin{split}
    I_1 &= \int \int_{\mathcal{R}_+} \left(y^2-x^2\right) \, e^{-\alpha(x-\gamma)^2-\beta y^2} \dd y \, \dd x \\
    &= \int \int_{\mathcal{R}_+^1} \left(y^2-x^2\right) \, e^{-\alpha(x-\gamma)^2-\beta y^2} \dd y \, \dd x + \int \int_{\mathcal{R}_+^2} \left(y^2-x^2\right) \, e^{-\alpha(x-\gamma)^2-\beta y^2} \dd y \, \dd x \\ 
    &= \int_{-\infty}^\infty \int_{|x|}^\infty \left(y^2-x^2\right) \, e^{-\alpha(x-\gamma)^2-\beta y^2} \dd y \, \dd x + \int_{-\infty}^\infty \int_{-\infty}^{-|x|}  \left(y^2-x^2\right) \, e^{-\alpha(x-\gamma)^2-\beta y^2} \dd y \, \dd x
\end{split}
\end{equation}
Since the integrand in the second integral is independent of the sign of $y$, we have,
\begin{equation}
\begin{split}
    I_1 &= \int_{-\infty}^\infty \int_{|x|}^\infty \left(y^2-x^2\right) \, e^{-\alpha(x-\gamma)^2-\beta y^2} \dd y \, \dd x + \int_{-\infty}^\infty \int_{|x|}^\infty \left(y^2-x^2\right) \, e^{-\alpha(x-\gamma)^2-\beta y^2} \dd y \, \dd x \\
    &= 2 \underbrace{\int_{-\infty}^\infty \int_{|x|}^\infty \left(y^2-x^2\right) \, e^{-\alpha(x-\gamma)^2-\beta y^2} \dd y \, \dd x}_{\text{\normalsize \(I_0\)}}
\end{split}
\end{equation}
We know the result of $I_0$ from Theorem~\ref{thrm:integral_main}. Therefore, we have $I_1 = 2I_0$. Now consider the following integral, the result of which is known (Lemma~\ref{lemma:for_crossing}),
\begin{equation}
\begin{split}
    I^\prime &= \int \int_{\R^2} \left(y^2-x^2\right) \, e^{-\alpha(x-\gamma)^2-\beta y^2} \dd y \, \dd x  \\
    &= \int \int_{\mathcal{R}_+} \left(y^2-x^2\right) \, e^{-\alpha(x-\gamma)^2-\beta y^2} \dd y \, \dd x + \int \int_{\mathcal{R}_-} \left(y^2-x^2\right) \, e^{-\alpha(x-\gamma)^2-\beta y^2} \dd y \, \dd x
\end{split}
\end{equation}
We recognize these integrals from Eq.~\eqref{eq:splitting_over_plus_and_minus} and have the following relation,
\begin{equation}
    I^\prime = I_1 + I_2 \quad \text{or} \quad I_2 = I^\prime - I_1
\end{equation}
Therefore, $I$ can be expressed in terms of the known integrals as follows,
\begin{equation}
\begin{split}
    I =& I_1 - I_2 \\
    =& I_1 - \left(I^\prime - I_1\right) \\
    =& 2I_1 - I^\prime \\
    =& 4I_0 - I^\prime \\
    =& 4\left(\frac{e^{-\alpha \gamma^2}}{2\alpha\beta} \left(1 + \sqrt{\pi} \, \gamma \sqrt{\alpha + \beta}  \, e^{\frac{\alpha^2 \gamma^2}{\alpha + \beta}}  \, \Erf\left(\frac{\alpha \gamma}{\sqrt{\alpha + \beta}}\right)\right) + \pi \left(\frac{\alpha - \beta - 2\alpha\beta\gamma^2}{(\alpha\beta)^{3/2}}\right) \Ot\left(\gamma\sqrt{\frac{2\alpha\beta}{\alpha+\beta}}, \sqrt{\frac{\alpha}{\beta}}\right)\right) \\&- \frac{\pi}{2} \left(\frac{\alpha - \beta - 2\alpha \beta \gamma^2}{\left(\alpha\beta\right)^{3/2}}\right) \\
    =& \frac{2e^{-\alpha \gamma^2}}{\alpha\beta} \left(1 + \sqrt{\pi} \, \gamma \sqrt{\alpha + \beta}  \, e^{\frac{\alpha^2 \gamma^2}{\alpha + \beta}}  \, \Erf\left(\frac{\alpha \gamma}{\sqrt{\alpha + \beta}}\right)\right) + 4 \pi \left(\frac{\alpha - \beta - 2\alpha \beta \gamma^2}{\left(\alpha\beta\right)^{3/2}}\right) \left(\Ot\left(\gamma\sqrt{\frac{2\alpha\beta}{\alpha+\beta}}, \sqrt{\frac{\alpha}{\beta}}\right)-\frac{1}{8}\right)
\end{split}
\raisetag{150pt}
\end{equation}
\end{proof}

\begin{corollary}[to Theorem.~\ref{thrm:integral_crossing}]
Let $\alpha > 0$ and $\beta > 0$. The following result holds:
\begin{equation}
 \int_{-\infty}^\infty \int_{-\infty}^\infty \left|y^2-x^2\right|\, e^{-\alpha x^2-\beta y^2} \dd y \, \dd x = \frac{2}{\left(\alpha \beta\right)^{3/2}} \left(\sqrt{\alpha \beta} + \left(\alpha-\beta\right)\left(\arctan\left( \sqrt{\frac{\alpha}{\beta}}\right) -\frac{\pi}{4}\right)\right)
\end{equation}
\end{corollary}

\begin{proof}
The integral is a special case of the integral in Theorem~\ref{thrm:integral_crossing} with the parameter $\gamma=0$. Eq.~\eqref{eq:crossing_theorem_equation} states that,
\begin{equation}
\begin{split}
 \int_{-\infty}^\infty \int_{-\infty}^\infty \left|y^2-x^2\right| \, e^{-\alpha(x-\gamma)^2-\beta y^2} \dd y \, \dd x =& \frac{2e^{-\alpha \gamma^2}}{\alpha\beta} \left(1 + \sqrt{\pi} \, \gamma \sqrt{\alpha + \beta}  \, e^{\frac{\alpha^2 \gamma^2}{\alpha + \beta}}  \, \Erf\left(\frac{\alpha \gamma}{\sqrt{\alpha + \beta}}\right)\right) \\& + 4 \pi \left(\frac{\alpha - \beta - 2\alpha \beta \gamma^2}{\left(\alpha\beta\right)^{3/2}}\right) \left(\Ot\left(\gamma\sqrt{\frac{2\alpha\beta}{\alpha+\beta}}, \sqrt{\frac{\alpha}{\beta}}\right)-\frac{1}{8}\right)
\end{split} 
\end{equation}
Upon substituting $\gamma=0$ on both sides, we get,
\begin{equation}
  \int_{-\infty}^\infty \int_{-\infty}^\infty \left|y^2-x^2\right| \, e^{-\alpha x^2-\beta y^2} \dd y \, \dd x = \frac{2}{\alpha\beta} 
    + 4 \pi \left(\frac{\alpha - \beta}{(\alpha\beta)^{3/2}}\right) \left(\Ot\left(0, \sqrt{\frac{\alpha}{\beta}}\right) - \frac{1}{8}\right).
\end{equation}
Since $\Ot(0,a) = \arctan\left(a\right) / (2\pi)$ \cite{owen1980table}, the integral becomes,
\begin{equation}
\begin{split}
  \int_{-\infty}^\infty \int_{-\infty}^\infty \left|y^2-x^2\right| \, e^{-\alpha x^2-\beta y^2} \dd y \, \dd x &= \frac{2}{\alpha\beta} 
    +  2\left(\frac{\alpha - \beta}{(\alpha\beta)^{3/2}}\right) \left(\arctan\left( \sqrt{\frac{\alpha}{\beta}}\right) -\frac{\pi}{4}\right)\\
    &= \frac{2}{\left(\alpha \beta\right)^{3/2}} \left(\sqrt{\alpha \beta} + \left(\alpha-\beta\right)\left(\arctan\left( \sqrt{\frac{\alpha}{\beta}}\right) -\frac{\pi}{4}\right)\right).
\end{split}
\end{equation}
\end{proof}

\begin{lemma} \label{lemma:inner_integral}
Let $\alpha > 0$, $\beta > 0$, and $\gamma \in \R$. The following result holds:
\begin{equation}
\begin{split}
    \int_{|x|}^{\infty} \left( y^{2} - x^{2} \right) e^{ -\alpha (x - \gamma)^{2} - \beta y^{2} } \dd y = \frac{1}{2 \beta}  |x|e^{ -\alpha (x - \gamma)^{2} -\beta x^{2}}  - \frac{\sqrt{\pi}}{4 \beta^{3/2}} \left(2 \beta x^2 - 1\right) e^{ -\alpha (x - \gamma)^{2} } \Erfc\left(\sqrt{\beta}|x|\right)
\end{split}
\end{equation}
\end{lemma}

\begin{proof}
We begin by evaluating the integral
\begin{equation}
I = \int_{|x|}^{\infty} \left( y^{2} - x^{2} \right) e^{ -\alpha (x - \gamma)^{2} - \beta y^{2} } \dd y.
\end{equation}
Since \(x\) is treated as a constant with respect to the integration variable \(y\), we can factor out the constant exponential term,
\begin{equation}
I = e^{ -\alpha (x - \gamma)^{2} } \left( \int_{|x|}^{\infty} y^{2} e^{ -\beta y^{2} } \dd y - x^{2} \int_{|x|}^{\infty} e^{ -\beta y^{2} } \dd y \right).
\end{equation}
These integrals are standard and can be evaluated in terms of the complementary error functions $\Erfc(z) = 1 - \Erf(z)$,
\begin{equation}
\int_{|x|}^{\infty} e^{ -\beta y^{2} } \dd y = \frac{1}{2} \sqrt{\frac{\pi}{\beta}} \Erfc\left(\sqrt{\beta}|x|\right)
\end{equation}
and
\begin{equation}
\int_{|x|}^{\infty} y^{2} e^{ -\beta y^{2} } \dd y = \frac{|x|}{2 \beta} e^{-\beta x^{2}} + \frac{\sqrt{\pi}}{4 \beta^{3/2}} \Erfc\left(\sqrt{\beta}|x|\right).
\end{equation}
Substituting these results back into the expression for \(I\), we get,
\begin{equation}
I = \frac{1}{2 \beta}  |x|e^{ -\alpha (x - \gamma)^{2} -\beta x^{2}} - \frac{\sqrt{\pi}}{4 \beta^{3/2}} \left(2 \beta x^2 - 1\right) e^{ -\alpha (x - \gamma)^{2} } \Erfc\left(\sqrt{\beta}|x|\right)
\end{equation}
\end{proof}

\begin{lemma} \label{lemma:I1}
Let $\alpha > 0$, $\beta > 0$, and $\gamma \in \R$. The following result holds:
\begin{equation}
     \int_{-\infty}^\infty |x| e^{ -\alpha (x - \gamma)^{2} -\beta x^{2}} \dd x = \frac{e^{-\alpha  \gamma ^2}}{(\alpha +\beta )^{3/2}} \left(\sqrt{\alpha +\beta } + \sqrt{\pi } \, \alpha  \gamma \, e^{\frac{\alpha ^2 \gamma ^2}{\alpha +\beta }} \Erf \left(\frac{\alpha  \gamma }{\sqrt{\alpha +\beta }}\right)\right)
\end{equation}
\end{lemma}

\begin{proof}
We denote the integral by $I$ and split it at $0$ to get,

\begin{equation}
\begin{split}
    I &= \int_{-\infty}^0 (-x) \, e^{-\beta x^2-\alpha(x-\gamma)^2} \dd x + \int_0^\infty x \, e^{-\beta x^2-\alpha(x-\gamma)^2} \dd x  \\
    &= \underbrace{\int_{0}^\infty x \, e^{-\beta x^2-\alpha(x+\gamma)^2} \dd x}_{\text{\normalsize \(I_1\)}} + \underbrace{\int_0^\infty x \, e^{-\beta x^2-\alpha(x-\gamma)^2} \dd x }_{\text{\normalsize \(I_2\)}}.
\end{split}
\end{equation}
We complete the square in $x$ for both $I_1$ and $I_2$, which gives us,
\begin{equation}
    I_1 = e^{-\frac{\alpha \beta \gamma^2}{\alpha + \beta}} \int_{0}^\infty x \, e^{-\left(\sqrt{\alpha+\beta}\,x + \frac{\alpha \gamma}{\sqrt{\alpha+\beta}}\right)^2} \dd x \quad \text{and} \quad I_2 = e^{-\frac{\alpha \beta \gamma^2}{\alpha + \beta}} \int_{0}^\infty x \, e^{-\left(\sqrt{\alpha+\beta}\,x - \frac{\alpha \gamma}{\sqrt{\alpha+\beta}}\right)^2} \dd x
\end{equation}
These are well-known integrals \cite{owen1980table} and are given by,
\begin{equation}
    I_1 = \frac{e^{-\alpha  \gamma ^2}}{2 (\alpha +\beta )^{3/2}} \left(\sqrt{\alpha +\beta }-\sqrt{\pi } \, \alpha  \gamma \, e^{\frac{\alpha ^2 \gamma ^2}{\alpha +\beta }} \Erfc\left(\frac{\alpha  \gamma }{\sqrt{\alpha +\beta }}\right)\right)
\end{equation}
\begin{equation}
    I_2 = \frac{e^{-\alpha  \gamma ^2}}{2 (\alpha +\beta )^{3/2}} \left(\sqrt{\alpha +\beta } + \sqrt{\pi } \, \alpha  \gamma \, e^{\frac{\alpha ^2 \gamma ^2}{\alpha +\beta }} \Erfc\left(-\frac{\alpha  \gamma }{\sqrt{\alpha +\beta }}\right)\right)
\end{equation}
Substituting $\Erfc(z) = 1 - \Erf(z)$ and noting that $\Erf(-z) = -\Erf(z)$, we get the following expression for $I$,
\begin{equation}
\begin{split}
        I &= I_1 + I_2 \\
        &= \frac{e^{-\alpha  \gamma ^2}}{(\alpha +\beta )^{3/2}} \left(\sqrt{\alpha +\beta } + \sqrt{\pi } \, \alpha  \gamma \, e^{\frac{\alpha ^2 \gamma ^2}{\alpha +\beta }} \Erf \left(\frac{\alpha  \gamma }{\sqrt{\alpha +\beta }}\right)\right) 
\end{split}
\end{equation}
\end{proof}

\begin{lemma} \label{lemma:I21}
Let $\alpha > 0$, $\beta > 0$, and $\gamma \in \R$. The following result holds:
\begin{equation}
    \int_{-\infty}^\infty \left(2 x^2 - 1\right) e^{ -\frac{\alpha}{\beta} (x - \gamma\sqrt{\beta})^{2}} \dd x = \sqrt{\pi} \sqrt{\frac{\beta}{\alpha}} \left(\frac{\beta}{\alpha} + 2 \beta \gamma^2 - 1 \right)
\end{equation}
\end{lemma}

\begin{proof}
We denote the integral under consideration by $I$ and make the following variable transformation, $x = k + \gamma \sqrt{\beta}$. This gives us
\begin{equation}
\begin{split}
    I &= \int_{-\infty}^\infty \left(2 \left( k + \gamma\sqrt{\beta} \right)^2 - 1\right) e^{-\frac{\alpha}{\beta} k^2} \dd k \\
    &= \int_{-\infty}^\infty \left(2 k^2 + 4 \gamma \sqrt{\beta} \, k + 2 \beta \gamma^2 - 1\right) e^{-\frac{\alpha}{\beta} k^2} \dd k \\
    &= 2 \underbrace{\int_{-\infty}^\infty k^2 e^{-\frac{\alpha}{\beta} k^2} \dd k}_{\text{\normalsize \(I_1\)}} + 4 \gamma \sqrt{\beta} \underbrace{\int_{-\infty}^\infty k \, e^{-\frac{\alpha}{\beta} k^2} \dd k}_{\text{\normalsize \(I_2\)}} + \left(2 \beta \gamma^2 - 1 \right) \underbrace{\int_{-\infty}^\infty e^{-\frac{\alpha}{\beta} k^2} \dd k}_{\text{\normalsize \(I_3\)}} 
\end{split}
\end{equation}
The integrand of $I_2$ is an odd function; therefore, $I_2=0$. $I_1$ and $I_3$ are standard integrals given by,
\begin{equation}
    I_1 = \frac{\sqrt{\pi}}{2} \left( \frac{\beta}{\alpha} \right)^{3/2} \quad \text{and} \quad I_3 = \sqrt{\pi} \sqrt{\frac{\beta}{\alpha}}
\end{equation}
Therefore, $I$ is given by,
\begin{equation}
\begin{split}
    I &= \sqrt{\pi} \left( \frac{\beta}{\alpha} \right)^{3/2} + \left(2 \beta \gamma^2 - 1 \right) \sqrt{\pi} \sqrt{\frac{\beta}{\alpha}} \\
    &= \sqrt{\pi} \sqrt{\frac{\beta}{\alpha}} \left(\frac{\beta}{\alpha} + 2 \beta \gamma^2 - 1 \right)
\end{split}
\end{equation}
\end{proof}

\begin{lemma} \label{lemma:I2221}
Let $\alpha > 0$, $\beta > 0$, and $\gamma \in \R$. The following result holds:
\begin{equation}
\begin{split}
    \int_{0}^\infty e^{-\frac{\alpha}{\beta}  \left(x + \gamma\sqrt{\beta } \right)^2 - x^2} \biggl( \gamma\sqrt{\beta } &-\left(x + \gamma\sqrt{\beta } \right) e^{\frac{4 \alpha  \gamma }{\sqrt{\beta }} x} - x\biggr)  \dd x = \\& -\frac{\beta \, e^{-\alpha  \gamma ^2} }{(\alpha +\beta )^{3/2}} \biggl(\sqrt{\alpha +\beta } +\sqrt{\pi } \,\gamma \, (2 \alpha +\beta ) \, e^{\frac{\alpha ^2 \gamma ^2}{\alpha +\beta }} \Erf\left(\frac{\alpha  \gamma }{\sqrt{\alpha +\beta }}\right)\biggr)
\end{split}
\end{equation}
\end{lemma}

\begin{proof}
We denote the integral by $I$ and expand to get the following four integrals,
\begin{equation}
\begin{split}
    I =&  \gamma\sqrt{\beta } \int_{0}^\infty e^{-\frac{\alpha}{\beta}  \left(x + \gamma\sqrt{\beta } \right)^2 - x^2} \, \dd x - \gamma\sqrt{\beta } \int_{0}^\infty e^{-\frac{\alpha}{\beta}  \left(x + \gamma\sqrt{\beta } \right)^2 - x^2 + \frac{4 \alpha  \gamma }{\sqrt{\beta }} x}  \, \dd x \\&
    - \int_{0}^\infty x\, e^{-\frac{\alpha}{\beta}  \left(x + \gamma\sqrt{\beta } \right)^2 - x^2} \, \dd x - \int_{0}^\infty x \, e^{-\frac{\alpha}{\beta}  \left(x + \gamma\sqrt{\beta } \right)^2 - x^2 + \frac{4 \alpha  \gamma }{\sqrt{\beta }} x}  \, \dd x
\end{split}
\end{equation}
Forming a square in $x$ in the exponential terms, we get,
\begin{equation}
\begin{split}
    I =&  \gamma\sqrt{\beta } \, e^{-\frac{\alpha  \beta  \gamma ^2}{\alpha +\beta }} \int_{0}^\infty e^{-\left(\sqrt{\frac{\alpha +\beta }{\beta }} \, x + \frac{\alpha  \gamma }{\sqrt{\alpha +\beta }} \right)^2} \, \dd x - \gamma\sqrt{\beta } \, e^{-\frac{\alpha  \beta  \gamma ^2}{\alpha +\beta }} \int_{0}^\infty e^{-\left(\sqrt{\frac{\alpha +\beta }{\beta }} \, x - \frac{\alpha  \gamma }{\sqrt{\alpha +\beta }} \right)^2}  \, \dd x \\&  - e^{-\frac{\alpha  \beta  \gamma ^2}{\alpha +\beta }} \int_{0}^\infty x\, e^{-\left(\sqrt{\frac{\alpha +\beta }{\beta }} \, x + \frac{\alpha  \gamma }{\sqrt{\alpha +\beta }} \right)^2} \, \dd x   - e^{-\frac{\alpha  \beta  \gamma ^2}{\alpha +\beta }} \int_{0}^\infty x \, e^{-\left(\sqrt{\frac{\alpha +\beta }{\beta }} \, x - \frac{\alpha  \gamma }{\sqrt{\alpha +\beta }} \right)^2}  \, \dd x
\end{split} \label{eq:four_integrals}
\end{equation}
These are standard integrals of normal functions. Using the results in \cite{owen1980table}, we get,
\begin{equation}
    \int_{0}^\infty e^{-\left(\sqrt{\frac{\alpha +\beta }{\beta }} \, x + \frac{\alpha  \gamma }{\sqrt{\alpha +\beta }} \right)^2} \, \dd x = \frac{\sqrt{\pi }}{2} \sqrt{\frac{\beta }{\alpha +\beta }} \left(1 - \Erf\left(\frac{\alpha  \gamma }{\sqrt{\alpha +\beta }}\right)\right)
\end{equation}
\begin{equation}
    \int_{0}^\infty e^{-\left(\sqrt{\frac{\alpha +\beta }{\beta }} \, x - \frac{\alpha  \gamma }{\sqrt{\alpha +\beta }} \right)^2} \, \dd x = \frac{\sqrt{\pi }}{2} \sqrt{\frac{\beta }{\alpha +\beta }} \left(1 + \Erf\left(\frac{\alpha  \gamma }{\sqrt{\alpha +\beta }}\right)\right)
\end{equation}
\begin{equation}
    \int_{0}^\infty x\, e^{-\left(\sqrt{\frac{\alpha +\beta }{\beta }} \, x + \frac{\alpha  \gamma }{\sqrt{\alpha +\beta }} \right)^2} \, \dd x = \frac{\beta}{2 (\alpha +\beta )^{3/2}}  \left(\sqrt{\alpha +\beta } \, e^{-\frac{\alpha ^2 \gamma ^2}{\alpha +\beta }}-\sqrt{\pi } \, \alpha  \gamma  \,\left(1-\Erf\left(\frac{\alpha  \gamma }{\sqrt{\alpha +\beta }}\right)\right)\right)
\end{equation}
\begin{equation}
    \int_{0}^\infty x\, e^{-\left(\sqrt{\frac{\alpha +\beta }{\beta }} \, x - \frac{\alpha  \gamma }{\sqrt{\alpha +\beta }} \right)^2} \, \dd x = \frac{\beta}{2 (\alpha +\beta )^{3/2}}  \left(\sqrt{\alpha +\beta } \, e^{-\frac{\alpha ^2 \gamma ^2}{\alpha +\beta }}+\sqrt{\pi } \, \alpha  \gamma  \,\left(1+\Erf\left(\frac{\alpha  \gamma }{\sqrt{\alpha +\beta }}\right)\right)\right)
\end{equation}
Substituting these expressions in Eq.~\eqref{eq:four_integrals} and simplifying, we get,
\begin{equation}
    I = -\frac{\beta \, e^{-\alpha  \gamma ^2} }{(\alpha +\beta )^{3/2}} \biggl(\sqrt{\alpha +\beta } +\sqrt{\pi } \,\gamma \, (2 \alpha +\beta ) \, e^{\frac{\alpha ^2 \gamma ^2}{\alpha +\beta }} \Erf\left(\frac{\alpha  \gamma }{\sqrt{\alpha +\beta }}\right)\biggr)
\end{equation}

\end{proof}

\begin{lemma} \label{lemma:I2222}
Let $\alpha > 0$, $\beta > 0$, and $\gamma \in \R$. The following result holds:
\begin{equation}
\begin{split}
    \int\limits_{0}^\infty e^{-x^2}\left(\Erf\left(\sqrt{\frac{\alpha}{\beta}}\left(x+\gamma\sqrt{\beta}\right)\right)+\Erf\left(\sqrt{\frac{\alpha}{\beta}}\left(x-\gamma\sqrt{\beta}\right) \right)\right) \,\dd x = 4 \sqrt{\pi} \,\Ot\left(\sqrt{\frac{2 \alpha \beta}{\alpha + \beta}} \, \gamma, \sqrt{\frac{\alpha}{\beta}}\right)
\end{split}
\end{equation}
\end{lemma}
 
\begin{proof}
Consider the result ``$10,016.6$'' from Owen's table of normal integrals \cite{owen1980table}. It states that,
\begin{equation}
    \int_0^\infty G^\prime(x) \, G(a+bx) \, \dd x = \frac{1}{2} G\left(\frac{a}{\sqrt{1+b^2}}\right) + \Ot\left(\frac{a}{\sqrt{1+b^2}}, b\right) \label{eq:owens_result}
\end{equation}
where $G(z)$ is the cumulative distribution function of a standard normal Gaussian random variable and $G^\prime(z)$ is its derivative, i.e.,
\begin{equation}
    G(z) = \frac{1}{\sqrt{2 \pi}}\int_{-\infty}^z e^{-t^2/2} \, \dd t \quad \text{and} \quad G^\prime(z) = \frac{1}{\sqrt{2 \pi}} e^{-z^2/2},
\end{equation}
and $\Ot(h, a)$ is Owen's T function defined as,
\begin{equation}
    \Ot(h, a) = \frac{1}{2\pi} \int_0^a \frac{e^{-\frac{1}{2}h^2(1+t^2)}}{1+t^2} \, \dd t.
\end{equation}
First, we note the following relation,
\begin{equation}
    G(z) = \frac{1}{2} \left(1 + \Erf\left(\frac{z}{\sqrt{2}}\right)\right)
\end{equation}
We substitute this expression into the $\LHS$ of Eq.~\eqref{eq:owens_result},
\begin{equation}
\begin{split}
    \LHS &= \int_0^\infty G^\prime(x) \, G(a+bx) \, \dd x \\
    &= \int_0^\infty \left(\frac{1}{\sqrt{2 \pi}} e^{-x^2/2}\right) \, \left(\frac{1}{2} \left(1 + \Erf\left(\frac{a+bx}{\sqrt{2}}\right)\right)\right) \, \dd x \\
    &= \frac{1}{2\sqrt{2 \pi}}\int_0^\infty e^{-x^2/2} \, \dd x + \frac{1}{2\sqrt{2 \pi}}\int_0^\infty e^{-x^2/2} \Erf\left(\frac{a+bx}{\sqrt{2}}\right) \, \dd x \\
    &= \frac{1}{4} + \frac{1}{2\sqrt{2 \pi}}\int_0^\infty e^{-x^2/2} \Erf\left(\frac{a+bx}{\sqrt{2}}\right) \, \dd x
\end{split} 
\end{equation}
Making the variable change $k=x/\sqrt{2}$ gives us,
\begin{equation}
    \LHS = \frac{1}{4} + \frac{1}{2\sqrt{\pi}}\int_0^\infty e^{-k^2} \Erf\left(bk + \frac{a}{\sqrt{2}}\right) \, \dd k
\end{equation}
Now, the $\RHS$ of Eq.~\eqref{eq:owens_result} can be written as,
\begin{equation}
\begin{split}
    \RHS &= \frac{1}{2} G\left(\frac{a}{\sqrt{1+b^2}}\right) + \Ot\left(\frac{a}{\sqrt{1+b^2}}, b\right) \\
    &= \frac{1}{4} \left(1 + \Erf\left(\frac{a}{\sqrt{2(1+b^2)}}\right)\right) + \Ot\left(\frac{a}{\sqrt{1+b^2}}, b\right)
\end{split}
\end{equation}
Equating $\LHS = \RHS$ gives us the following relation,
\begin{equation}
    \int_0^\infty e^{-k^2} \Erf\left(bk + \frac{a}{\sqrt{2}}\right) \, \dd k = \sqrt{\pi}\left( \frac{1}{2}\Erf\left(\frac{a}{\sqrt{2(1+b^2)}}\right) + 2 \,\Ot\left(\frac{a}{\sqrt{1+b^2}}, b\right) \right)
\end{equation}
Finally, we replace $a/\sqrt{2}$ by $c$ on both sides to get the relation,
\begin{equation}
    \int_0^\infty e^{-k^2} \Erf\left(bk + c\right) \, \dd k = \sqrt{\pi}\left( \frac{1}{2}\Erf\left(\frac{c}{\sqrt{1+b^2}}\right) + 2 \,\Ot\left(\sqrt{\frac{2}{1+b^2}} \, c, b\right) \right)
\end{equation}
Replacing $c\rightarrow-c$ and using the properties that $\Erf(-z) = -\Erf(z)$ and $\Ot(-h, a) = \Ot(h, a)$,  we also have the relation,
\begin{equation}
    \int_0^\infty e^{-k^2} \Erf\left(bk - c\right) \, \dd k = \sqrt{\pi}\left( -\frac{1}{2}\Erf\left(\frac{c}{\sqrt{1+b^2}}\right) + 2 \,\Ot\left(\sqrt{\frac{2}{1+b^2}} \, c, b\right) \right)
\end{equation}
Adding these two equations gives the following result,
\begin{equation}
    \int_0^\infty e^{-k^2} \left(\Erf\left(bk + c\right) + \Erf\left(bk - c\right)\right)\, \dd k = 4 \sqrt{\pi} \,\Ot\left(\sqrt{\frac{2}{1+b^2}} \, c, b\right)
\end{equation}
Substituting $b\rightarrow \sqrt{\alpha/\beta}$ and $c\rightarrow \gamma \sqrt{\alpha}$ into the equation above gives us,
\begin{equation}
    \int_0^\infty e^{-k^2} \left(\Erf\left(\sqrt{\frac{\alpha}{\beta}} \, k + \gamma \sqrt{\alpha}\right) + \Erf\left(\sqrt{\frac{\alpha}{\beta}} \, k - \gamma \sqrt{\alpha}\right)\right)\, \dd k = 4 \sqrt{\pi} \,\Ot\left(\sqrt{\frac{2 \alpha \beta}{\alpha + \beta}} \, \gamma, \sqrt{\frac{\alpha}{\beta}}\right)
\end{equation}
\end{proof}

\begin{lemma} \label{lemma:for_crossing}
Let $\alpha > 0$, $\beta > 0$, and $\gamma \in \R$. The following result holds:
\begin{equation}
    \int_{-\infty}^\infty \int_{-\infty}^\infty \left(y^2-x^2\right) \, e^{-\alpha(x-\gamma)^2-\beta y^2} \dd y \, \dd x = \frac{\pi}{2} \left(\frac{\alpha - \beta - 2\alpha \beta \gamma^2}{\left(\alpha\beta\right)^{3/2}}\right)
\end{equation}
\end{lemma}

\begin{proof}
Let the integral we need to evaluate be denoted by $I$. First, we integrate in $y$ assuming $x$ is constant,
\begin{equation}
\begin{split}
    I &= \int_{-\infty}^\infty e^{-\alpha(x-\gamma)^2} \left(\int_{-\infty}^\infty y^2 \, e^{-\beta y^2} \dd y - x^2 \int_{-\infty}^\infty  e^{-\beta y^2} \dd y\right) \, \dd x \\
    &= \int_{-\infty}^\infty e^{-\alpha(x-\gamma)^2} \left(\frac{\sqrt{\pi}}{2\beta^{3/2}} - \sqrt{\frac{\pi}{\beta}} \, x^2\right) \, \dd x \\
    &= \frac{\sqrt{\pi}}{2\beta^{3/2}} \int_{-\infty}^\infty e^{-\alpha(x-\gamma)^2} \, \dd x - \sqrt{\frac{\pi}{\beta}}  \int_{-\infty}^\infty x^2 e^{-\alpha(x-\gamma)^2} \, \dd x
\end{split}
\end{equation}
Evaluating the integral in $x$ and simplifying the expressions, we get the final expression for $I$,
\begin{equation}
\begin{split}
    I &= \frac{\sqrt{\pi}}{2\beta^{3/2}} \int_{-\infty}^\infty e^{-\alpha(x-\gamma)^2} \, \dd x - \sqrt{\frac{\pi}{\beta}}  \int_{-\infty}^\infty x^2 e^{-\alpha(x-\gamma)^2} \, \dd x \\
    &= \frac{\sqrt{\pi}}{2\beta^{3/2}} \sqrt{\frac{\pi}{\alpha}} - \sqrt{\frac{\pi}{\beta}}  \left(\frac{\sqrt{\pi }}{2 \alpha ^{3/2}} \left(1 + 2 \alpha  \gamma ^2\right)\right) \\
    &= \frac{\pi}{2\sqrt{\alpha\beta}} \left(\frac{1}{\beta} - \frac{1}{\alpha} - 2\gamma^2\right) \\
    &= \frac{\pi}{2} \left(\frac{\alpha - \beta - 2\alpha \beta \gamma^2}{\left(\alpha\beta\right)^{3/2}}\right)
\end{split}
\end{equation}

\end{proof}

\end{document}